\documentclass[11pt,reqno]{amsart}
\usepackage[english]{babel}
\usepackage{amsfonts}
\usepackage[utf8]{inputenc}
\usepackage{textgreek}
\usepackage{amsthm}
\usepackage{mathrsfs}
\usepackage{amsmath}
\usepackage{mathrsfs}
\usepackage{enumerate}
\usepackage{cite}
\usepackage{bbm}
\usepackage{xcolor}
\usepackage{graphicx}
\usepackage{frontespizio}
\usepackage[makeroom]{cancel}
\usepackage[normalem]{ulem}
\usepackage{amssymb}
\usepackage[right=2.5cm,bottom=2.5cm,footskip=30pt,left=2.5cm]{geometry}
\usepackage[colorlinks, citecolor=citegreen, linkcolor=refred,urlcolor=blue, unicode,psdextra,hypertexnames=false]{hyperref}
\usepackage{mathtools}
\usepackage{enumitem}
\usepackage[toc,page]{appendix}
\usepackage{esint}
\usepackage{todo}
\usepackage{aligned-overset}
\definecolor{citegreen}{rgb}{0,0.3,0}
\definecolor{refred}{rgb}{0.5,0,0}
\makeatletter
\def\author@andify{
	\nxandlist {\unskip{} $\cdot$ \penalty-2}
	{\unskip {} $\cdot$ \penalty-2}
	{\unskip {} $\cdot$ \penalty-2}}
\makeatother
\newcommand{\orcid}[1]{\unskip {} \raisebox{-.3ex}{\href{https://orcid.org/#1}{\includegraphics{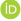}}}}

\title[Second-order estimates for the $p$-Laplacian in RCD spaces]{Second-order estimates for the $p$-Laplacian\\in  RCD spaces}
\let\oldemail\email
\def\email#1{\oldemail{\href{mailto:#1}{\textcolor{black}{#1}}}}
\author[L.~Benatti]{Luca Benatti\orcid{0000-0002-4685-7443}}
\address{L.~Benatti, Universit\`a di Pisa, Largo Bruno Pontecorvo 5, 56127 Pisa, Italy}
\email{luca.benatti@dm.unipi.it}
\author[I. Y.~Violo]{Ivan Yuri Violo\orcid{0000-0001-9592-5150}}
\address{I. Y.~Violo,  Centro di Ricerca Matematica Ennio De Giorgi, Scuola Normale Superiore, Piazza dei Cavalieri 3, 56126 Pisa (PI), Italy}
\email{ivan.violo@sns.it}
\newcommand{\eps}{\varepsilon}

\newcommand{\rr}{\mathbb{R}}
\newcommand{\nn}{\mathbb{N}}

\newcommand{\nchi}{{\raise.3ex\hbox{$\chi$}}}

\newcommand{\sfd}{{\sf d}}
\newcommand{\Lip}{{\rm Lip}}

\renewcommand{\phi}{\varphi}
\newcommand{\restr}[1]{\lower3pt\hbox{$|_{#1}$}}

\newcommand{\X}{{\rm X}}

\newcommand{\fr}{\penalty-20\null\hfill$\blacksquare$} 
\definecolor{mygray}{gray}{0.9}
\newcommand{\diam}{{\rm diam}}

\newcommand{\la}{\langle}
\newcommand{\ra}{\rangle}
\renewcommand{\div}{{\rm div}}
\newcommand{\mea}{\mathfrak{m}}
\newcommand{\mm}{\mathfrak{m}}

\newcommand{\LIP}{\mathsf{LIP}}

\newcommand{\test}{{\sf{Test}}}

\newcommand{\lims}{\varlimsup}

\renewcommand{\d}{{\mathrm d}}
\newcommand{\loc}{\mathsf{loc}}
\newcommand{\W}{\mathit{W}^{1,2}}

\renewcommand{\H}[1]{{\rm Hess}(#1)}

\newcommand{\supp}{\mathop{\rm supp}\nolimits}

\newcommand{\lip}{{\rm lip}}

\DeclareMathOperator*{\esssup}{ess\,sup}
\DeclareMathOperator*{\essinf}{ess\,inf}

\newcommand{\Xdm}{(\X,\sfd,\mm)}
\newcommand{\RCD}{\mathrm{RCD}}

\newcommand{\R}{\mathbb{R}}

\newcommand{\tr}{{\rm tr}}

\newcommand{\rcd}{\mathrm{RCD}}

\newcommand{\dom}{{\sf D}}

\renewcommand{\limsup}{\varlimsup}
\renewcommand{\liminf}{\varliminf}

\newcommand{\LL}{\mathcal{L}}
\newcommand{\alme}{\mea\text{-a.e.}}

\input{brace.tex}
\theoremstyle{plain}
\newtheorem{theorem}{Theorem}[section]
\newtheorem{lemma}[theorem]{Lemma}
\newtheorem{prop}[theorem]{Proposition}

\newtheorem{cor}[theorem]{Corollary}

\theoremstyle{definition}
\newtheorem{definition}[theorem]{Definition}
\newtheorem{remark}[theorem]{Remark}

\numberwithin{equation}{section}
\begin{document}
	\begin{abstract}
		We establish quantitative second-order Sobolev regularity for functions having a $2$-integrable $p$-Laplacian in bounded metric spaces satisfying the Riemannian Curvature Dimension condition, with $p$ in a suitable range. In the finite-dimensional case, we also obtain Lipschitz regularity under the assumption that the $p$-Laplacian is sufficiently integrable. Our results cover both $p$-Laplacian eigenfunctions and $p$-harmonic functions having relatively compact level sets.
		
	\end{abstract}
	\maketitle
	\medskip
	
	\noindent \underline{\smash{Keywords}}: nonlinear potential theory, metric measure spaces, RCD spaces, regularity estimates, degenerate elliptic PDEs.

	\section{Introduction and main results}
	After providing the existence of a weak solution to an elliptic PDE, it is natural to wonder whether it fits its strong formulation. This issue boils down to investigating the regularity of the solution, more specifically if it belongs to a Sobolev class with a higher order of derivation. The most classical example dates back to the solution of the Poisson equation. This result can be summarized as follows: given any open set $\Omega \subset \rr^n$ one has
	\begin{equation}\label{eq:2 reg intro}
		u \in W^{1,2}_\loc(\Omega),\,\,  \Delta u=f  \in L^2_\loc(\Omega) \implies  u \in W^{2,2}_\loc(\Omega).
	\end{equation}
	An analogous result can also be found in the setting of metric measure spaces. More precisely, to deal with second-order Sobolev spaces, the class of $\RCD(K,n)$ spaces provides a natural, albeit quite broad, framework within which to focus our analysis. With this locution, one addresses that subclass of metric measure spaces $\Xdm$ satisfying a synthetic notion of Ricci curvature bounded below by $K\in \rr$ (see \cite{AmbICM,G23} for the relevant background and for historical notes on the topic). In \cite{Gigli14}, building upon \cite{Savare13}, it is showed that given $\Xdm$ an $\RCD(K,\infty)$ space
	\begin{equation}\label{eq:2reg rcd}
		u \in \W(\X),\, \Delta u =f \in L^2(\mm) \implies u \in W^{2,2}(\X)
	\end{equation}
	holds, together with the quantitative estimate 
	\begin{equation}\label{eq:CZintro}
		\int |\H u|_{HS}^2 \d \mm \le \|\Delta u\|_{L^2(\mm)}^2+K^-\||\nabla u|\|_{L^2(\mm)}^2,
	\end{equation}
	where $\H u$ represents the Hessian and $|\,\cdot\,|_{HS}$ the Hilbert-Schmidt norm. In this setting, \eqref{eq:2reg rcd} plays a fundamental role in the RCD theory which goes beyond the regularity statement itself. It shows that there are many functions in $W^{2,2}(\X)$, which is the cornerstone to develop a rich second-order calculus which has been exploited in many recent papers (see e.g.\  \cite{gigli_monotonicityformulasharmonicfunctions_2023,H19,MS22,KKL23,DGP21,BS18,GiMar23,caputo2021parallel}).

	We aim to investigate the regularity of solutions to the $p$-Poisson equation
	\begin{equation}\label{eq:p pois intro}
		\Delta_p u=f,
	\end{equation}
	where $u \in W^{1,p}(\X)$, $p\in (1,\infty)$ and  $f$ is a sufficiently integrable function. Given $u \in W^{1,p}(\X)$ we say that $u$ has a $p$-Laplacian, writing $u\in \dom(\Delta_p)$, if there exists (a unique) $f \in L^1_\loc(\mm)$ such that
	\begin{equation}\label{eq:weak sol intro}
		\int_{\X}\la |\nabla u|^{p-2}\nabla u,\nabla \phi\ra \,\d \mm=-\int_{\X} \phi f\, \d \mm, \quad \,\,\, \forall \phi \in \LIP_{bs}(\X).
	\end{equation}
	In this case, we set $\Delta_p u \coloneqq f$ (see Definition \ref{def:plapl} for details). The PDE \eqref{eq:p pois intro} is a natural nonlinear generalization of the Poisson equation, which corresponds to the case $p=2$. In the smooth setting, a statement akin to \eqref{eq:2 reg intro} is already at our disposal in literature. In \cite{CiaMaz18}, the authors showed that for any open set $\Omega\subset \rr^n$ and for all $p\in(1,\infty)$ it holds
	\begin{equation}\label{eq:p reg intro}
		u \in W^{1,p}_\loc(\Omega),\,\,  \Delta_p u=f  \in L^2_\loc(\Omega) \implies |\nabla u|^{p-2}\nabla u \in \W_\loc(\Omega).
	\end{equation}
	Observe that the vector field $|\nabla u|^{p-2}\nabla u$ appearing here is precisely the one in the definition of the $p$-Laplacian $\Delta_pu=\div(|\nabla u|^{p-2}\nabla u)$. This addresses the question of whether the solution can be regarded in a strong sense. In light of this, our aim is to prove the counterpart of \eqref{eq:p reg intro} in the nonsmooth setting, much like \eqref{eq:2reg rcd} is for \eqref{eq:2 reg intro}. For technical reasons, we have to restrict our analysis to those $p$'s belonging to a certain \textit{regularity interval}, which is given in the following definition.
	\begin{definition}[Regularity interval]\label{def:regularity interval}
		Let $\Xdm$ be a bounded $\rcd(K,N)$, with  $N\in[2,\infty]$ and set $\delta_\X\coloneqq \frac{\lambda_1(\X)K^-}{1+\lambda_1(\X)K^-} \in [0,1)$ ($\lambda_1(\X)$ being the first non-zero Laplacian eigenvalue, see \eqref{eq:spectral gap}). We define the open interval $\mathcal{RI}_{\X}\subset  (1,3+\frac{2}{N-2})$ by
		\begin{equation}\label{eq:regularity interval}
			\mathcal{RI}_{\X}\coloneqq \begin{cases}
				(1,\infty), & \text{if $N=2$ and $\delta_\X=0$,}\\
				\left(2-\sqrt{1-\delta_\X},\, 2+ \sqrt{1-\delta_\X}\right)& \text{if $N=\infty$,}\\
				\left(2-\sqrt{1-\delta_\X},\, 2+ \sqrt{1-\delta_\X}\frac{N-\delta_\X}{N-2+\delta_\X}\right)& \text{otherwise.}
			\end{cases}
		\end{equation}
	\end{definition}
	As one can expect, $\mathcal{RI}_{\X}$ is always a non-trivial interval containing $p=2$ and it is invariant under scaling of both the distance and measure. The generalizations of \eqref{eq:2reg rcd} and \eqref{eq:CZintro} to the case $p\neq 2 $ read now as follows. 
	
	\begin{theorem}[Regularity of $p$-Laplacian, $N=\infty$]\label{thm:main rcd inf}
		Let $\Xdm$ be an $\rcd(K,\infty)$ space with $\diam(\X)
			\le D<\infty
		$. Fix $p \in \mathcal{RI}_{\X}$ and suppose that $u \in \dom(\Delta_p)$ with  $\Delta_p u \in L^2(\mm)$. Then $|\nabla u|^{p-2}\nabla u\in H^{1,2}_{C}(T\X)$ and in particular $|\nabla u|^{p-1}\in \W(\X)$. Moreover 
		\begin{equation}\label{eq:p-calderon intro}
			\int |\nabla (|\nabla u|^{p-2}\nabla u)|^2\d \mm\le C ( \|\Delta_pu\|_{L^2(\mm)}^2+K^{-}\||\nabla u|^{p-1}\|_{L^1(\mm)}^2 ), 
		\end{equation}
		where $C>0$ is a constant depending only on $p,K$ and $D.$
	\end{theorem}
	Note that the right-hand side of \eqref{eq:p-calderon intro} is finite, indeed $\||\nabla u|^{p-1}\|_{L^1(\mm)}<+\infty$ since $u \in W^{1,p}(\X)$. 
	Nevertheless,  except for the case $p=2$, we do not conclude that $u \in W^{2,2}(\X)$. This would be false even in the smooth setting (see Remark \ref{rmk:counterexampleW22}).

	\smallskip
	
	A well-known consequence of second-order estimates is Lipschitz regularity results for solutions to PDEs. A path classically followed is to apply a De Giorgi-Nash-Moser iteration scheme to the modulus of the gradient $|\nabla u|$. Recall that this iteration method can be used in high generality to obtain $L^\infty$-bounds for subsolutions to elliptic partial differential equations and that ultimately relies on the Sobolev inequality. Consider as an example a harmonic function in a Riemannian manifold with non-negative Ricci curvature. In this framework, we know the validity of a (local) Sobolev inequality while the Bochner identity says precisely that $|\nabla u|^2$ is subharmonic. 
	
	In the setting of $\RCD(K,N)$ spaces with $N < \infty$, both a Sobolev inequality (see \textsection\ref{sec:sobolev}) and the following weak Bochner inequality \begin{equation}\label{eq:bochner}
		\frac12 \Delta |\nabla u|^2 \ge \frac{(\Delta u)^2}{N}+\la \nabla u, \nabla \Delta u\ra + K|\nabla u|^2,
	\end{equation}
	for $u$ sufficiently regular (see e.g.\ \cite{Hua-Kell-Xia13}), are available. On Alexandrov spaces the local Lipschitzianity of harmonic functions was established in \cite{Pet03,ZZ12}, while in the larger class of $\RCD(K,N)$ spaces, with $N<\infty$, this follows from  \cite{Jiang} and building upon \cite{Koskela-Rajala-Shanmugalingam03}. In the recent \cite{G23bis,MS22}  Lipschitz regularity was established for harmonic maps from $\rcd$ spaces to ${\sf CAT}(0)$ spaces, generalizing the previous results in Alexandrov spaces  \cite {ZZ18} in turn extending the ones in Riemannian manifolds \cite{ES64} (see the introduction of \cite{MS22} for more references on this topic). The local Lipschitzianity result for the Poisson equation $\Delta u =f \in L^q$ in RCD setting was instead established in \cite{Jiang} for $q=\infty$ and in \cite{Kell13} for  $q>N$.
	
	A key ingredient in the De Giorgi-Nash-Moser method is that the target function, the modulus of the gradient in our case, is of class $W^{1,2}_\loc$. Hence, thanks to our second-order Sobolev regularity result, we can apply this scheme for $p\in \mathcal{RI}_{\X}$.
	\begin{theorem}[Regularity of $p$-Laplacian, $N<\infty$]\label{thm:main intro}
		Let $\Xdm$ be an $\rcd(K,N)$ with  $N\in[2,\infty)$, $\diam(\X)<\infty$ and fix $p \in \mathcal{RI}_{\X}$ where $ \mathcal{RI}_{\X}\subset(1,\infty)$ is given by \eqref{eq:regularity interval}. Let $u \in \dom(\Delta_p)$ (with $\Delta_p u \in L^1(\mm))$ and fix $\Omega\subset \X$ open. 
		All the following hold.
		\begin{enumerate}[label=\roman*)]
			\item\label{item:main intro item1} If $\Delta_p u \in L^2(\Omega;\mea)$  then $|\nabla u|^{p-2}\nabla u\in H^{1,2}_{C,\loc}(T\X;\Omega)$ and in particular $|\nabla u|^{p-1}  \in W^{1,2}_\loc(\Omega)$.
			\item\label{item:main intro item2} If $\Delta_p u \in L^q(\Omega;\mea)$ for some $q>N$, then $u\in \LIP_{\loc}(\Omega).$
		\end{enumerate}
	\end{theorem}
	
	Both the regularity results in \ref{item:main intro item1} and \ref{item:main intro item2} come with quantitative estimates (see Theorem \ref{thm:main detailed}). Notice that in Theorem \ref{thm:main intro}, even if the conclusions are local, we had to assume that $u$ has a $p$-Laplacian in the whole space $\X$. At the moment, excluding the case $p=2$, we are not able to achieve a purely local regularity result as stated in \eqref{eq:p reg intro} in the flat Euclidean setting. In the case $p=2$, one can multiply $u$ by a suitable cut-off supported in $\Omega$ and extend it to be $0$ outside $\Omega$. The new function coincides with $u$ on a given compact subset of $\Omega$ and its Laplacian belongs to $L^1(\mm)$ (see e.g.\ \cite{AMS16}). Differently from the Poisson equation, this argument collides with the nonlinearity of the operator when $p\neq 2$.

	\medskip
	As a consequence of Theorem \ref{thm:main intro} we can obtain the following regularity result for eigenfunctions of the $p$-Laplacian.
	\begin{cor}[Regularity of $p$-eigenfunctions]\label{cor:p eigen}
		Let $\Xdm$ be a bounded $\rcd(K,N)$, with  $N\in[2,\infty)$ and fix $p \in \mathcal{RI}_{\X}$. Suppose $u\in W^{1,p}(\X)$ satisfies $\Delta_p u=-\lambda u |u|^{p-2}$ for some $\lambda \ge0.$ Then $|\nabla u|^{p-2}\nabla u\in H^{1,2}_{C}(\X)$, $|\nabla u|^{p-1}  \in W^{1,2}(\X)$
		and  $u\in \LIP(\X).$
	\end{cor}
	For $p=2,$ the Lipschitzianity of eigenfunctions follows from \cite{Jiang} (see also \cite[Proposition 7.1]{AHT17} for a direct proof using heat kernel estimates). Gradient estimates for $p$-eigenfunctions in weighted Riemannian manifolds satisfying the $\RCD(K,N)$ condition were also obtained in \cite{DD16}.

	Our last result concerns the Lipschitzianity and second-order regularity of a class of $p$-harmonic functions in $\RCD$ spaces.

	\begin{cor}[Regularity of $p$-harmonic function with relatively compact level sets]\label{cor:pharm}
		Let $\Xdm$ be a bounded $\rcd(K,N)$ space, with  $N\in[2,\infty)$ and fix $p \in \mathcal{RI}_{\X}$. Suppose $u$ is $p$-harmonic in some open set $\Omega\subset \X$ and that $U\coloneqq u^{-1}(a,b)\subset\subset \Omega$ for some $a,b\in \rr$. Then $u \in H^{2,2}_\loc(U)$, $|\nabla u|^{p-1}\in W^{1,2}_\loc(U)$ and $u\in \LIP_\loc(U)$.
	\end{cor}
	As explained earlier, it is unclear how to obtain from Theorem \ref{thm:main intro} regularity estimates for functions $u$ defined only in some open subset of the space due to the non-linearity of the operator. Corollary \ref{cor:pharm} is obtained by exploiting the homogeneity of $\Delta_p$. The key observation here is that a function $u$ with relatively compact level sets can be extended to the whole space by post-composition with a cut-off function.
 
    The (local) H\"older regularity of $p$-harmonic functions for $p\neq 2$ is well known in the far more general setting of \textit{doubling} spaces supporting a local \textit{Poincaré inequality} (see  \cite[Chapter 8]{BB13} or the Appendix at the end of this note), via standard iteration methods. Corollary  \ref{cor:pharm} on the other hand shows  local Lipschitzianity of some class  $p$-harmonic functions for $p\neq 2$ in a suitable range. This last result, the best of the authors' knowledge,   is  new  even for Alexandrov spaces.

	\subsection*{Overview of the strategy}
	The usual method to obtain regularity results for the $p$-Laplacian is to consider a sequence of more regular approximating operators. To illustrate this, consider the case of $p$-harmonic functions with Dirichlet boundary conditions in an open set $\Omega\subset \rr^n$
	\[
	\Delta_p(u)=0, \quad u \in W^{1,p}_0(\Omega).
	\]
	If we consider the regularized problems
	\[
	\Delta_{p,\eps}(u_\eps)\coloneqq  \div((|\nabla u_\eps|^2+\eps)^{\frac{p-2}2}\nabla u_\eps)=0, \quad u_\eps \in W^{1,p}_0(\Omega),
	\]
	then it holds $u_\eps \to u$ in $W^{1,p}(\Omega)$. Moreover, by classical elliptic regularity (see \cite[Chapter 4]{UralLadybook}) we have $u_\eps\in C^\infty(\Omega)$. Therefore, to obtain the regularity for $u$ it is sufficient to derive appropriate estimates for $u_\eps$ independent of $\eps$ (see e.g.\ \cite{Lewis83,Evans82,Lou08,CiaMaz18}). However, deriving uniform a \textit{priori} estimates is not the main obstacle in our metric setting, rather, it is the regularity of the solutions $u_\eps$.
	
	The convergence of $u_\eps$ to $u$ as $\eps \to 0$ can be proved also in this framework (see \textsection \ref{sec:approximationresults}). On the other hand, in $\rcd$ spaces the elliptic regularity theory stops at  H\"older estimates à la De Giorgi-Nash-Moser (see Appendix \ref{sec:moser}),  with the only exception of the Laplacian (recall \eqref{eq:2reg rcd}). The main obstacle is the lack of a difference quotients method.  To overcome this issue we exploit the fact that for $p$ close to $2$ the operator $\Delta_{p,\eps}$ is in a sense close to the Laplacian, indeed
	\[
	(|\nabla u|^2+\eps)^{\frac{p-2}2}\Delta_{p,\eps}(u)=\Delta u+(p-2)\frac{\H u(\nabla u, \nabla u)}{|\nabla u|^2+\eps}.
	\]
	This was observed in \cite{ManWeit88} where it was used to show $W^{2,2}$-regularity of harmonic functions for $1<p<3+\frac{2}{n-2}$, by exploiting the so called \emph{Cordes conditions} for elliptic equations in non-divergence form (see \cite{TalentiSopra,cordesbook,Cordes56,Campanato67}).  In our case, we will use this idea to obtain the regularity of solutions to $\Delta_{p,\eps}(u)=f$ (see \textsection \ref{sec:existence_regularised}), from the regularity property of the Laplacian. This will be done in two steps. First for any $w\in \W(\X)$ we obtain a solution $U_w\in \W(\X)$ 
	\[
	\Delta U_w+(p-2)\frac{\H w(\nabla w, \nabla w)}{|\nabla w|^2+\eps}=\frac{f}{(|\nabla w|^2+\eps)^{\frac{p-2}2}}\in L^2(\mm),\quad  \Delta U_w \in L^2(\mm),
	\]
	by a fixed point method inspired by \cite[$\mathsection$ 1.2]{cordesbook} and \cite{TalentiSopra} (see \textsection \ref{sec:existence pt1}). Then we show that there exists a further fixed point for the map $w\mapsto U_w$ and thus obtain a function $u\in \W(\X)$ such that $\Delta_{p,\eps}(u)=f$ and $\Delta u\in L^2(\mm)$ (see \textsection\ref{sec:existence pt2}). From it, the regularity of solutions to $\Delta_p(u)=f$ will be obtained via uniform estimates in $\eps$. These estimates are derived by taking suitable test functions in the Bochner inequality (see \textsection \ref{sec:eps regularity} and \textsection \ref{sec: main results}).

	\section{Preliminaries}
	\subsection{First order calculus on metric measure spaces}\label{sec:first order}
	We will denote by the triple $\Xdm$ a \emph{metric measure space}, i.e.\  a complete and separable metric space $(\X,\sfd)$ endowed with a boundedly finite Borel measure $\mea\ge 0$ such that $\supp(\mm)=\X.$
	For any open set $\Omega\subset \X$ we will denote by \(\LIP_{\mathrm{loc}}(\Omega)\) and by $\LIP_{bs}(\Omega)$ respectively the space of all locally Lipschitz functions in $\Omega$ and Lipschitz functions with support bounded and contained in $\Omega$. By $L^{p}_\loc(\Omega)$, $p \in [1,\infty)$, we denote the space of $\mea$-a.e.\ equivalence classes of Borel functions $f:\Omega \to \rr$ such that $f\restr B\in L^p(\Omega;\mea)$ for every $B\subset \Omega$ closed and bounded.
	The \emph{slope} \(\lip(f)\colon\X\to[0,+\infty)\) of a function \(f\in\LIP_{\mathrm{loc}}(\X)\) is defined as \(\lip(f)(x)\coloneqq 0\) if \(x\in\X\) is an isolated point and
	\[
	\lip(f)(x)\coloneqq\lims_{y\to x}\frac{|f(x)-f(y)|}{\sfd(x,y)}\quad\text{ if }x\in\X\text{ is an accumulation point.}
	\]
	
	We assume the reader to be familiar with the definition and properties of Sobolev spaces in metric measure spaces and normed modules, referring to \cite{BB13} and \cite{GP20} for detailed accounts on these topics. We will denote by $W^{1,p}(\X)$, with $p\in(1,\infty)$, the $p$-Sobolev space on a metric measure space $\Xdm$ and by $|D f|_{p,w}\in L^p(\mm)$ the minimal $p$-weak upper gradient of a function $f \in W^{1,p}(\X).$  In this note we will always work on  a metric measure space \ $\Xdm$ satisfying:
	\begin{itemize}
		\item \emph{independent weak upper gradient},
		\item $\W(\X)$ is a Hilbert space, i.e.\ $\Xdm$ is \emph{infinitesimally Hilbertian} (see \cite{Gigli12}).
	\end{itemize} Examples of spaces satisfying the above are $\RCD(K,\infty)$ spaces and infinitesimally Hilbertian PI-spaces, which we will introduce in the sequel (see also \cite{GN22,GH14,Cheeger99}). 
	By independent weak upper gradient, we mean that
	\begin{enumerate}[label=\alph*)]
		\item\label{item:firststrongindep} $W^{1,p}\cap W^{1,q}(\X)$ is dense in $W^{1,p}(\X)$ for all $p,q \in (1,\infty),$ 
		\item\label{item:secondindep}if $f \in W^{1,p}\cap W^{1,q}(\X)$, then $|D f|_{p,w}=|D f|_{q,w}$ $\mea$-a.e.,
	\end{enumerate}
	we refer to \cite{GN22} for further details.  In particular condition $b)$ allows us to drop the subscript for the minimal w.u.g.\ and simply write $|D f|_w$.   As observed e.g.\ in \cite{GN22}, it is possible to give a universal notion of gradient operator for functions in $W^{1,p}(\X)$ for arbitrary $p$, whenever $\X$ is infinitesimally Hilbertian and the $p$-independence weak upper gradient assumptions are fulfilled.  We recall here a simplified construction which is enough for our purposes.

	Since for every $p\in(1,\infty)$ the restriction 
	\[
	\nabla : W^{1,p}\cap W^{1,2}(\X)\to L^p(T\X)\subset L^0(T\X)
	\]
	is bounded in $L^p(T\X)$ and since $W^{1,p}\cap W^{1,2}(\X)$ is dense in $W^{1,p}(\X)$ (see $a)$ above) we can uniquely extended $\nabla $ to a linear operator on  the whole $W^{1,p}(\X),$ denoted with the same symbol. It is also easy to check that $\nabla$ retains the usual calculus rules  and satisfies
	\[
	|\nabla f|=|D f|_w, \quad \forall f \in W^{1,p}(\X).
	\]
	Thanks to the above we will always write $|\nabla f|$ and actually drop the notion $|D f|_w.$
	For any $\Omega \subset \X$ open we denote by  $W^{1,p}_\loc(\Omega)$ the space of functions $f\in L^{p}_\loc(\Omega)$ such that $f\eta \in W^{1,p}(\X)$ for every $\eta \in \LIP_{bs}(\Omega).$
	Thanks to the locality property, we can define a notion of gradient also for functions in $W^{1,p}_\loc(\Omega)$, 
	\[
	\nabla : W^{1,p}_\loc(\Omega)\to  L^0(T\X)\restr\Omega
	\]
	($L^0(T\X)\restr\Omega$ denotes the module-localization, see e.g.\ \cite[Definition 3.1.11]{GP20}) satisfying $|\nabla f|\in L^{p}_\loc(\Omega)$ and 
		the usual calculus rules.
	\begin{remark}\label{rmk:lip dense}
		By the density in energy of Lipschitz functions (see \cite{AGS13}), if $\X$ is infinitesimally Hilbertian and has independent weak upper gradient, then $\LIP_{bs}(\X)$ is dense in $W^{1,p}(\X)$ for every $p\in(1,\infty)$ (see e.g.\ the argument in \cite[Proposition 5.9]{GN22}). In particular, by cut-off, given any open set $\Omega\subset \X$, we have also that $\LIP_{bs}(\Omega)$ is dense in the subset of functions of $ W^{1,p}\cap L^\infty(\X)$ having support bounded and contained in $\Omega$.\fr 
	\end{remark}

	To conclude, we recall the Sobolev and Poincaré inequalities on PI spaces, whose definition is given in the following.
	\begin{definition}[PI space]\label{def:PI}
		A m.m.s.\ \((\X,\sfd,\mm)\) is said to be a \emph{PI space} if:
		\begin{itemize}
			\item it is \emph{uniformly locally doubling}, i.e.\ there exists a function   $C_D:(0,\infty)\to (0,\infty)$ such that
			\[
			\mm\big(B_{2r}(x)\big)\leq C_D(R)\,\mm\big(B_r(x)\big)\quad\text{ for every }0<r<R\text{ and }x\in\X,
			\]
			\item it supports a \emph{weak local \((1,1)\)-Poincar\'{e} inequality}, i.e.\   there exist a constant $\lambda\ge 1$ and a function  $ C_P:(0,\infty)\to (0,\infty)$ such that
			for every $f \in \LIP_\loc(\X)$ it holds
			\[
			\fint_{B_r(x)}\bigg|f-\fint_{B_r(x)}f\,\d\mm\bigg|\,\d\mm\leq C_P({R})\,r\fint_{B_{\lambda r}(x)}\lip(f)\,\d\mm
			\quad\text{ for every }0<r<R\text{ and }x\in\X.
			\]
		\end{itemize}
	\end{definition}

	It is  well known that PI spaces support  Sobolev-type inequalities reported below (see \cite[Theorem 5.1]{HK00}, \cite[Theorem 4.21]{BB13} and also \cite[Theorem 5.1]{BB18}.  For the statement, we recall that $(\X,\sfd)$ is said to be $L$-quasiconvex for some constant $L>0$ if for every $x,y \in \X$ there exists a curve from $x$ to $y$ of length less than or equal to $ L\sfd(x,y)$. Recall also that in any PI-space inequality \eqref{eq:pi dimension} always holds for some $s>1$  (see Remark \ref{rmk:app}).
	
	\begin{theorem}[Sobolev and Poincar\'e inequalities]\label{thm:improved poincaret}
		Let $\Xdm$ be a PI-space satisfying
		\begin{equation}\label{eq:pi dimension}
			\frac{\mea(B_r(y))}{\mea(B_{R}(y))}\ge c(R_0)\left(\frac rR\right)^s, \quad \forall\, 0<r<R\le R_0, \, \forall y \in \X,
		\end{equation}
		for some constant $s>1$ and some function $c:(0,\infty)\to (0,\infty)$.  Then for every $p\in(1,s]$ and every $R_0>0$ there exists a constant $C$ depending only on $p$, $\lambda$, $R_0$ and on the functions $c$,  $C_D$, $C_P$   such that for every $B_R(x)\subset\X$ with $R\le R_0$ the following hold.
		\begin{enumerate}[label=\roman*)]
			\item\hfill$
			\displaystyle\left(\fint_{B_R(x)} |f-f_{B_R(x)}|^{p^*} \, \d \mm \right)^\frac{1}{p^*} \le CR \left(\fint_{B_{2\lambda R}(x)} |\nabla f|^{p}\, \d \mm \right)^\frac{1}{p}, \qquad \forall \, f \in W^{1,p}(\X),$\hfill\refstepcounter{equation}\label{eq:improved poincaret}{\normalfont(\theequation)}\\
			with $f_{B_R(x)}\coloneqq\fint_{B_R(x)}f\, \d \mm$ and where $p^*\coloneqq \frac {ps}{s-p}$ (if $p<s$) and $p^*<\infty$ is arbitrary for $p=s$ (in which case $C$ depends also on the choice of $p^*$).
			\item \hfill$\displaystyle
			\left(\fint_{B_R(x)} |f|^{p^*} \d \mm \right)^\frac p{p^*}\le C \left(	\fint_{B_{2\lambda R}(x)} R^p|\nabla f|^p +|f|^p\, \d \mm\right), \qquad \forall f \in W^{1,p}(\X).$ \hfill \refstepcounter{equation}\label{eq:local sobolev}{\normalfont(\theequation)}
		\end{enumerate}
		Moreover if $(\X,\sfd)$ is $L$-quasiconvex then the constant $2\lambda$ in both \eqref{eq:improved poincaret} and \eqref{eq:local sobolev} can be replaced by $L.$
	\end{theorem}

	If $\Xdm$ satisfies only a Poincar\'e inequality, but it is not locally doubling,  Theorem \ref{thm:improved poincaret} does not apply. Nevertheless, a $(p,p)$-type Poincar\'e inequality is available as shown in the next result, which we could not find explicitly stated in the literature.
	For the statement, it will be relevant to recall the notion of spectral gap for an arbitrary metric measure space $\Xdm$ with $\mea(\X)<+\infty$, that is
	\begin{equation}\label{eq:spectral gap}
		\lambda_1(\X)\coloneqq \inf \left\{\int |D u|_{2,w}^2 \d\mm \ : \ u \in \W(\X),\, \int u \,\d \mm=0,\, \int u^2\d \mm=1  \right\}.
	\end{equation}
	We also recall that for every $A\subset \X$ with $\mea(A)<\infty$ and every function $f \in L^1(\mm;A)$, there exists a \emph{median} for $f$ in $A$, i.e.\ a number $m\in \rr$ satisfying $\mea(\{f< m\}\cap A)\le \frac{\mea(A)}2$ and $\mea(\{f> m\}\cap A)\le \frac{\mea(A)}2.$ It is well known that $m$  realizes the minimum in $\inf_{c \in \rr} \int_A|f-c|\d \mm. $
	
	\begin{theorem}[$(p,p)$-Poincar\'e inequality]\label{thm:pPoincare}
		Let $\Xdm$ be a m.m.s.\  supporting a weak local $(1,1)$-Poincar\'e inequality and  with $\diam(\X)\le D<+\infty$. Then for every $p\in(1,\infty)$ there exists a constant $C>0$ such that
		\begin{equation}\label{eq:pPoincare}
			\int |f-f_{\X}|^p\d \mm\le C \int |D f|_{p,w}^p\d \mm, \quad \forall f \in W^{1,p}(\X),
		\end{equation}
		where $f_{\X}\coloneqq \fint_{{\X}} f \d \mm$. In particular, it holds $\lambda_1(\X)>0$.
	\end{theorem}
	\begin{proof}
		It is enough to prove  \eqref{eq:pPoincare} for  $f \in \LIP(\X)$ and with $\lip(f)$ in place of $|\nabla f|$ and then argue by relaxation. Moreover, since $\int |f-f_{\X} |^p\d \mm\le 2^p \int |f-c |^p\d \mm$ for every $c \in \rr$ and $f \in L^1(\mm)$ (see e.g.\ \cite[Lemma 4.17]{BB13}), it is enough to show
		\begin{equation}\label{eq:poincare zero median}
			\int |f|^p\d \mm\le \tilde C \int \lip(f)^p\d \mm
		\end{equation}
		for some constant $\tilde C>0$ depending only on $R_0,\,K$ and $p$ and every $f \in \LIP(\X)$ for which zero is a median for $f$ in ${\X}$. Fix one of such $f \in \LIP(\X)$ and define $f^+\coloneqq f\wedge 0$ and $f^-\coloneqq -(f\vee 0).$ Note that zero is a median in ${\X}$ also for $(f^+)^p$ and $(f^-)^p.$ By the $(1,1)$-Poincar\'e inequality we have that 
		\begin{equation}\label{eq:11p}
			\int_{\X} |g-m|\d \mm\le \int_{\X} |g-g_\X|\d \mm\le C_1 \int_{\X} \lip(g)\d \mm, \quad \forall g \in \LIP(\X),
		\end{equation}
		where $C_1$ is a constant depending only on $K$, $D$ and $m$ is any median of $g$, having used that $\int_{\X} |g-m|\d \mm\le \inf_{c \in \rr} \int_{\X} |g-c|\d \mm.$ 
		Therefore by \eqref{eq:11p} and since $\lip((f^+)^p)\le p |f|^{p-1}\lip(f)$, $\lip((f^-)^p)\le p |f|^{p-1}\lip(f)$ we have
		\[
		\int_{\X} |f|^p \d \mm= \int_{\X} |f^+|^p+|f^-|^p \d \mm\le 2pC_1 \int_{\X} \lip(f)|f|^{p-1}\d \mm\le 2pC_1 \left(\int_{\X} \lip(f)^p\right)^\frac1p\left(\int_{\X} |f|^p\right)^{1-\frac1p},
		\]
		which proves \eqref{eq:poincare zero median} and finishes the proof.
	\end{proof}

	\subsection{\texorpdfstring{$p$}{p}-Poisson equation in metric measure spaces}
	
	We can now introduce the main subject of this paper, expanding on what we have already mentioned in the introduction.
	\begin{definition}[$p$-Laplacian]\label{def:plapl}
		Fix $p \in (1,\infty)$. Let $\Xdm$ be inf.\ Hilbertian with independent weak upper-gradient and $\Omega\subset \X$ be open. A function $u \in W^{1,p}_{\loc}(\Omega)$ belongs to $\dom(\Delta_p,\Omega)$ if and only if there exists (unique) $\Delta_p u\in L^1_{\loc}(\Omega)$ such that
		\begin{equation}\label{eq:def plapl}
			\int_{\Omega}\la |\nabla u|^{p-2}\nabla u,\nabla \phi\ra \d \mm=-\int_\Omega \phi \Delta_{p}u  \d \mm, \quad \forall \phi \in \LIP_{bs}(\Omega).
		\end{equation}
		For simplicity  when $\Omega=\X$ we will simply write $\dom(\Delta_{p})$.
	\end{definition}
	The assumption of independent weak upper-gradient in the previous definition is needed to consider the gradient of functions in $W^{1,p}$ (see \textsection \ref{sec:first order}). In particular, for $p=2$ this assumption is not needed.
	Recall also that in our convention $\nabla$ takes value in  $L^0(T\X)$, hence the scalar product in \eqref{eq:def  plapl} is well defined as the scalar product in $L^0(T\X).$ The above definition makes sense since $|\la |\nabla u|^{p-2}\nabla u,\nabla \phi\ra|\le |\nabla u|^{p-1}|\nabla \phi| \in L^{1}(\mm)$.

	\begin{remark}\label{rmk:test w1p plapl}
		By the density of Lipschitz functions in $W^{1,p}$ (recall Remark \ref{rmk:lip dense}), we have that \eqref{eq:def plapl} also holds for every $\phi \in W^{1,p}\cap L^\infty(\X)$ having support bounded and contained in $\Omega$. Moreover, in the case $\Omega=\X$ and $\Delta_{p} u \in L^{p'}(\mm)$, $p'\coloneqq \frac p{p-1}$, the validity of \eqref{eq:def plapl} extends also in duality with all $\phi \in W^{1,p}(\X)$. \fr
	\end{remark}

	We also introduce the set of functions with $L^2$-Laplacian in $\Omega$ as:
	\[
	\dom(\Delta,  \Omega)\coloneqq \{u \in \dom(\Delta_2,\Omega)\ : \ \Delta_2u\in L^2_\loc(\Omega)\}
	\]
	and we will write $\Delta u$ in place of $\Delta_2u$  wherever $u \in \dom(\Delta,\Omega)$. Finally, we set $\dom(\Delta)\coloneqq \dom(\Delta,  \X)\cap \W(\X)$ (note that by definition we have only $\dom(\Delta,  \X)\subset \W_\loc(\X)$). In particular, the notation $\dom(\Delta)$ is consistent with the usual one used in literature (see e.g.\ \cite[Section 5.2.1]{GP20}). If $\diam(\X)<+\infty$ it follows immediately from \eqref{eq:def plapl} that
	\begin{equation}\label{eq:zero mean laplacian}
		\int \Delta u\d \mm =0, \quad \forall  u \in \dom(\Delta).
	\end{equation}
	Whenever $\mm(\X)<\infty$, the Laplacian can also be used to characterize the spectral gap $\lambda_1(\X)$  (defined in \eqref{eq:spectral gap}) as follows (see e.g.\ \cite[Proposition 4.8.3]{BakryGentilLedoux14}) as the maximimal constant such that 
	\begin{equation}\label{eq:Laplacian poincare}
		\lambda_1(\X)\int |D u|_{2,w}^2 \d\mm\le  \int (\Delta u)^2\d \mm, \quad \forall u \in \dom(\Delta).
	\end{equation}
	The Laplacian $\Delta$ is a closed operator in the following sense:
	\begin{equation}\label{eq:closure laplacian}
		\parbox{13cm}{given $u_n \in \dom(\Delta,  \Omega)$ and $u \in W^{1,2}_\loc(\Omega)$, if $|\nabla u-\nabla u_n|\to 0$ in $L^2(\Omega)$ and $\Delta u_n \rightharpoonup G$ in $L^2(\Omega)$, then $u \in \dom(\Delta,  \Omega)$ with $\Delta u=G.$}
	\end{equation}
	The following existence result is a standard consequence of the direct method of the calculus of variations 
		and the Poincaré inequality \eqref{eq:pPoincare}.
	\begin{prop}\label{prop:p-existence}
		Fix $p \in (1,\infty)$ and set $p'\coloneqq\frac{p}{p-1}.$ Let $\Xdm$ be a bounded m.m.s.\ with independent weak upper gradient  and satisfying a weak local $(1,1)$-Poincar\'e inequality. Then, for every $f \in L^{p'}(\mm)$ such that $\int f \d \mm=0$ there exists a unique $u \in \dom(\Delta_p)$ satisfying $\int u \d \mm=0$ and $\Delta_p u=f.$
	\end{prop}

	For future reference, we report the following classical monotonicity inequalities for the $p$-Laplacian, which proof is the same as in the Euclidean setting (see e.g.\ \cite[Lemma A.0.5]{peral}).
\begin{lemma}[Monotonicity inequalities]\label{lem:monotonicity}
	For every $p \in(1,\infty)$ there exists a constant $c_p>0$ such that the following is true.
	Let $\Xdm$ be an inf.\ Hilbertian metric measure space and $v,w \in L^0(T\X)$. Then
	\begin{equation}\label{eq:monotonicity}
		\la |v|^{p-2}v-|w|^{p-2}w ,v-w \ra\ge \begin{cases}
			c_p |v-w|^p & \text{ if $p\ge 2$,}\\
			c_p \frac{ |v-w|^2}{ (|v|+|w|)^{2-p}}& \text{ if $1<p<2$,} 
		\end{cases} \quad \mea\text{-a.e.,}
	\end{equation}
	with the convention that, if $p<2$,  $|v|^{p-2}v=0$ (resp. $|w|^{p-2}w=0$) whenever $|v|=0$ (resp. $|w|=0$) and that right-hand side is zero when $|w|=|v|=0.$
\end{lemma}

\begin{remark}\label{rmk:existence}
	In our definition of $\Delta_p u$ we assume that $u \in W^{1,p}(\X).$ In Proposition \ref{prop:p-existence}, we showed existence and uniqueness in this space of solutions to $\Delta_p (u)=f$ when $f \in L^{p'}$, $p'\coloneqq \frac{p}{p-1}$. However, when $f$ is in merely in $L^2(\mm)$ (or even $L^1(\mm)$), as in our main results, if $p$ is not large enough solutions might not exist. This is a well-known issue for $p$-Laplacian type operators that already occurs in the Euclidean framework. As a matter of fact, in \cite{CiaMaz18} the authors prove that second-order Sobolev regularity estimates for the $p$-Laplacian in $\rr^n$ hold for suitable generalized solutions defined via approximation and not necessarily in $W^{1,p}_\loc$ (see \cite[Remark 2.8]{CiaMaz18}).  In fact, also our main result could be extended to more general notions of solutions to \eqref{eq:p pois intro}, see Remark \ref{rmk:sola} for a further discussion.\fr
\end{remark}

\subsection{Second-order calculus on RCD spaces}	\label{sec:rcd}

In this note, we will mainly work on $\rcd(K,N)$ spaces for $N\in [1,\infty]$, which is a well-known subclass of metric measure spaces. We will not recall their definition and main properties and refer to \cite{AmbICM,G23} for surveys on this topic and further references. 

We will assume the reader to be familiar with the theory of second-order calculus on $\rcd$ spaces developed in \cite{Gigli14} (see also \cite{GP20}). 

We will denote by $\dim(\X)$
the \emph{essential-dimension}  of an  $\rcd(K,N)$ space with $N<\infty$ (see \cite{Gigli14,GP16,BS18}), which is an integer not greater than $N$.
In the case $\dim(\X)=N$ by \cite[Prop.\ 4.1]{Han14} it holds
\begin{equation}\label{eq:tr=lapl}
	\tr \H f=\Delta f, \quad \forall f \in \dom(\Delta). 
\end{equation}

We recall the space of test functions introduced in \cite{Savare13}

\[
\test(\X)\coloneqq\{f \in \dom(\Delta)\cap L^\infty(\mea)\cap\LIP(\X) \ | \ \Delta f \in \W(\X)\}.
\]
By \cite[Theorem 3.3.8]{Gigli14} we have $\test(\X)\subset W^{2,2}(\X)$ and we can define the space $H^{2,2}(\X)$ as the closure of $\test(\X)$ in $W^{2,2}(\X).$

\begin{prop} \label{prop:gradgrad}
	Let $\Xdm$ be an $\rcd(K,\infty)$ space. For every $v \in H^{1,2}_C(T\X)$ it holds that $|v|\in \W(\X)$ and 
	\begin{equation}\label{eq:gradient covariant}
		|\nabla |v||\le |\nabla v|, \quad \alme.
	\end{equation}
	It holds that $\dom(\Delta)\subset H^{2,2}(\X)$. In particular  for every $u \in \dom(\Delta)$ it holds that $\nabla u \in H^{1,2}_C(\X)$, $|\nabla u|\in \W(\X)$ and 
	\begin{equation}\label{eq:musical}
		\begin{split}
			&\nabla \nabla u=\H u,\\
			&|\nabla u|\nabla |\nabla u|=\H u\left(\nabla u\right), 
		\end{split}
	\end{equation}
\end{prop}
\begin{proof}
	The first part of the statement together with \eqref{eq:gradient covariant} is proved in \cite[Lemma 3.5]{DGP21}. The fact that $\dom(\Delta)\subset H^{2,2}(\X)$ together with first in \eqref{eq:musical} is instead a consequence of Theorem 3.4.2-(iv)  and Proposition 3.3.18 in \cite{Gigli14}. From this by definition of $H^{1,2}_C(T\X)$ follows that $\nabla u \in H^{1,2}_C(\X)$ and the fact that $|\nabla u|\in \W(\X)$ follows then from the first part. Finally, the second in \eqref{eq:musical} is contained in  \cite[Proposition 3.3.22]{Gigli14}. 
\end{proof}

In the sequel, we will need the following approximation result which, to the best of our knowledge, is not present in the literature.
\begin{lemma}\label{lem:test dense}
	Let $\Xdm$ be an $\rcd(K,\infty)$ space. Then for all $u \in \dom(\Delta)$ there exists a sequence $(u_n)\subset \test(\X)$ such that $\Delta u_n\to \Delta u$ in $L^2(\mm)$,  $u_n \to u $ in $H^{2,2}(\X)$  and $|\nabla u_n|\to |\nabla u|$ in $\W(\X)$.  
\end{lemma}

For the proof, we need the following technical convergence result.
\begin{lemma}\label{lem:rly}
	Let $\Xdm$ be an $\rcd(K,\infty)$ space and suppose that $u_n \to u$ in $H^{2,2}(\X)$. Then $|\nabla u_n|\to |\nabla u|$ in $\W(\X).$
\end{lemma}
\begin{proof}
	Fix an arbitrary subsequence, still denoted by $u_n$. To conclude it is enough to show that there exists a further subsequence such that $|\nabla u_n|\to |\nabla u|$ in $\W(\X).$  Up to passing to a (non-relabelled) subsequence we can assume that $|\nabla u_n-\nabla u|\to 0$ $\mea$-a.e.\ and $|\H {u-u_n}|\to 0$ $\mea$-a.e..
	We need to prove that 
	\[
	\int |\nabla (|\nabla u_n|-|\nabla u|)|^2\d \mm\to 0, \quad \text{as $n\to +\infty$}.
	\]
	The goal is to apply the dominated convergence theorem. First observe that  $|\nabla (|\nabla u_n|-|\nabla u|)|^2\le 2|\H u|^2+2|\H {u_n}|^2$ and that since $|\H u_n|$ are convergent in $L^2(\mm)$, up to a subsequence, it holds that  $|\H {u_n}|\le g$ for some $g \in L^2(\mm)$. Hence to conclude it suffices to show
	\begin{equation}\label{eq:gradgradtozero}
		|\nabla (|\nabla u_n|-|\nabla u|)|\to 0 \quad \alme.
	\end{equation} 
	To see this we recall that by the second in \eqref{eq:musical} we have
	\[
	|\nabla (|\nabla u_n|-|\nabla u|)|=\left|\H{u_n}\left(\frac{\nchi_{A_n}\nabla u_n}{|\nabla u_n|}\right) - \H{u}\left(\frac{\nchi_{A_0}\nabla u}{|\nabla u|}\right)\right|, \quad \alme,
	\]
	where $A_n\coloneqq \{|\nabla u_n|>0\}$ and ${A_0}=\{|\nabla u|>0\}$.
	We then proceed to estimate the above expression as follows
	\begin{equation}
		\begin{split}
			&\left|\H{u_n}\left(\frac{\nchi_{A_n}\nabla u_n}{|\nabla u_n|}\right) - \H{u}\left(\frac{\nchi_{A_0}\nabla u}{|\nabla u|}\right)\right|
			\\&\qquad=\left|\H{u_n}\left(\frac{\nchi_{A_n}\nabla u_n}{|\nabla u_n|}\pm\frac{\nchi_{A_0}\nabla u}{|\nabla u|}\right) - \H{u}\left(\frac{\nchi_{A_0}\nabla u}{|\nabla u|}\right) \right|\\
			&\qquad\le \big|\H{u_n}|\left(\frac{|\nabla u-\nabla u_n|}{|\nabla u|}\nchi_{A_n\cap A_0} + \nchi_{\X\setminus A_0} \right)+ |\H{u-u_n}|,\qquad \alme,
		\end{split}
	\end{equation}
	where in the last step we used that  $|\H {u_n}|=0$ $\mea$-a.e.\ in $\X\setminus A_n$ by the locality of the Hessian. Moreover, again by the locality of the Hessian 
		(see \cite[Proposition 3.4.9]{Gigli14}) 
		we have
		\begin{equation}\label{eq:hessian collapsing}
			|\H{u_n}|=|\H{u_n}-\H{u}|\to 0,\quad  \alme \text{ in $\X\setminus A_0$}.
		\end{equation}
		From this and recalling that $|\nabla u_n-\nabla u|\to 0$ $\mea$-a.e.\ we obtain \eqref{eq:gradgradtozero}, which concludes the proof of the lemma.
	\end{proof}
	
	\begin{proof}[Proof of Lemma \ref{lem:test dense}]
		The existence of a sequence $(u_n)\subset \test(\X)$ such that $\Delta u_n\to \Delta u$ and $u_n \to u $ in $\W(\X)$ is proved in \cite[Lemma 2.2]{H19}. This implies that $u_n \to u $ in $H^{2,2}(\X)$ by \cite[Proposition 3.3.9]{Gigli14}. The fact that $|\nabla u_n|\to |\nabla u|$ in $\W(\X)$ then follows from Lemma \ref{lem:rly}
	\end{proof}

	We introduce the improved Bochner inequality proved in \cite{Han14} (see \cite{Gigli14} for the case $N=\infty$).
	\begin{theorem}[Improved Bochner-inequality]\label{thm:improved bochner}
		Let $\Xdm$ be any ${\rm RCD}(K,N)$  space with $K\in \rr$ and $N \in [1,\infty]$. Then for any $f \in \test(\X)$ it holds that $|\nabla f|^2\in \W(\X)$ and
		\begin{equation}\label{eq:improvedboch}
			-\int \left\langle \nabla \phi ,\frac{\nabla |\nabla f|^2}{2} \right\rangle\d\mea\ge \begin{multlined}[t][.6\textwidth]\int \left(|\H f|^2+\frac{(\Delta f-{\rm tr}\H f)^2}{N-\dim(\X)} \right.\\\left.\vphantom{\displaystyle\frac{(\Delta f-{\rm tr}\H f)^2}{N-\dim(\X)} }+\la \nabla f, \nabla \Delta f\ra +K|\nabla f|^2 \right)\phi \d \mea, 
			\end{multlined}
		\end{equation}
		for any $ \phi \in \W\cap L^\infty_+(\X)$ with bounded support, where the term containing $\dim(\X)$ is taken to be 0 if $\dim(\X)=N$ or $N=\infty$. 
	\end{theorem}
	
	From the previous result, we immediately obtain the following inequality, which will play a key role in the note (see also \cite[Corollary 3.3.9]{Gigli14} for the case $N=\infty$).
	\begin{cor}
		Let $\Xdm$ be a bounded ${\rm RCD}(K,N)$  space with $K\in \rr$ and $N \in [1,\infty]$. Then 
		\begin{equation}\label{eq:talenti}
			\int \left(|\H f|^2+\frac{(\Delta f-{\rm tr}\H f)^2}{N-\dim(\X)}\right)\d \mm \le (1+K^{-}\lambda_1(\X))\int (\Delta u)^2\d \mm, \quad \forall u \in \dom(\Delta),
		\end{equation}
		where the term containing $\dim(\X)$ is taken to be 0 if $\dim(\X)=N$ or $N=\infty$. 
	\end{cor}
	\begin{proof}
		Combining \eqref{eq:improvedboch} with $\phi=1$ and \eqref{eq:Laplacian poincare}, we obtain that \eqref{eq:talenti} holds for all $f \in \test(\X).$ The validity for a general $f \in \dom(\Delta)$ then follows by the approximation using Lemma \ref{lem:test dense} and in the case $N<\infty$ observing that by definition  $|\tr \H g|^2\le \dim(\X)|\H g|^2$ for every $g \in W^{2,2}(\X).$ 
	\end{proof}
	
	We now report a version of the Leibniz rule for the covariant derivative.
	\begin{lemma}\label{lem:H12sobolev}
		Let $\Xdm$ be an $\rcd(K,\infty)$ space. Then for every $v \in H^{1,2}_C(T\X)$ and $f \in L^\infty\cap W^{1,2}(\X)$ it holds that $fv\in H^{1,2}_C(T\X)$ and
		\begin{equation}
			\nabla(fv)=f \nabla v + \nabla f \otimes v.
		\end{equation}
	\end{lemma}
	\begin{proof}
		The case $f \in \LIP_{bs}(\X)$ is proved in \cite[Lemma 2.17]{AntPasPoz21}. The general case follows by approximation in $W^{1,2}(\X)$ by a sequence of functions in $\LIP_{bs}(\X)$ and uniformly bounded in $L^\infty(\mm)$.
	\end{proof}
	
	We conclude by introducing the space of vector fields with local covariant derivative in $L^2$.
	\begin{definition}[Local covariant derivative]\label{def:local covariant}
		Let $\Xdm$ be an $\rcd(K,\infty)$ space and $\Omega\subset \X$ be open. We define the space 
		\begin{equation}\label{eq:def loc covariant}
			H^{1,2}_{C,\loc}(T\X;\Omega)\coloneqq \{ v \in L^0(T\X)\restr{\Omega} \ : \ \eta v\in H^{1,2}_C(T\X),\, \forall \, \eta \in \LIP_{bs}(\Omega)\}.
		\end{equation}
		Moreover, for every $v \in H^{1,2}_{C,\loc}(T\X;\Omega)$ we can define $|\nabla v|\in L^2_\loc(\Omega)$ as follows: for every open set $\Omega'\subset \Omega$ such that $\overline \Omega'\subset \Omega$ we define
		\[
		\nchi_{\Omega'}|\nabla v|\coloneqq \nchi_{\Omega'}	|\eta \nabla v|,\quad \alme,
		\]
		where $\eta\in \LIP_{bs}(\Omega)$ and $\eta=1$ in $\Omega'$. This definition is well-posed thanks to the locality property of the covariant derivative (see \cite[Proposition 3.4.9]{Gigli14}).
	\end{definition}
	It also holds that
	\begin{equation}\label{eq:grad grad local}
		v\in H^{1,2}_{C,\loc}(T\X;\Omega)\implies |v|\in W^{1,2}_\loc(\Omega).
	\end{equation}
	Indeed for every $\eta \in \LIP_{bs}(\Omega)$ we have $\eta|v|=|\eta^+v|-|\eta^-v|\in W^{1,2}(\X)$, since $\eta^+,\eta^-\in \LIP_{bs}(\X)$ and by the first part of Proposition \ref{prop:gradgrad}.
	
	The above definition is motivated by the following result.
	\begin{lemma}\label{lem:lsc local energy}
		Let $\Xdm$ be an $\rcd(K,\infty)$ space and $\Omega\subset \X$ be open. Suppose that the sequence $v_n \in H^{1,2}_C(T\X)$ satisfies 
		\begin{equation}\label{eq:finite covariant energy}
			\sup_n \int_{\Omega} |v_n|^2+|\nabla v_n|^2 \d \mm <+\infty,\quad |v_n-v|\to0\quad \alme, 
		\end{equation}
		for some $v \in L^0(T\X)\restr\Omega.$ Then $v \in H^{1,2}_{C,\loc}(T\X;\Omega)$ and for every $B_{2R}(x)\subset \Omega$ it holds
		\begin{equation}\label{eq:liminf convariant energy}
			\int_{B_R(x)} |v|^2+|\nabla v|^2 \d \mm \le 4\liminf_n \int_{B_{2R}(x)} (1+R^{-1})^{2}|v_n|^2+|\nabla v_n|^2 \d \mm.
		\end{equation}
	\end{lemma}
	\begin{proof}
		Fix $\eta \in \LIP_{bs}(\Omega)$. Set $w_n\coloneqq \eta v_n \in L^2(T\X)$. By \eqref{eq:finite covariant energy} we have that the sequence $(w_n)$ is bounded in $L^2(T\X)$. Hence, up to a subsequence, it converges weakly in $L^2(T\X)$ to some $w \in L^2(T\X).$ Applying Mazur's lemma, for every $n \in \nn$ we can find $N_n\in \nn$, $N_n\ge n$ and numbers $\{a_{k,n}\}_{k=n}^{N_n}\subset[0,1]$ satisfying $\sum_{k=1}^{N_n} a_{k,n}=1$ and such that $W_n\coloneqq \sum_{k=n}^{N_n} a_{k,n} w_n\to w$ in $L^2(T\X).$ In particular $|W_n-w|\to 0$ $\mea$-a.e.. However, by the second in \eqref{eq:finite covariant energy} it also holds that $|W_n-\eta v|\to 0$ $\mea$-a.e., which shows that $w=\eta v.$ Next, by Lemma \ref{lem:H12sobolev} we have $w_n=\eta v_n\in H^{1,2}_C(T\X)$  for every $n \in \nn$ with $\nabla w_n =\nabla \eta \otimes v_n+\eta \nabla v_n$. Therefore also $W_n \in  H^{1,2}_C(T\X)$  for every $n \in \nn$ and
		\begin{align*}
			\sqrt{\Vert|W_n|\Vert_{L^2(\mm)}^2+\||\nabla W_n|\|_{L^2(\mm)}^2}&\le \sum_{k=n}^{N_n} a_{k,n} \||w_n|\|_{L^2(T\X)}+  \sum_{k=n}^{N_n} a_{k,n} \||\nabla w_n|\|_{L^2(\mm)}\\
			&\le (\|\eta\|_\infty +\Lip(\eta)) \||v_n|\|_{L^2(\supp(\eta))} +   \|\eta\|_\infty \||\nabla v_n|\|_{L^2(\supp(\eta))}.
		\end{align*}
		This and the lower semicontinuity of $\||\nabla W_n|\|_{L^2(\mm)}$ under convergence in $L^2(T\X)$ (see \cite[Theorem 3.4.2]{Gigli14}) prove that $\eta v \in H^{1,2}_C(T\X)$ and by the arbitrariness of $\eta$ also that  $v \in H^{1,2}_{C,\loc}(T\X;\Omega)$. Finally \eqref{eq:liminf convariant energy} follows by taking $\eta$ such that $\eta=1$ in $B_R(x)$ with $\supp(\eta)\subset B_{2R}(x)$ and $\Lip(\eta)\le 2R^{-1}$.
	\end{proof}
	
	\begin{remark}\label{rmk:local to global}
		If $(\X,\sfd)$ is also proper and $v\in L^0(T\X)\restr\Omega$, for some $\Omega\subset \X$ open, satisfies $v \in H^{1,2}_{C,\loc}(T\X;B_r(x))$ for every $B_{2r}(x)\subset \Omega$, then $v \in  H^{1,2}_{C,\loc}(T\X;\Omega)$. The proof is a standard argument using partitions of unity (see e.g.\ the proof of \cite[Proposition\ 3.17]{Gigli-Mondino12}). \fr
	\end{remark}
	
	\subsection{Sobolev-Poincaré inequalities in RCD spaces}\label{sec:sobolev}
	
	Recall that any $\RCD(K,N)$ spaces satisfies a weak local   $(1,1)$-Poincar\'e inequality  with $C_P$ depending only on $N$ and $K$ \cite{Rajala12,Rajala12-2}. Moreover,
	by the Bishop-Gromov inequality \cite{Sturm06-2}, any $\rcd(K,N)$ space $\Xdm,$ with $N<\infty$, satisfies 
	\begin{equation}\label{eq:RCD doulbing}
		\frac{\mea(B_r(x))}{\mea(B_R(x))}\ge C_{R_0,K,N} \left(\frac{r}{R}\right)^N, \quad \forall x \in \X, \, \forall 0<r<R\le R_0,
	\end{equation}
	where  $C_{R_0,K,N}$ is a constant depending only on $R_0,K,N$. In particular, $\Xdm$ is PI space and by Theorem \ref{thm:improved poincaret} it supports a Sobolev inequality, which we report in the following statement. 
	\begin{prop}\label{prop:sobolev rcd}
		Let $\Xdm$ be an $\rcd(K,N)$ space with $N<\infty$. Then   for every $B_R(x)\subset\X$ with $R\le R_0$ inequalities \eqref{eq:improved poincaret} and \eqref{eq:local sobolev} hold with $s=N$, with the constant $C$ depending only on $K,\,N,\,p$ and $R_0$ and with $2\lambda$ replaced by  one. 
		
		Moreover for every $\delta>0$ it holds 
		\begin{equation}\label{eq:sobolev trick}
			\int_{B_R(x)} f^2 \d \mm\le \delta \int_{B_R(x)} R^2|\nabla f|^2 \d \mm  +  \frac{\tilde C\delta^{-\frac N2} }{\mea(B_R(x))} \left(\int_{B_R(x)} |f|\d \mm\right)^2, \quad \forall f \in W^{1,2}(\X),
		\end{equation}
		where $\tilde C$ is a constant depending only on $K,\,N$ and $R_0$.
	\end{prop}
	\begin{proof}
		From \eqref{eq:RCD doulbing} we have that \eqref{eq:pi dimension} holds with $s=N$ and with a constant $c$ again depending only on $N$ and $K$. In particular, $\Xdm$ is uniformly locally doubling with doubling constant $C_D$ depending only on $N$ and $K$. Moreover $\Xdm$ supports a weal local   $(1,1)$-Poincar\'e inequality as mentioned above. Hence the first part of the statement follows directly from Theorem \ref{thm:improved poincaret}. The fact that $2\lambda$ can be taken replaced by the constant one follows by the last part of Theorem \ref{thm:improved poincaret} and the fact that $(\X,\sfd)$ is geodesic (see \cite{Sturm06-2}).
		
		Finally inequality \eqref{eq:sobolev trick} follows from  \eqref{eq:local sobolev} with $p=2$ by using that
		$$ \left(\fint_{B_R(x)} f^{2} \d \mm \right)^\frac 1{2}\le \delta  \left(\fint_{B_R(x)} f^{2^*} \d \mm \right)^\frac 1{2^*}+\delta^{-N/2}\fint_{B_R(x)} |f| \d \mm $$ for every $\delta>0.$
	\end{proof}
	Inequality \eqref{eq:sobolev trick}, contained in the previous statement, was inspired by \cite[$(5.4)$]{CiaMaz18}.
	\medskip
	
	The next result is a slight variation of inequality \eqref{eq:improved poincaret}.
	\begin{lemma}
		Let $\Xdm$ be a bounded $\rcd(K,N)$ space, $N<\infty,$ and fix $p\in (1,N).$ Then for every $B_R(x)\subset \X$ with $R\le R_0$ there exists a constant $C>0$ depending only on $R_0,\,K,\,N$ and $p$ such that
		\begin{equation}\label{eq:sobolev-poincare median}
			\left(\fint_{B_R(x)}|f-m_f|^{p^*}\d \mm \right)^{\frac1{p^*}}\le C  \left(\fint_{B_R(x)}|\nabla f|^{p}\d \mm \right)^{\frac1{p}},\quad \forall \, f \in W^{1,p}(\X),
		\end{equation}
		where $p^*\coloneqq \frac{Np}{N-p}$ and $m_f$ is any median for $f$ in $B_R(x).$
	\end{lemma}
	\begin{proof}
		Set $B\coloneqq B_R(x)$. Without loss of generality, we can assume that $\mea(B)=1.$
		It is sufficient to show that
		\begin{equation}\label{eq:pre median}
			\left(\int_{B}|g|^{p^*}\d \mm \right)^{\frac1{p^*}}\le \tilde C \left(\int_{B}|\nabla g|^{p}\d \mm \right)^{\frac1{p}},\quad \forall \, g\in W^{1,p}(\X) \text{ such that $\mea(\{|g|>0\}\cap B)\le 1/2$,}
		\end{equation}
		where $\tilde C>0$ is a constant depending only on $R_0,\,K,\,N$ and $p$. Indeed to conclude it would be enough to apply \eqref{eq:pre median}  to  $g\coloneqq (f-m_f)^+$ and  $g\coloneqq (f-m_f)^-$. To show \eqref{eq:pre median} we use the Sobolev-Poincar\'e inequality \eqref{eq:improved poincaret} in combination with the H\"older inequality
		\begin{align*}
			\|g\|_{L^{p^*}(B)}&\le \tilde C \|\nabla g\|_{L^p(B)}+ \left|\int g\,\d \mm \right|\le \tilde C \|\nabla g\|_{L^p(\mm)} +\mea(\{|g|>0\}\cap B)^{1-\frac1{p^*}}  	\|g\|_{L^{p^*}(B)} \\
			&\le \tilde C \|\nabla g\|_{L^p(B)}+ \left(\frac12\right)^{1-\frac1{p^*}}  	\|g\|_{L^{p^*}(B)},
		\end{align*}
		which proves \eqref{eq:pre median}.
	\end{proof}

	We conclude with a version of \eqref{eq:sobolev trick} for $\rcd(K,\infty)$ spaces.
	\begin{prop}
		Let $\Xdm$ be a bounded $\rcd(K,\infty)$ space with $\diam(\X)\le D<+\infty$. Then for every $\delta<\min(1,  (K^{-}D^2)^{-1})$ it holds
		\begin{equation}\label{eq:sobolev trick inf}
			\int f^2 \d \mm\le 32\delta^2 D^2 \int |\nabla f|^2 \d \mm  +  \frac{8e^{\frac2 \delta}}{\mm(\X)} \left(\int |f|\d \mm\right)^2, \quad \forall f \in W^{1,2}(\X).
		\end{equation}
	\end{prop}
	\begin{proof}
		We can assume that $\mm(\X)=1$ and $\|f\|_{L^2(\mm)}=1.$
		From the {\rm HWI} inequality (see e.g.\ \cite[eq.\ $(5.7),(5.9)$]{GMS15} or \cite[Corollary 30.22]{Villani09}) we have 
		\begin{align*}
			\int f^2 \log(f^2) \d \mm&\le 2W_2(f^2\mm,\mm) \left( \int |\nabla f|^2\d \mm \right)^\frac12 + \frac{K^{-}}2 W_2(f^2\mm,\mm)^2
			\\&\le 2 D \left( \int |\nabla f|^2\d \mm \right)^\frac12 + \frac{K^{-}}2 D^2,
		\end{align*}
		where $W_2(\cdot,\cdot)$ denotes the $2$-Wasserstein distance in $(\X,\sfd).$ Combining this with the elementary inequality 
		$$t^2\le \delta t^2\log(t^2)+e^{\frac1 \delta} t, \quad \text{for all $t\ge 0$ and $\delta\in(0,1)$ }$$
		we reach
		\[
		1=\int f^2 \d \mm \le 2 \delta D \left( \int |\nabla f|^2\d \mm \right)^\frac12 + \delta \frac{K^{-}}2 D^2+ e^{\frac1 \delta} \int |f|\d \mm.
		\]
		Assuming $\delta < (K^{-}D^2)^{-1} $ and squaring on both sides gives \eqref{eq:sobolev trick inf}.
	\end{proof}

	\section{Approximation results for the \texorpdfstring{$p$}{p}-Laplacian}\label{sec:approximationresults}
	
	The main goal of this section is to obtain some approximation results for the $p$-Laplacian. To do so we introduce the $\varepsilon$-regularised operator used, for instance, in  \cite{dibenedetto_alphalocalregularityweak_1983} by DiBenedetto to prove $C^{1,\beta}$-regularity of solutions to the $p$-Poisson equation in the flat Euclidean case.
	\begin{definition}[$(p,\eps)$-Laplacian]
		Fix $p \in (1,\infty)$, $\eps> 0$ and let $\Xdm$ be an $\rcd(K,\infty)$ space. A function $u \in W^{1,p}(\X)$ belongs to $\dom(\Delta_{p,\eps})$ if and only if there exists (unique) $\Delta_{p,\eps}u\in L^{1}_\loc(\X)$ such that
		\begin{equation}\label{eq:def eps plapl}
			\int_{\X}\la (|\nabla u|^2+\eps)^\frac{p-2}2\nabla u,\nabla \phi\ra \d \mm=-\int_\X \phi \Delta_{p,\eps}u  \d \mm, \quad \forall \phi \in \LIP_{bs}(\X).
		\end{equation}
	\end{definition}
	As in the definition of $p$-Laplacian, the scalar product appearing in \eqref{eq:def eps plapl} is the one of $L^0(T\X).$
	The integral on the left-hand side of \eqref{eq:def eps plapl} is well defined, indeed if $p\ge2$ there exists a constant $c_{p,\eps}>0$ such that
	\begin{equation}\label{eq:eps well def >2}
		(|\nabla u|^2+\eps)^\frac{p-2}2|\nabla u|\le c_{p,\eps} (|\nabla u|^{p-1}+1),
	\end{equation}
	while for $p<2$ 
	\begin{equation}\label{eq:eps well def <2}
		(|\nabla u|^2+\eps)^\frac{p-2}2|\nabla u|\le |\nabla u|^{p-1},
	\end{equation}
	and the right-hand sides of both \eqref{eq:eps well def <2} and \eqref{eq:eps well def >2} are in $L^1_\loc(\X)$, because $|\nabla u|\in L^p(\mm)$. In fact, we are not yet claiming any extra regularity property of $\Delta_{p,\eps}$ with respect to $\Delta_p$. The existence of regular solutions to $\Delta_{p,\eps} u=f$ will be instead discussed in \textsection \ref{sec:eps regularity}.
	\begin{remark}\label{rmk:test w1p eps plapl}
		By the density of $\LIP_{bs}(\X)$ in $W^{1,p}(\X)$ and  thanks to \eqref{eq:eps well def <2} and \eqref{eq:eps well def >2}  we have that \eqref{eq:def eps plapl} holds also for every $\phi \in W^{1,p}\cap L^\infty(\X)$ having bounded support. Moreover if $\Delta_{p,\eps} u \in L^{p'}(\X)$, where $p'\coloneqq \frac p{p-1}$, the validity of \eqref{eq:def eps plapl} extends also to all $\phi \in W^{1,p}(\X)$  (cf.\ with Remark \ref{rmk:test w1p plapl}). \fr
	\end{remark}
	We will prove two approximation results: the first one  (Proposition \ref{prop:existence of weak solutions})  says that solutions of $\Delta_{p,\eps} u=f$ converge to a solution of $\Delta_p u=f$, as $\eps \to 0^+$; the second (Proposition \ref{prop:L1 theory}) says that if $f_n \to f \in L^1(\mm)$, then solutions of $\Delta_p u=f_n$ converge to a solution of $\Delta_{p} u=f$.
	
	\subsection{Approximation via  regularized \texorpdfstring{$p$}{p}-Laplacian operators}
	Here, we prove that solutions to $\Delta_{p,\eps} u=f$ converge to a solution of $\Delta_p u=f$, as $\eps \to 0^+$. This statement is made precise by the following proposition.
	\begin{prop} [$\eps$-approximation of the $p$-Laplacian]\label{prop:existence of weak solutions}
		Fix $p\in(1,\infty)$ and let $p'\coloneqq \frac{p}{p-1}$. Let $\Xdm$ be a bounded $\rcd(K,\infty)$ space and  let $f \in L^{p'}(\mm)$.
		Suppose that the sequence  $u_n \in W^{1,p}(\X)$,  satisfies $\Delta_{p,\eps_n}(u_n)=f$ with $\eps_n \to 0^+$ and $\int u_n \d \mm=0$. Suppose also that $u \in W^{1,p}(\X)$ satisfies $\Delta_p(u)=f$ and $\int u \d \mm=0$.
		Then $u_n \to u$ in $W^{1,p}(\X)$.
	\end{prop}
	\begin{proof}
		Without loss of generality, we assume that $\mea(\X)=1.$ Moreover, up to ignoring finitely many elements of the sequence $u_n$, we can assume that $\eps_n\le 1.$ We start by deriving uniform bounds on $u_n$. If $p\ge 2,$ by choosing $u_n$ itself as test function in \eqref{eq:def eps plapl} (recall Remark \ref{rmk:test w1p eps plapl}) and  by the Young's inequality we get
		\[
		\int_\X |\nabla u_n|^p\d \mm\le -\int_\X f u_n\d \mm \le \int_\X c_{p,\delta}|f|^{p'}  \d \mm + \delta \int_\X |u_n|^p\d \mm,
		\]
		for every $\delta>0$ and for some constant $c_{p,\delta}>0$ depending only on $p$ and $\delta$. Similarly for $p<2$ 
		\[
		\frac12\int_\X |\nabla u_n|^p\d \mm-c_p\le \int_\X f u_n\d \mm \le \int_\X c_{p,\delta}|f|^{p'}  \d \mm + \delta \int_\X |u_n|^p\d \mm,
		\]
		where in the first inequality we used that $\eps\le 1$ and that $\mu(t)\coloneqq\frac{t^2}{(t^2+1)^{\frac{2-p}{2}}}-\frac12t^{p}\ge -c_p$ for every $t\ge 0$  and for some constant $c_p>0$ (indeed $\mu(t)\to 0$ as $t \to 0^+$ and $\mu(t)\to +\infty$ as $t\to +\infty$).  Therefore in both cases choosing $\delta>0$ small enough and by the $(p,p)$-Poincaré inequality of Theorem \ref{thm:pPoincare} and we obtain that $u_n$ is bounded in $W^{1,p}(\X).$
		Taking $u_n-u$ as test function both in $\Delta_{p,\eps_n}(u_n)=f$ and $\Delta_{p}(u)=f$ and subtracting the two identities gives
		\[
		\int \la |\nabla u|^{p-2}\nabla u,\nabla (u-u_n)\ra-\la (|\nabla u_n|^2+\eps_n)^{\frac{p-2}{2}}\nabla u_n,\nabla (u-u_n)\ra\d \mm  =0
		\]
		that we rewrite as
		\begin{equation}\label{eq:difference un u}
			\begin{multlined}[c][.7\textwidth]
				\int \la |\nabla u|^{p-2}\nabla u-|\nabla u_n|^{p-2}\nabla u_n,\nabla (u-u_n)\ra\\=\int \left((|\nabla u_n|^2+\eps_n)^{\frac{p-2}{2}}-|\nabla u_n|^{p-2}\right) \la \nabla u_n,\nabla (u-u_n)\ra\d \mm.
			\end{multlined}
		\end{equation}
		The right-hand side is dominated by
		\begin{align*}
			\begin{multlined}[t][.7\textwidth]
				\int \left|(|\nabla u_n|^2+\eps_n)^{\frac{p-2}{2}}-|\nabla u_n|^{p-2}\right| |\nabla u_n||\nabla (u-u_n)|\d \mm\\
				\begin{aligned}
					&\le C_p\eps_n^{\alpha_p} \int  (|\nabla u_n|^{p-1}+1)  |\nabla (u-u_n)|\d \mm,
				\end{aligned}
			\end{multlined}
		\end{align*}
		where  $\alpha_p\coloneqq \min(\frac12,\frac{p-1}4)>0$ and where we used that for every $p\in(1,\infty)$ there exists a constant $C_p>0$  depending only on $p$ such that
		\[
		|(t^2+\eps_n)^\frac{p-2}{2}-t^{p-2}|t\le C_p \eps_n^{\alpha_p} (t^{p-1}+1), \quad \forall\,\,t\ge0.
		\] 
		Indeed if $t^2\ge \sqrt{\eps_n}$ we have
		\[
		|(t^2+\eps_n)^\frac{p-2}{2}-t^{p-2}|t=t^{p-1}\left|\left(1+\frac{\eps_n}{t^2}\right)^\frac{p-2}{2}-1\right|\le c_p t^{p-1}\frac{\eps_n}{t^2}\le c_p\sqrt \eps_nt^{p-1},
		\]
		while if $t^2\le\sqrt \eps_n $
		\[
		|(t^2+\eps_n)^\frac{p-2}{2}-t^{p-2}|t\le \begin{cases}
			2t^{p-1}\le 2\eps_n^\frac{p-1}{4}(t^{p-1}+1), & \text{if $p<2$},\\
			2(t^2+\eps_n)^{\frac{p-2}{2}}t\le 2^\frac p2\eps_n^\frac{p-1}{4}(t^{p-1}+1),& \text{if $p\ge2$}.
		\end{cases}
		\]
		We proceed applying the H\"older inequality 
		\begin{equation}\label{eq:nabla un eps}
			\begin{multlined}[c][.8\textwidth]
				\int \left|(|\nabla u|^2+\eps_n)^{\frac{p-2}{2}}-|\nabla u|^{p-2}\right| |\nabla u_n||\nabla (u-u_n)|\d \mm\\\le  C_p\eps_n^{\alpha_p} (\||\nabla u_n|\|_{L^p(\mm)}+1)^{p-1}(\|\nabla u_n\|_{L^p(\mm)}+\|\nabla u\|_{L^p(\mm)}),
			\end{multlined}
		\end{equation}
		up to increasing the constant $C_p.$
		For $p\ge 2$, combining \eqref{eq:nabla un eps} with \eqref{eq:difference un u} and Lemma \ref{lem:monotonicity} we get (up to further increasing the constant $C_p$)
		\[
		\limsup_n \int |\nabla u-\nabla u_n|^p \d \mm\le \limsup_n   C_p\eps_n^{\alpha_p} (\||\nabla u_n|\|_{L^p(\mm)}+1)^{p-1}(\|\nabla u_n\|_{L^p(\mm)}+\|\nabla u\|_{L^p(\mm)})=0,
		\]
		by the the boundedness of $|\nabla u_n|$ in $L^p(\mm).$
		Instead for $p<2$ we first use the H\"older inequality to write, 
		\begin{align*}
			\int |\nabla u-\nabla u_n|^p \d \mm\le \left ( \int \frac{|\nabla u-\nabla u_n|^2}{(|\nabla u|+|\nabla u_n|)^{2-p}} \d \mm \right)^{\frac p2} \left( \int (|\nabla u|+|\nabla u_n|)^p\d \mm \right)^\frac{2-p}{2}
		\end{align*}
		that  combined with \eqref{eq:nabla un eps}, \eqref{eq:difference un u} and Lemma \ref{lem:monotonicity} yields
		\[
		\limsup_n \left(\int |\nabla u-\nabla u_n|^p \d \mm\right)^{\frac{2}{p}}\le \limsup_nC_p\eps_n^{\alpha_p} (\||\nabla u_n|\|_{L^p(\mm)}+1)^{p-1}(\|\nabla u_n\|_{L^p(\mm)}+\|\nabla u\|_{L^p(\mm)})^{3-p}=0.
		\]
		Summing up we obtained that $|\nabla u_n-\nabla u|\to 0$ in $L^p(\mm)$. From this, the claimed convergence in $W^{1,p}(\X)$ follows using again the  $(p,p)$-Poincaré inequality.
	\end{proof}

	\subsection{Approximation via  regularized source term}
	Here we prove the following convergence result. 
	\begin{prop}\label{prop:L1 theory}
		Let $\Xdm$ be a bounded $\rcd(K,N)$ space, $N<\infty$, and let $u \in W^{1,p}(\X)$ be such that $\Delta_p u=f \in L^1(\mm)$. Suppose there exists a sequence $(u_n)\subset W^{1,p}(\X)$ such that $\Delta_p (u_n)=f_n$ with $f_n \to f$ in $L^1(\mm)$. Then there exists a subsequence $u_{n_k}$ such that
		\begin{equation}\label{eq:gradient convergence}
			|\nabla u_{n_k}-\nabla u|^{p-1}\to 0\quad \text{in $L^{1}(\mm).$}
		\end{equation}
	\end{prop}
	
	\begin{proof}
		The argument is inspired by the techniques from \cite{BBGGPV} in the Euclidean case, adapted to our setting.
		Since $\Xdm$ is also an $\rcd(K,N')$ space for every $N'\ge N$ we can assume that $p<N.$
		Up to subtracting a constant from $u$ and $u_n$, $n \in \nn$, by the locality of the gradient, we can assume that zero is a median for $u$ and for all $u_n$, $n \in \nn$.  Additionally, up to passing to a non-relabelled subsequence, we can assume that $\|f_n\|_{L^1(\mm)}\le 2\|f\|_{L^1(\mm)}.$
		To prove \eqref{eq:gradient convergence} it is sufficient to show that there exists $q>1$ such that
		\begin{equation}\label{eq:gradient to zero in meas}
			\begin{split}
				&|\nabla u_n-\nabla u|\to 0 \quad \text{ in $\mea$-measure},\\
				&\sup_n \||\nabla u|^{p-1}\|_{L^q(\mm)}+\||\nabla u_n|^{p-1}\|_{L^q(\mm)}  <+\infty,
			\end{split}
		\end{equation}
		Indeed \eqref{eq:gradient convergence} can be deduced from \eqref{eq:gradient to zero in meas} by standard arguments (see e.g.\ \cite[Lemma 8.2] {HK00}).
		
		We start by deducing several uniform bounds.
		For every $k>0$ define $F_k(t)\coloneqq (-k)\vee t \wedge k,$ $t \in \rr$. Taking as test functions $\phi=F_k\circ u \in W^{1,p}\cap L^\infty(\X)$ and $\phi=F_k\circ u_n\in W^{1,p}\cap L^\infty(\X)$	respectively in the weak formulation of $\Delta_p u=f$ and $\Delta_p (u_n)=f_n$ we obtain that
		\begin{equation}\label{eq:integral gradient bounds}
			\int_{\{|u|\le k\}} |\nabla u|^p\d \mm\le k \|f\|_{L^1(\mm)}, \quad \int_{\{|u_n|\le k\}} |\nabla u_n|^p\d \mm\le k \|f_n\|_{L^1(\mm)},\,\,\, \forall \, n \in \nn.
		\end{equation}
		Moreover, zero is a median for both $F_k\circ u$ and $F_k\circ u_n$, therefore applying \eqref{eq:sobolev-poincare median} we get that for every $k>0$
		\[
		\int |F_k\circ u|^{p^*}\d \mm \le C_p(k \|f\|_{L^1(\mm)})^{p^*/p}, \quad 	\int |F_k\circ u_n|^{p^*}\d \mm \le C_p (k \|f_n\|_{L^1(\mm)})^{p^*/p},\,\,\, \forall \, n \in \nn
		\]
		and by the Markov inequality
		\begin{equation}\label{eq:function measure bounds}
			\begin{split}
				&\mea(\{|u|> k\})=	\mea(\{|F_{2k}\circ u|^{p^*}> k^{p^*}\}) \le 2C_p k^{p^*/p-p^*} \|f\|_{L^1(\mm)},\\
				&\mea(\{|u_n|> k\})=\mea(\{|F_{2k}\circ u_n|^{p^*}> k^{p^*}\}) \le 4C_p k^{p^*/p-p^*} \|f\|_{L^1(\mm)},
			\end{split}
		\end{equation}
		where $C_p$ is a constant independent of $n$ and $k$. Combining \eqref{eq:function measure bounds} and \eqref{eq:integral gradient bounds} we get for every $k>0$ and $t>0$
		\[
		\mea(\{|\nabla u|>k \})\le 	\mea(\{|\nabla u|>k \}\cap\{|u|\le t\}) +\mea(\{|u|> t\}) \le
		\frac{t}{k^p}\|f\|_{L^1(\mm)}+2C_p t^{p^*/p-p^*} \|f\|_{L^1(\mm)}.
		\]
		Taking $t=k^\frac{N-p}{N-1}$ and repeating the same argument for $u_n$ gives
		\begin{equation}\label{eq:gradient measure bounds}
			\begin{split}
				&\mea(\{|\nabla u|>k \})\le  k^{-h}(2C_p+1) \|f\|_{L^1(\mm)},\\
				&\mea(\{|\nabla u_n|>k \})\le k^{-h}(4C_p+2) \|f\|_{L^1(\mm)},
			\end{split}
		\end{equation}
		where $h\coloneqq \frac{N(p-1)}{N-1}>p-1.$ This already shows the second in \eqref{eq:gradient to zero in meas}.
		
		It remains to prove the first in \eqref{eq:gradient to zero in meas}. 
		Fix $\eps>0$ arbitrary. By \eqref{eq:function measure bounds} and \eqref{eq:gradient measure bounds} there exists  a constant $k_\eps>0$ such that for every $n\in\nn$
		\begin{equation}\label{eq:eps measure bounds}
			\mea(\{|u|>k_\eps\})+\mea(\{|u_n|>k_\eps\})+\mea(\{|\nabla u|>k_\eps\})+\mea(\{|\nabla u_n|>k_\eps\})<\eps/2.
		\end{equation}
		Taking as test function $\phi=F_{2k_\eps}\circ (u_n- u) \in W^{1,p}\cap L^\infty(\X)$ in both $\Delta_p u_n=f_n$ and $\Delta_p (u)=f$ and subtracting the two resulting identities, we obtain
		\[
		\begin{split}
			\int_{\{|u_n-u|<2k_\eps\}} \la|\nabla u|^{p-2}\nabla u-|\nabla u_n|^{p-2}\nabla u_n, \nabla u-\nabla u_n\ra \d \mm\le 2k_\eps \|f-f_n\|_{L^1(\mm)}.
		\end{split}
		\] 
		Applying Lemma \ref{lem:monotonicity} and noting that $\{|u_n|<k_\eps,\, |u|<k_\eps\}\subset \{|u_n-u|<2k_\eps\}$ we get that for some constant $c_p>0$ depending only on $p$
		\begin{equation}\label{eq:monotonicity estimates}
			\begin{aligned}[c]
				&\int_{\{|u_n|<k_\eps,\, |u|<k_\eps\}} c_p|\nabla u-\nabla u_n|^p \le  2k_\eps \|f-f_n\|_{L^1(\mm)},&\text{ if $p\ge 2$,}\\
				&\int_{\{|u_n|<k_\eps,\, |u|<k_\eps,\, |\nabla u|<k_\eps, \,  |\nabla u_n|<k_\eps \}} c_p (2k_\eps)^{2-p} |\nabla u-\nabla u_n|^2 \le 2k_\eps \|f-f_n\|_{L^1(\mm)}, &\text{ if $p< 2$.}
			\end{aligned}
		\end{equation}
		Therefore in both cases, applying the Markov inequality, for every $\delta>0$ it holds
		\begin{equation*}
			\begin{multlined}[c][.8\textwidth]
				\mea\big(\{|\nabla u-\nabla u_n|>\delta, \, |u_n|<k_\eps,\, |u|<k_\eps, |\nabla u|<k_\eps, \,  |\nabla u_n|<k_\eps\}\big)\\\le 2 c_p^{-1} k_\eps\delta^{-(p\vee 2)} \max((2k_\eps)^{2-p},1)  \|f-f_n\|_{L^1(\mm)}.
			\end{multlined}
		\end{equation*}
		Combining this with \eqref{eq:eps measure bounds} yields
		\[
		\mea(\{|\nabla u-\nabla u_n|>\delta\})\le \delta^{-p} 2k_\eps \|f-f_n\|_{L^1(\mm)} + \eps/2.
		\]
		Since by assumption $\|f_n-f\|_{L^1(\mm)}\to 0$ as $n\to +\infty$ and by the arbitrariness of $\eps>0$ we obtain the first in \eqref{eq:gradient to zero in meas} and conclude the proof.
	\end{proof}

	\section{Uniform a priori  estimates for \texorpdfstring{$\Delta_{p,\eps}$}{\textDelta\_(p,\textepsilon)}}\label{sec:eps regularity}
	The goal of this section is to obtain second-order regularity and gradient estimates for solutions of 
	\begin{equation}\label{eq:eps poisson}
		\Delta_{p,\eps} u =f \in L^2(\mm),
	\end{equation}
	\emph{assuming} a priori that $u$ has second-order regularity, that is $u \in \dom(\Delta)\subset W^{2,2}(\X).$ Technically speaking we will actually not consider exactly solution of \eqref{eq:eps poisson}, since we do not want to assume extra integrability of $|\nabla u|$ other than $L^2$. Instead we consider the operator $D_{\eps,p}(u)$, defined below, given by formally expanding $\div((|\nabla u|^2+\eps)^{\frac{p-2}{2}}\nabla u).$ 
	
	\begin{definition}[Developed $(p,\eps)$-Laplacian]
		Let $\eps>0 $ and $p\in(1,\infty)$. We define the operator $D_{p,\eps}: \dom(\Delta)\to L^2(\mm)$, that we call \textit{developed $(p,\eps)$-Laplacian}, as
		\begin{equation}\label{eq:deps}
			D_{p,\eps}(u)\coloneqq\Delta u+(p-2)\frac{\H u(\nabla u, \nabla u)}{|\nabla u|^2+\eps}=\Delta u+(p-2)\frac{|\nabla u|\la\nabla |\nabla u|,\nabla u\ra}{|\nabla u|^2+\eps}.
		\end{equation}
	\end{definition}
	Note that $D_{p,\eps}$ is defined only for functions in $\dom(\Delta)$, while $\Delta_{p,\eps}$ makes sense  for all functions in $W^{1,p}(\X).$ Nevertheless, at least at a formal level, we have $(|\nabla u|^2+\eps)^\frac{p-2}{2}D_{p,\eps} u=\Delta_{p,\eps} u$. 
	
	We will start in  \textsection \ref{sec:developed pre} below to deduce some initial properties and estimates concerning the developed ($p,\eps$)-Laplacian. Then we will establish $L^2$-uniform estimates on the  Hessian (in \textsection \ref{sec:developed second est}) and $L^\infty$-uniform estimates on the gradient (in \textsection  \ref{sec:developed grad est}), both involving $D_{p,\eps}$ but independent of the parameter $\eps$. These estimates will play a crucial role in the proof of our main results, in particular in obtaining regularity estimates for $\Delta_p$ by sending $\eps\to 0^+$.
	
	\subsection{Preliminary properties and estimates for the developed \texorpdfstring{($p,\eps$)}{(p,\textepsilon)}-Laplacian} \label{sec:developed pre}
	The following result gives a rigorous version of the fact that  $(|\nabla u|^2+\eps)^\frac{p-2}{2}D_{p,\eps} u=\Delta_{p,\eps} u$.
	\begin{lemma}\label{lem:develop}
		Fix $p \in(1,\infty)$ and $\eps>0.$
		Let  $\Xdm$ be an $\rcd(K,\infty)$ space and let $u \in\ \dom(\Delta)$. If $p\ge 2$ suppose in addition that  $|\nabla u|^{p-1}\in L^{1}(\mm),\,|\nabla u|^{2(p-2)}\in L^{1}(\mm)$.  Then 
		\begin{equation}\label{eq:not yet eps plapl}
			\int_{\X}\la (|\nabla u|^2+\eps)^\frac{p-2}2\nabla u,\nabla \phi\ra \d \mm=-\int_\X \phi (|\nabla u|^2+\eps)^{\frac{p-2}{2}} D_{p,\eps}(u)   \d \mm, \quad \forall \phi \in \LIP_{bs}(\X).
		\end{equation}
		In particular if also $u \in W^{1,p}(\X)$ then $u \in \dom(\Delta_{p,\eps})$ and $\Delta_{p,\eps}u=(|\nabla u|^2+\eps)^{\frac{p-2}{2}} D_{p,\eps}(u).$
	\end{lemma}
	\begin{proof} \textsc{Case $p\ge 2$.}	
		Fix $k\in \nn$. Since $|\nabla u|\in \W(\X)$ (see Proposition \ref{prop:gradgrad}) we have $w_k\coloneqq ((|\nabla u|\wedge k)^2+\eps)^{\frac{p-2}{2}} \in W^{1,2}\cap L^\infty(\X)$ and
		\begin{equation}\label{eq:gradient lq}
			\nabla w_k=(p-2)\nchi_{\{|\nabla u|\le k\}}|\nabla u|(|\nabla u|^2+\eps )^{\frac{p-2}{2}-1}\nabla |\nabla u|.
		\end{equation}
		Moreover, for every $\phi \in \LIP(\X)$ we have $w_k\phi \in \W(\X)$. Therefore by the Leibniz rule for the gradient and integrating by parts we obtain
		\[
		\int w_k\la \nabla u, \nabla \phi \ra\d \mm =\int \la \nabla u, \nabla (\phi w_k) \ra-\phi\la \nabla u,  w_k \ra\d \mm
		= \int -\phi (w_k\Delta u+\la \nabla u,  w_k \ra)\d \mm. \]
		Substituting the expression for $w_k$ and $\nabla w_k$ (as given in \eqref{eq:gradient lq}) 
		\begin{equation*}
			\begin{multlined}[c][.8\textwidth]
				\int ((|\nabla u|\wedge k)^2+\eps)^{\frac{p-2}{2}}\la \nabla u, \nabla \phi \ra\d \mm\\ =\int
				-\phi ((|\nabla u|\wedge k)^2+\eps)^{\frac{p-2}{2}}\bigg(\Delta u+\nchi_{\{|\nabla u|\le k\}}(p-2)|\nabla u| \frac{\la\nabla u,  \nabla |\nabla u| \ra}{|\nabla u|^2+\eps}\bigg)\d \mm.
			\end{multlined}
		\end{equation*}
		Thanks to the assumptions $|\nabla u|^{p-1}\in L^{1}(\mm)$ and $|\nabla u|^{2(p-2)}\in L^{1}(\mm)$ and since $\Delta u \in L^2(\mm)$ and $ |\nabla |\nabla u||\in L^2(\mm)$, we deduce  that both the integrands  are dominated by an $L^1(\mm)$-function independet of $k$. Hence letting $k \to +\infty$ and using the dominated convergence theorem we deduce that \eqref{eq:def eps plapl} holds with $\Delta_{p,\eps} u= (|\nabla u|^2+\eps)^{\frac{p-2}{2}} D_{p,\eps}(u)$ and so the conclusion follows.

		\noindent\textsc{Case $p< 2$.}
		In this case, we have directly that $(|\nabla u|^2+\eps)^{\frac{p-2}{2}} \in \W\cap L^\infty(\X)$ with
		$$\nabla(|\nabla u|^2+\eps)^{\frac{p-2}{2}}=(p-2)|\nabla u|(|\nabla u|^2+\eps )^{\frac{p-2}{2}-1}\nabla |\nabla u|,$$
		without the need of a cut-off.
		Indeed, $|\nabla u|\in \W(\X)$ and the function $(t^2+\eps)^{(p-2)/2} \in C^1(\rr)$ is bounded with bounded derivative.
		Therefore $(|\nabla u|^2+\eps)^{\frac{p-2}{2}}\nabla \phi\in \W(\X)$ for every $\phi \in \LIP_{bs}(\X)$. Arguing as above integrating by parts we reach
		\[
		\int (|\nabla u|^2+\eps)^{\frac{p-2}{2}}\la \nabla u, \nabla \phi \ra\d \mm =\int
		-\phi (|\nabla u|^2+\eps)^{\frac{p-2}{2}}\bigg(\Delta u+(p-2)|\nabla u| \frac{\la\nabla u,  \nabla |\nabla u| \ra}{|\nabla u|^2+\eps}\bigg)\d \mm,
		\]
		which is the sought conclusion.
	\end{proof}
	For convenience, given $u \in \dom(\Delta)$ we write recalling \eqref{eq:musical}
	\[
	\Delta_\infty(u)\coloneqq |\nabla u|\la \nabla |\nabla u|,\nabla u \ra=\H u (\nabla u,\nabla u) \in L^0(\mm).
	\]
	Hence, we can rewrite the operator $D_{\eps,p}$ as
	\begin{equation}\label{eq:deps inf}
		D_{\eps,p}(u)=\Delta u+(p-2) \frac{\Delta_\infty(u)}{|\nabla u|^2+\eps}, \quad \forall p \in(1,\infty),\, \eps>0.
	\end{equation}
	Next, we prove a technical inequality inspired by some computations done in the smooth setting in \cite[Lemma 3.4]{LZZ21}, that will be used in the next section to obtain uniform regularity estimates for $D_{\eps,p}.$
	\begin{prop}\label{prop:nightmare}
		For every $N\in[2,\infty]$, $p\in(1,\infty)$ and  $\alpha > \frac12 \big(p-3-\frac{p-1}{N-1}\big)$
		there exists constants $\lambda=\lambda(p,N,\alpha)\in(0,1)$ and $C_1=C_1(p,N,\alpha)\ge 0$ such that the following holds. Let $\Xdm$ be an $\rcd(K,N)$ space, with $n\coloneqq\dim(\X)\ge 2$ (if $N<\infty$) and $u \in \dom(\Delta)$. Then for every $\eps>0$ 
		\begin{equation}\label{eq:infinity laplacian estimate}
			\begin{split}
				[(p-2)^2-2\alpha(p-2)]&\frac{(\Delta_\infty(u))^2}{(|\nabla u|^2+\eps)^2}-2\alpha\frac{|\nabla |\nabla u||^2 |\nabla u|^2}{|\nabla u|^2+\eps}\\
				&\le \lambda \left[	|\H u|^2+\frac{(\Delta u-\tr \H u)^2}{N-n}\right]+C_1\,(D_{p,\eps}(u))^2,
			\end{split}
		\end{equation} 
		where the term containing $N-n$ is not present if either $N=+\infty$ or $N=n.$
		Moreover, if $\alpha\ge p-2$ the constants $\lambda$ and $C_1$ can be taken independent of both $\alpha$ and $N$.
	\end{prop}
	In the above proposition, it might look strange to have both the dependence on $N$ of the constants and the possibility of taking $N=\infty$. However note that the range of the admissible $\alpha$'s increases as $N$ decreases, hence we cannot take $C_1$ and $\lambda$ to depend only on $\alpha$ and $p.$
	The proof of Proposition \ref{prop:nightmare} relies on the following fundamental inequality, well known in the smooth case (see e.g.\ \cite{Sa22} and also \cite{Colding12,WZ11,LZZ21,FMP19}).
	\begin{prop}[Key inequality]
		Let $\Xdm$ be an $\rcd(K,N)$ space with $N\in[1,\infty]$. Then for every $u \in W^{2,2}(\X)$  it holds
		\begin{equation}\label{eq:key}
			|\nabla u|^4|\H u|^2 \ge  2|\nabla u|^4|\nabla |\nabla u||^2+\frac{\left(|\nabla u|^2\tr \H u - \Delta_\infty u\right)^2}{\dim(\X)-1}-(\Delta_\infty u)^2, \quad \mea\text{-a.e.,}
		\end{equation}
		where term containing $\dim(\X)$ is not present if  $N=\infty$ or $\dim(\X)=1$. 
	\end{prop}
	\begin{proof}
		First we consider the case $N<\infty.$
		The statement is a direct consequence of an elementary inequality for quadratic operators on finite dimensional vector space (see \cite[Lemma 2.1]{Sa22}). Indeed, for every symmetric real matrix $A \in \rr^{n\times n}$, $n \ge 2$ it holds
		\[
		|v|^4|A|^2\ge 2|v|^2|A v|^2+\frac{(|v|^2\tr A-\la Av,v\ra )^2}{n-1}-\la v,Av \ra^2, \quad \forall v \in \rr^n,
		\]
		where $|A|$ denotes the Hilbert-Schmidt norm of $A$. If  $n=1$ the above is still trivially true (with equality)   without the term containing $n-1$.  Inequality
		\eqref{eq:key} now follows by computing in coordinates recalling the second in \eqref{eq:musical}.

		If $N=\infty$ we recall that since $\H u$ is a symmetric tensor and by the construction of the tensor products of Hilbert modules there exist tensors $A_k$, $k\in \nn$ of the form $A_k=\sum_{i,j=1}^{N_k} a^k_{i,j} e_i\otimes e_j$, such that $\{a^k_{i,j}\}_{i,j}$ is a symmetric  real matrix, $\la e_i,e_j \ra=\delta_{i,j}$ $\mm$-a.e.\  and $|A_k-\H u|_{HS}\to 0$ in $L^2(\mm)$ (see \cite[Section 1.5]{Gigli14}). Computing in coordinates, thanks to the above inequality with $n=N_k$ we obtain that
		\[
		|w|^4|A_k|_{HS}^2\ge 2|w|^2|A_k w|^2-\la w,A_k w \ra^2, \quad \text{$\mm$-a.e.\ for every $w\in L^0(T\X)$}.
		\]
		Passing to the limit, since $|A_k|_{HS}\to |\H u|_{HS}$, $|A_k w|\to |\H u(w)|$ and $\la w,A_k w \ra\to \la w,\H u w \ra$ all hold $\mm$-a.e., we obtain \eqref{eq:key}.
%
	\end{proof}
	
	We can now prove the main technical estimate.
	\begin{proof}[Proof of Proposition \ref{prop:nightmare}]
		Note that  \eqref{eq:infinity laplacian estimate} is trivially satisfied whenever $|\nabla u|=0,$ since the left-hand side is zero by the locality of the gradient. Hence all the estimates from now on will be done on the set $\{|\nabla u|>0\}$ (well defined up to measure zero sets).
		
		From \eqref{eq:key} 
		\begin{equation}\label{eq:passo 1}
			|\H u|^2\ge 2|\nabla |\nabla u||^2+ \frac{(\tr \H u-\frac{\Delta_\infty(u)}{|\nabla u|^2} )^2}{n-1}-\frac{(\Delta_\infty(u))^2}{|\nabla u|^4}, \quad \alme,
		\end{equation} 
		without the term containing $n$ in the case $N=\infty.$ 
		We proceed assuming that $N<\infty$ and by giving a lower bound on the term containing $n$.
		
		First, by \eqref{eq:deps inf}, we have the following identity
		$$
		\tr \H u-\frac{\Delta_\infty(u)}{|\nabla u|^2}=\frac{\Delta_\infty(u)}{|\nabla u|^2}\left((2-p) \frac{|\nabla u|^2}{|\nabla u|^2+\eps}-1 \right)+D_{\eps,p}(u)+\tr \H u-\Delta u.
		$$
		Therefore, using twice the Young's inequality, we obtain
		\begin{equation}\label{eq:delta1delta2}
			\begin{split}
				\bigg(\tr \H u-\frac{\Delta_\infty(u)}{|\nabla u|^2}\bigg )^2 &\ge\begin{multlined}[t] (1-\delta_1) \left[\frac{\Delta_\infty(u)}{|\nabla u|^2}\left((2-p) \frac{|\nabla u|^2}{|\nabla u|^2+\eps}-1 \right)+ D_{\eps,p}(u)\right]^2\\
					+\left(1-\frac1\delta_1\right)(\Delta u-\tr \H u)^2\end{multlined}\\
				&\ge \begin{multlined}[t](1-\delta_1-\delta_2)\frac{(\Delta_\infty(u))^2}{|\nabla u|^4}\left((2-p) \frac{|\nabla u|^2}{|\nabla u|^2+\eps}-1 \right)^2\\+\left(1-\delta_1-\frac1{\delta_2}\right)D_{\eps,p}(u)^2
					+\left(1-\frac1\delta_1\right)(\Delta u-\tr \H u)^2,\end{multlined}
			\end{split}
		\end{equation}
		for every $\delta_1,\delta_2>0$. This yields
		\begin{equation}\label{eq:delta2}
			\begin{split}
				\frac{\big(\tr \H u-\frac{\Delta_\infty(u)}{|\nabla u|^2}\big)^2}{n-1} &\ge \bigg(\frac{1}{N-1}-\frac{\delta_2}{n-1}\bigg)\frac{(\Delta_\infty(u))^2}{|\nabla u|^4}\left((2-p) \frac{|\nabla u|^2}{|\nabla u|^2+\eps}-1 \right)^2\\
				&\qquad+\left(\frac{1}{N-1}-\frac1{(n-1)\delta_2}\right)D_{\eps,p}(u)^2-\frac{(\Delta u-\tr \H u)^2}{N-n},
			\end{split}
		\end{equation}
		where the last term is taken to be zero if $N=n.$ 
		To see this in the case $n<N$ we simply choose $\delta_1=\frac{N-n}{N-1}$ in \eqref{eq:delta1delta2} and divide both sides by $n-1$. If instead $N=n$, since $\Delta u=\tr \H u$ (see \eqref{eq:tr=lapl}), we have that \eqref{eq:delta1delta2} actually holds also with $\delta_1=0$ taking $\left(1-\frac1\delta_1\right)(\Delta u-\tr \H u)^2$ to be identically zero. Hence to get \eqref{eq:delta2}  we only need to divide both sides by $N-1=n-1$.

		Resuming the argument for all $N\in[1,\infty]$, we add on both sides of \eqref{eq:passo 1}  the quantity 
		$$-(1+\delta_2)	 \left[(p-2)(p-2-2\alpha)\frac{(\Delta_\infty(u))^2}{(|\nabla u|^2+\eps)^2}-2\alpha\frac{|\nabla |\nabla u||^2 |\nabla u|^2}{|\nabla u|^2+\eps}\right]$$
		and, in the case $N<\infty$, we also plug in inequality \eqref{eq:delta2}. Doing so we reach
		\begin{equation}\label{eq:monster}
			\begin{split}
				\begin{multlined}[t][.9\textwidth]    
					|\H u|^2-(1+\delta_2)\left[(p-2)(p-2-2\alpha)\frac{(\Delta_\infty(u))^2}{(|\nabla u|^2+\eps)^2}-2\alpha \frac{|\nabla |\nabla u||^2 |\nabla u|^2}{|\nabla u|^2+\eps}\right]\ge\\
					\begin{aligned}[t]&\begin{multlined}[t][.4\textwidth]\frac{(\Delta_\infty(u))^2}{|\nabla u|^4} 
							\left[\bigg(\frac{1}{N-1}-\frac{\delta_2}{n-1}\bigg)\left((2-p) \frac{|\nabla u|^2}{|\nabla u|^2+\eps}-1 \right)^2\right.\\\left.\vphantom{\left((2-p) \frac{|\nabla u|^2}{|\nabla u|^2+\eps}-1 \right)^2}-(1+\delta_2)(p-2)(p-2-2\alpha) \frac{|\nabla u|^4}{(|\nabla u|^2+\eps)^2}-1\right]\end{multlined}\\
						&+|\nabla |\nabla u||^2 \left(2+2\alpha(1+\delta_2) \frac{|\nabla u|^2}{|\nabla u|^2+\eps}\right) \\& + \left(\frac{1}{N-1}-\frac1{(n-1)\delta_2}\right)D_{\eps,p}(u)^2-\frac{(\Delta u-\tr \H u)^2}{N-n},
					\end{aligned}
				\end{multlined}
			\end{split}
		\end{equation}
		where if $N=\infty$ all the terms containing $n$ or $N$ are taken to be zero and the last term is taken to be zero also if $N=n.$
		By assumption $$2\alpha>p-3-\frac{p-1}{N-1}=-2+(p-1)\left(1-\frac{1}{N-1}\right)\ge -2.$$
		This means that, even when $\alpha<0$, we have $2\alpha(1+\delta_2)>-2$ if $\delta_2>0$ is chosen small enough depending on $\alpha,$ which we assume from now on. In particular
		\[
		2+2\alpha(1+\delta_2) \frac{|\nabla u|^2}{|\nabla u|^2+\eps}\ge 0.
		\]
		This means that the factor multiplying $|\nabla |\nabla u||^2$  in \eqref{eq:monster} is non-negative and thus we can plug in the inequality $|\nabla |\nabla u||^2\ge \frac{(\Delta_\infty(u))^2}{|\nabla u|^4} $ to obtain
		\begin{align*}
			|\H u&|^2-(1+\delta_2)	 \left[(p-2)(p-2-2\alpha)\frac{(\Delta_\infty(u))^2}{(|\nabla u|^2+\eps)^2}-2\alpha \frac{|\nabla |\nabla u||^2 |\nabla u|^2}{|\nabla u|^2+\eps}\right]\\
			&\ge\begin{multlined}[t][.8\textwidth]\frac{(\Delta_\infty(u))^2}{|\nabla u|^4} \underbracea{\left[\bigg(\frac{1}{N-1}-\frac{\delta_2}{n-1}\bigg)\left((2-p) \frac{|\nabla u|^2}{|\nabla u|^2+\eps}-1 \right)^2\right.}\\\underbracebd{\left.\vphantom{\left((2-p) \frac{|\nabla u|^2}{|\nabla u|^2+\eps}-1 \right)^2}-(1+\delta_2)(p-2)(p-2-2\alpha) \frac{|\nabla u|^4}{(|\nabla u|^2+\eps)^2}+ 
					2(1+\delta_2)\alpha\frac{|\nabla u|^2}{|\nabla u|^2+\eps} + 1\right]}_{\eqqcolon A(\delta_2)\kern11cm}\end{multlined}\\
			&+\left(\frac{1}{N-1}-\frac1{(n-1)\delta_2}\right)D_{\eps,p}(u)^2-\frac{(\Delta u-\tr \H u)^2}{N-n},
		\end{align*}
		where again if $N=\infty$ all the terms  containing $n$ or $N$, are taken to be zero and if $N=n$, the last term is not present.
		To conclude it is sufficient to show that $A(\delta_2)\ge0$ for some $\delta_2>0$ depending only on $N,p$ and $\alpha$. Indeed, \eqref{eq:infinity laplacian estimate} would then follow dividing both sides of the inequality above by $(1+\delta_2).$ 
		To show this we denote $t\coloneqq \frac{|\nabla u|^2}{|\nabla u|^2+\eps}\in [0,1)$ and rewrite
		\begin{align*}
			A(\delta_2)&=\begin{multlined}[t][.6\textwidth]\underbrace{\left[\frac{1}{N-1}((2-p)t-1 )^2-(p-2)(p-2-2\alpha) t^2+2\alpha t+1\right]}_{
					Q(t)\coloneqq}\\
				-\frac{\delta_2}{n-1}((2-p)t-1 )^2-\delta_2(p-2)(p-2-2\alpha)t^2+2\alpha \delta_2 t\end{multlined}\\
			&\ge Q(t)-\delta_2 c(p),
		\end{align*}
		where $c(p)>0$ is a constant depending only on $p$, having used that  $\alpha>\frac12 \left(p-3-\frac{p-1}{N-1}\right)$, $n \ge 2$ and $|t|\le 1$.   Hence, it is enough to show that
		$$\min_{t \in[0,1]} Q(t)\ge \delta(N,p,\alpha)>0,$$
		where $ \delta(N,p,\alpha)>0$ is a constant depending only on $N,p$ and $\alpha$ (independent of $\alpha$ and $N$ if $\alpha\ge p/2$).
		We rewrite the polynomial $Q(t)$, $t \in \R$, as
		\begin{equation}\label{eq:espando A}
			Q(t)= \begin{cases}\left(\frac{(p-2)^2}{N-1}-(p-2)(p-2-2\alpha) \right)t^2+\left(2\frac{(p-2)}{N-1}+2\alpha\right)t+\frac1{N-1}+1,   &\text{if $N<\infty$,}\\
				1-(p-2)(p-2-2\alpha) t^2+2\alpha t,& \text{if $N=\infty$.}
			\end{cases}
		\end{equation}
		$Q(t)$ attains the minimum in $[0,1]$ at $t=0$ or $t=1$, or at some $\bar t \in (0,1)$, where $\frac{\d}{\d t} Q(t)\restr{t=\bar t}=0$. In the latter case we have that 
		\begin{align*}
			Q(\bar t)=\left(\frac{(p-2)}{N-1}+\alpha\right)\bar t+\frac{1}{N-1}+1,
		\end{align*}
		which holds also for $N=\infty.$
		If $\frac{(p-2)}{N-1}+\alpha\ge 0$, then $Q(\bar t)\ge 1.$ Otherwise, since $\bar t \le 1,$
		\begin{align*}
			Q(\bar t)&\ge \frac{(p-2)}{N-1}+\alpha+\frac{1}{N-1}+1\ge\frac{(p-2)}{N-1} +\frac{p-3}{2}-\frac{p-1}{2(N-1)}+\frac{1}{N-1}+1\\
			&=\frac{p-1}{2(N-1)} +\frac{p-3}{2}+1> \frac{p-1}{2(N-1)}\ge0,
		\end{align*}
		where we used the assumption on $\alpha.$ On the other hand $Q(0)=1+\frac1{N-1}.$ 
		Moreover
		\begin{align*}
			Q(1)&=\frac{(p-2)^2}{N-1}-(p-2)(p-2-2\alpha)  +\frac{2(p-2)}{N-1}+2\alpha+\frac1{N-1}+1\\
			&=1-(p-2)^2+\frac{(p-2)^2+2(p-2)+1}{N-1}+2\alpha(p-1)\\
			&=(3-p)(p-1)+\frac{(p-1)^2}{N-1}+2\alpha(p-1)=(p-1)\left( \frac{p-1}{N-1}+3-p+2\alpha\right).
		\end{align*}
		Since by assumption $\alpha>\frac12 \left(p-3-\frac{p-1}{N-1}\right)$ we obtain $Q(1)\ge \delta(p,N,\alpha)>0$ as desired. Moreover, if $\alpha\ge p-2$ then $Q(1)\ge (p-1)^2$, in which case the constant $ \delta(N,p,\alpha)$ can be taken independent of both $N$ and $\alpha$. 
	\end{proof}
	
	\subsection{Second-order regularity estimates}\label{sec:developed second est}
	In this part, we show some $L^2$-weighted bounds on the Hessian, in terms of $L^2$-bounds on $D_{\eps,p}$. The key feature is that the constants appearing in the estimate do not depend on $\eps.$ More precisely, the main goal is to prove the following.
	
	\begin{theorem}\label{thm:p-calderon}
		For every $N\in[2,\infty]$, $p\in(1,3+\frac{2}{N-2})$ and  $\alpha > \frac12 \big(p-3-\frac{p-1}{N-1}\big)$
		there exists a constant $C=C(p,N,\alpha)\ge0$ (independent of $\alpha$ if $\alpha\ge p-2$) such that the following holds. Let $\Xdm$ be an $\rcd(K,N)$ space, with $\dim(\X)=n\ge2$ (if $N<\infty$) and fix $u \in \dom(\Delta)$.
		Then, for every $M>0$, $\eps>0$ and $\eta \in \LIP_{bs}(\X)$ it holds
		\begin{equation}\label{eq:hessian estimate}
			\begin{split}
				\int |\H u|^2& ((|\nabla u|\wedge M)^2+\eps )^{\alpha}\eta^2 \d \mm \\
				&\le C \int \big[(1+|\alpha|^2) (D_{\eps,p} u)^2\eta^2+ |\nabla u|^2(|\nabla \eta|^2+K^-\eta^2)\big]((|\nabla u|\wedge M)^2+\eps)^{\alpha}\d \mm,
			\end{split}
		\end{equation}
		where $M$ is taken $+\infty$ when $\alpha<0$. Moreover, if $|\H u||\nabla u|^{\alpha },|\nabla u|^{\alpha +1}\in L^2(\supp(\eta);\mm)$ the above holds for all $p\in(1,\infty)$ and  $M=+\infty$.
	\end{theorem}
	The core idea to obtain inequality \eqref{eq:hessian estimate} is to plug in suitable test functions in the Bochner inequality and to use the identity
	\[D_{\eps,p}(u)=\Delta u+(p-2) \frac{\Delta_\infty(u)}{|\nabla u|^2+\eps}\]
	(recall \eqref{eq:deps inf}) to get rid of the Laplacian terms. In the computations the technical estimate in Proposition \ref{prop:nightmare}, proved in the previous section, will play a key role.
	Similar arguments appeared recently in the smooth setting for the derivation of second-order estimate for $p$-harmonic functions (see \cite{LZZ21,Sa22,DPZZ20}).
	
	Before proving Theorem \ref{thm:p-calderon}, we state two inequalities following from it.

	\begin{cor}\label{cor:p-calderon covariant}
		For every $N\in[2,\infty)$, $R_0>0$, $K\in \rr$ and $p\in(1,3+\frac{2}{N-2})$ there exists a constant $C=C(p,N,K,R_0)>0$ such that the following holds. Let $\Xdm$ be an $\rcd(K,N)$ space and  let $u \in \dom(\Delta)$. Then for all $M>0,\eps>0$, $x \in \X$, $R\le R_0$  and $r<R$ it holds
		\begin{equation}\label{eq:covariant estimate}
			\begin{split}
				\int_{B_{r}(x)} R^{-2}|v_\eps|^2+|\nabla v_\eps |^2 \d \mm\le\begin{multlined}[t][.6\textwidth]
					C \int_{B_{R}(x)}  (D_{\eps,p} u)^2 ((|\nabla u|\wedge M)^2+\eps)^{p-2} \d \mm\\+ \frac{CR^N\mea(B_{R}(x))^{-1}}{(R-r)^{N+2}}  \left(\int_{B_R(x)} |v_\eps|\d \mm\right)^2,
				\end{multlined} 
			\end{split}
		\end{equation}
		where $v_\eps \coloneqq ((|\nabla u|\wedge M)^2+\eps)^\frac{p-2}{2}\nabla u \in H^{1,2}_C(\X)$ and  $M$ is taken to be $+\infty$ when $p<2.$  
	\end{cor}
	\begin{proof}It sufficient to prove the statement for $r\ge R/2.$
		Since $((|\nabla u|\wedge M)^2+\eps)^\frac{p-2}{2}\in \W\cap L^\infty(\X)$ and $\nabla u \in H^{1,2}_C(T\X)$ (recall Proposition  \ref{prop:gradgrad}) by Lemma \ref{lem:H12sobolev}  we have $v_\eps\in W^{1,2}_C(T\X)$ and
		\begin{equation*}
			\nabla v_\eps = ((|\nabla u|\wedge M)^2+\eps)^\frac{p-2}{2}\nabla \nabla u + \nchi_{\{|\nabla u|\le M\}} |\nabla u|((|\nabla u|\wedge M)^2+\eps)^\frac{p-4}{2} (p-2)\big(\nabla |\nabla u|\otimes \nabla u\big).
		\end{equation*}
		Therefore,
		\begin{equation}\label{eq:convariant bound hessian}
			|\nabla v_\eps |\le |p-1| ((|\nabla u|\wedge M)^2+\eps)^\frac{p-2}{2} |\H u|, \quad \mea\text{-a.e..}
		\end{equation}
		Fix a ball $B_{R}(x)\subset \X$, with $R\le R_0$ and numbers $\frac12\le s<t\le 1$. Let $\eta \in \LIP_c(B_{tR}(x))$ be  such that $\eta=1$ in $B_{sR}(x)$ and $\Lip(\eta)\le R^{-1}(t-s)^{-1}.$
		Combining \eqref{eq:convariant bound hessian} and \eqref{eq:hessian estimate} with $\alpha=p-2$ gives
		\[
		\int_{B_{sR}(x)} |\nabla v_\eps |^2 \d \mm \le C \int_{B_{R}(x)}  (D_{\eps,p} u)^2 ((|\nabla u|\wedge M)^2+k)^{p-2} \d \mm+ \frac{C}{R^2(t-s)^2}\int_{B_{tR}(x)}  |v_\eps|^2 \d \mm,
		\]
		where $C\ge1$ is a constant depending only on $K,N,p$ and $R_0.$ Using inequality \eqref{eq:sobolev trick} with $\delta=\frac12C^{-1}(t-s)^2<1$, recalling that $|\nabla |v_\eps||\le |\nabla v_\eps|$ (see \eqref{eq:gradient covariant}   and by the doubling property of the space, we can bound the last term as follows
		\[
		\frac{C}{R^2(t-s)^2}\int_{B_{tR}(x)} |v_\eps|^2\le \frac12\int_{B_{tR}(x)} |\nabla v_\eps|^2 \d \mm  +  \frac{\tilde C}{R^2(t-s)^{N+2}\mea(B_{R}(x))} \left(\int_{B_R(x)} |v_\eps|\d \mm\right)^2,
		\]
		where $\tilde C$ is a constant depending only on $K,N,p$ and $R_0.$ Therefore we have
		\[
		\int_{B_{sR}(x)} |\nabla v_\eps |^2 \d \mm \le \frac12\int_{B_{tR}(x)} |\nabla v_\eps|^2 \d \mm+ A(R) +  \frac{B(R)}{(tR-sR)^{N+2}},
		\]
		where $A(R)\coloneqq C \int_{B_{R}(x)}  (D_{\eps,p} u)^2 ((|\nabla u|\wedge M)^2+k)^{p-2} \d \mm$ and $B(R)\coloneqq \frac{\tilde CR^N}{\mea(B_{R}(x))} \left(\int_{B_R(x)} |v_\eps|\d \mm\right)^2.$ Applying a standard iteration argument (see e.g.\ \cite[Lemma 4.3]{HanLinbook}) we obtain
		\[
		\int_{B_{r}(x)} |\nabla v_\eps |^2 \d \mm\le c A(R)+ \frac{c}{(R-r)^{N+2}} B(R), \quad \forall r \in \left[\frac R2,R\right),
		\]
		with $c>0$ a constant depending  only on $N.$ This proves \eqref{eq:covariant estimate} without the $|v_\eps|^2$ term on the left-hand side. From this, another application of \eqref{eq:sobolev trick} concludes the proof.
	\end{proof}

	Without an upper bound on the dimension, we have the following.
	\begin{cor}
		For every  $K\in \rr$, $D>0$ and $p\in(1,3)$ there exists a constant $C=C(p,K^-,D)>0$ such that the following holds. Let $\Xdm$ be an $\rcd(K,\infty)$ space with $\diam(\X)\le D$ and  let $u \in \dom(\Delta)$. Then for all $M>0,\eps>0$ it holds
		\begin{equation}\label{eq:covariant estimate N=inf}
			\int_{\X}|\nabla v_\eps |^2 \d \mm\le C \int_{\X}  (D_{\eps,p} u)^2 ((|\nabla u|\wedge M)^2+\eps)^{p-2} \d \mm + CK^-\left(\int_\X |v_\eps|\d \mm\right)^2,
		\end{equation}
		where $v_\eps \coloneqq ((|\nabla u|\wedge M)^2+\eps)^\frac{p-2}{2}\nabla u \in H^{1,2}_C(\X)$  and  $M$ is taken to be $+\infty$ when $p<2.$ 
	\end{cor}
	\begin{proof}
		Simply combine  the inequalities \eqref{eq:hessian estimate} (with $\alpha=p-2$, $\eta \equiv 1$)
		and \eqref{eq:sobolev trick inf}, recalling both 
		$|\nabla |v_\eps||\le |\nabla v_\eps|$ (by \eqref{eq:gradient covariant}) and \eqref{eq:convariant bound hessian}.
	\end{proof}

	We pass to the proof of Theorem \ref{thm:p-calderon}. The first step is the following 	Bochner-type inequality.
	\begin{lemma}[Bochner-type inequality]\label{lem:bochner}
		Let $\Xdm$ be an $\rcd(K,N)$ space, $N\in (1,\infty]$ and let $u \in \dom(\Delta)$. Fix $\psi \in \LIP\cap L^\infty(\rr)$ with $\psi'$ continuous up to a negligible set and satisfying  
		\begin{equation}\label{eq:psi' bound}
			|\psi '(t)|\le c (1+|t|)^{-1}, \quad \quad \text{ for a.e.\ $t \in\rr$,}
		\end{equation} for some constant $c>0$. Then, for every $\eta \in \LIP_{bs}(\X)$, $\eta \ge 0,$ it holds
		\begin{equation}\label{eq:starting point}
			\begin{split}
				\int_\X &\left(|\H u|^2 + \frac{(\tr \H u-\Delta u)^2}{N-\dim(\X)}\right)
				\psi(|\nabla u|) \eta \d \mm \le 	\int_\X \la\nabla u\Delta u -|\nabla u|\nabla |\nabla u|, \nabla \eta\ra \psi(|\nabla u|)\d\mm\\
				&+ \int_\X \la\nabla u\Delta u -|\nabla u|\nabla |\nabla u|, \nabla |\nabla u|\ra \psi'(|\nabla u|)\eta \d\mm
				+\int_\X  ((\Delta u)^2+K^{-}|\nabla u|^2)\psi(|\nabla u|) \eta\d\mm, 
			\end{split}
		\end{equation}
		where the term containing $\dim(\X)$ is not present if $\dim(\X)=N$ or $N=\infty$.
	\end{lemma}

	\begin{proof}
		From the improved  Bochner inequality in \eqref{eq:improvedboch} and integration by parts we have that for every $u \in \test(\X)$ 
		\begin{equation}\label{eq:modified bochner}
			\begin{split}
				\int &\left(|\H u|^2 + \frac{(\tr \H u-\Delta u)^2}{(N-\dim(\X))}\right) \phi \d \mm \\
				&\le \int \la\nabla u\Delta u -\frac12\nabla |\nabla u|^2, \nabla \phi\ra + ((\Delta u)^2+K^{-}|\nabla u|^2) \phi \d\mm, \quad \forall \phi \in \W\cap L^\infty_+(X),
			\end{split}
		\end{equation}
		where the term with $\dim(\X)$ is not present if $\dim(\X)=N$ or $N=\infty$. 
		Fix $\psi$ as in the statement. 
		We choose $\phi\coloneqq\psi(|\nabla u|)\eta$, with $\eta \in \LIP_c(\X)^+$. Since $|\nabla u|\in \W(\X)$ (recall Proposition \ref{prop:gradgrad}), we have that 
		$\phi \in \W\cap L^\infty (\X)$ and 
		$$\nabla \phi =\psi'(|\nabla u|)\nabla |\nabla u| \eta+\nabla \eta \psi(|\nabla u|).$$ Plugging our choice of $\phi$ in \eqref{eq:modified bochner} and using that $\nabla |\nabla u|^2=2|\nabla u|\nabla |\nabla u|$ (which is a consequence of $|\nabla u|\in \W(\X)\cap L^\infty(\mm)$ and  the chain rule) gives
		\begin{align*}
			\int_\X\bigg(|\H u|^2 + \frac{(\tr \H u-\Delta u)^2}{(N-\dim(\X))}\bigg)\psi(|\nabla u|) &\eta \d \mm \\ 	 \le\int_\X \la\nabla u\Delta u -|\nabla u|\nabla |&\nabla u|, \nabla \eta\ra \psi(|\nabla u|)  \d\mm\\
			+\int_\X \la\nabla u&\Delta u -\nabla |\nabla u||\nabla u|, \nabla |\nabla u|\ra \psi'(|\nabla u|)\eta  \d\mm
			\\+ \int_\X  ((\Delta u)^2+K^{-}|&\nabla u|^2) \psi(|\nabla u|)\eta \d\mm,
		\end{align*}
		This proves \eqref{eq:starting point} when $u \in \test(\X)$.
		For a general $u \in \dom(\Delta)$, we consider a sequence $u_n\in\test(\X)$  such that $\Delta u_n \to \Delta u$ in $L^2(\mm)$, $u_n \to u$ in $W^{2,2}(\X)$ and $|\nabla u_n|\to |\nabla u|$ in $\W(\X)$, which existence is given by Lemma \ref{lem:test dense}. We want to pass to the limit on each term using dominated convergence. To this aim, up to passing to a subsequence, we can assume that there exists a function $G \in L^2(\mm)$ such that
		\begin{equation}\label{eq:total domination}
			|\H {u_n}|,\, |\Delta u_n|,\,|\nabla u_n|\le G, \quad \alme, \quad \forall n \in\nn 
		\end{equation}
		and also that $|\nabla u-\nabla u_n|\to 0$ $\mea$-a.e., $|\Delta u-\Delta u_n| \to 0$ $\mea$-a.e., $|\H {u_n-u}|\to 0$ $\mea$-a.e.\ and $|\nabla |\nabla u|-\nabla |\nabla u_n||\to 0$ $\mea$-a.e..
		This permits to pass to the limit in the left-hand side, indeed $\psi \in L^\infty(\rr)$ and $|\tr \H {u_n}|^2\lesssim_{N}|\H {u_n}|^2. $  For the first term on the right-hand side we have
		\[
		|\la\nabla u_n\Delta u_n -|\nabla u_n|\nabla |\nabla u_n|, \nabla \eta\ra|\le 2\Lip(\eta)\|\psi\|_\infty G^2\in L^1(\mm).
		\]
		Moreover
		\begin{equation}\label{eq:mae technical conv}
			\begin{multlined}[c][.8\textwidth]
				|\la\nabla u_n\Delta u_n -|\nabla u_n|\nabla |\nabla u_n|, \nabla \eta\ra -\la\nabla u\Delta u -|\nabla u|\nabla |\nabla u|, \nabla \eta\ra|\\
				\begin{aligned}
					&\le \Lip(\eta)\left(|\nabla u_n\Delta u_n-\nabla u \Delta u|+ |\nabla |\nabla u_n||\nabla u_n|-\nabla |\nabla u||\nabla u|| \right)\\
					&\le \Lip(\eta)G\left(2|\nabla u-\nabla u_n|+|\Delta u-\Delta u_n|+|\nabla |\nabla u|-\nabla |\nabla u_n||\right)\to 0 ,\quad \alme,
				\end{aligned}
			\end{multlined}
		\end{equation}
		where we used \eqref{eq:total domination} and that $|\nabla |\nabla u_n||\le |\H {u_n}|$ (see \eqref{eq:gradient covariant}). 
		This justifies the passage to the limit in the first term of the right-hand side.  For the second term  observe that by \eqref{eq:psi' bound} and \eqref{eq:total domination}
		\[
		\big|\la \nabla u_n  \Delta u_n-|\nabla u_n||\nabla |\nabla u_n||,\nabla |\nabla u_n| \ra \psi'(|\nabla u_n|) \big|\le (1+c) \cdot G^2, \quad \mea\text{-a.e..}
		\] 
		Moreover, arguing as in \eqref{eq:mae technical conv} we can show 
		\[
		\la \nabla u_n  \Delta u_n -|\nabla u_n|\nabla |\nabla u_n|,|\nabla u_n|\nabla |\nabla u_n| \ra \psi'(|\nabla u_n|)  \overset{\mea\text{-a.e.}}{\rightarrow}  \la \nabla u  \Delta u -|\nabla u|\nabla |\nabla u|,|\nabla u|\nabla |\nabla u| \ra \psi'(|\nabla u|).
		\]
        The above convergence clearly holds for $\mm$-a.e.\ point $x$ such that $\psi'$ is continuous at $|\nabla u|(x)$. On the other hand setting $A\coloneqq \{x \ : \ |\nabla u|(x)\in N\}$ where $N$ is the negligible set of points where $\psi'$ is not defined or continuous,  by locality $|\nabla |\nabla u||=0$ $\mm$-a.e.\ in $A.$ However $|\nabla |\nabla u_n||\to |\nabla |\nabla u||$  $\mm$-a.e.\ and so $|\nabla |\nabla u_n||(x)\to 0$ for $\mm$-a.e.\ $x\in A$. Therefore the convergence above holds pointwise $\mm$-a.e.\ also in $A.$
		Therefore, we can apply the dominated convergence theorem to pass to the limit in the second term of the right-hand side. Finally, the last term can be dealt with by applying dominated convergence one last time, recalling \eqref{eq:total domination}.
	\end{proof}
	From the previous result, we deduce the following technical inequality that combined with Proposition \ref{prop:nightmare} of the previous section will give Theorem \ref{thm:p-calderon}.

	\begin{lemma}[Key estimate]\label{lem:bochner modified}
		For every $\delta>0$ and $p \in(1,\infty)$ there exists a constant $C_2=C_2(p,\delta)>0$ such that the following holds.
		Let $\Xdm$ be an $\rcd(K,N)$ space, $N\in (1,\infty]$ and let $u \in \dom(\Delta)$. For every $\alpha\in \rr$, $\eps>0$ and $\eta \in \LIP_{bs}(\X)$ it holds
		\begin{equation}\label{eq:bochner modified}
			\begin{multlined}[c][.9\textwidth]
				\int \left((1-\delta)|\H u|^2 + \frac{(\tr \H u-\Delta u)^2}{(N-\dim(\X))}\right)|\nabla u|_{M}^{\alpha} \eta^2\d \mm\\
				\le \int_{\{|\nabla u|\le M\}} -2\alpha(p-2) \frac{(\Delta_\infty u)^2 |\nabla u|_{M}^{\alpha}}{(|\nabla u|^2+\eps)^2}\eta^2-2\alpha  |\nabla |\nabla u||^2 \frac{|\nabla u|^2}{|\nabla u|^2+\eps}|\nabla u|_{M}^{\alpha}\eta^2\d \mm\\
				+\int (p-2)^2\frac{(\Delta_\infty u)^2 (|\nabla u|_{M})^{\alpha}}{(|\nabla u|^2+\eps)^2}\eta^2\d \mm\\+ C\int (1+|\alpha|^2)(D_{\eps,p}(u))^2|\nabla u|_{M}^{\alpha}\eta^2+(|\nabla \eta|^2+K^-\eta^2)|\nabla u|^2|\nabla u|_{M}^{\alpha}\d \mm,
			\end{multlined}
		\end{equation}
		where $|\nabla u|_{M}\coloneqq(|\nabla u|\wedge M)^2+\eps$ for $\alpha\ge 0$, and $|\nabla u|_{M}\coloneqq|\nabla u|^2+\eps$ for $\alpha< 0$. Moreover,  if $\alpha<0$ all integrals are over the whole $\X$. Finally, the term containing $\dim(\X)$ is not present if $\dim(\X)=N$ or $N=\infty$.
	\end{lemma}
	\begin{proof}
		Set $f\coloneqq D_{\eps,p}(u)$. Recall that by \eqref{eq:deps inf} it holds
		\begin{equation}\label{eq:pde sviluppata}
			\Delta u=-(p-2)\frac{\Delta_\infty u}{|\nabla u|^2+\eps}+f.
		\end{equation}
		Fix $M>0$, $\alpha>0$ and $\eta \in \LIP_c(\X)$ arbitrary. Define 
		$$\psi(t)\coloneqq \begin{cases}
			((|t| \wedge M)^2+\eps)^\alpha, \quad t\in \rr, & \text{ if $\alpha \ge 0$,}\\
			(t^2+\eps)^\alpha, \quad t\in \rr, & \text{ if $\alpha< 0$.}\\
		\end{cases}  $$
		Since
		\begin{equation}\label{eq:psi' precise}
			\psi'(t)=\begin{cases}\nchi_{\{|t|\le M\}}\frac{2\alpha t\cdot \psi(t)}{t^2+\eps},&\text{ if $\alpha \ge 0$,}\\ \frac{2\alpha t\cdot \psi(t)}{t^2+\eps},&\text{ if $\alpha < 0$,} 
			\end{cases}\qquad\text{ a.e.\ $t\in \rr$,}
		\end{equation}
		we have that $\psi$ satisfies \eqref{eq:psi' bound} for some $c$ depending on $\eps$ and $\alpha$. 
		The strategy is to apply Lemma \ref{lem:bochner}  with $\psi$ and $\eta^2$ and to estimate the right-hand side of \eqref{eq:starting point} by plugging in  \eqref{eq:pde sviluppata}.   
		
		The first term on the right-hand of \eqref{eq:starting point} can be bounded as follows
		\begin{equation*}
			\begin{aligned}[c]\int\left \la\nabla u\Delta u -|\nabla u|\nabla |\nabla u|, \nabla \eta^2\right\ra &\psi(|\nabla u|) \d\mm\\&\le  \begin{multlined}[t][.3\textwidth]
					\int |p-2|\frac{|\Delta_\infty u|}{|\nabla u|^2+\eps} |\nabla u||\nabla \eta^2 |\psi(|\nabla u|)\,\d\mm\\+\int |f| |\nabla u||\nabla \eta^2 |\psi(|\nabla u|)+|\nabla |\nabla u|||\nabla \eta^2||\nabla u|\psi(|\nabla u|)\,\d\mm\end{multlined}\\
				&\le \begin{multlined}[t]\int (f^2 +\delta_1 |\H u|^2)\eta^2\psi(|\nabla u|)\,\d\mm\\+\int4(((p-2)^2+1)\delta_1^{-1}+1)|\nabla \eta|^2|\nabla u|^2\big)\psi(|\nabla u|)\,\d\mm,
				\end{multlined}
			\end{aligned}
		\end{equation*}
		for all $\delta_1>0,$ where in the last step we used the inequalities $|\Delta_\infty(u)|\le |\H u||\nabla u|^2$, $|\nabla |\nabla u||\le |\H u|$ and the Young's inequality on each term.

		As the second term in the right-hand side of \eqref{eq:starting point} is concerned, we estimate only the first half
		\bgroup\allowdisplaybreaks
		\begin{align*}
			\int \la\nabla u\Delta u, \nabla |\nabla u|\ra \psi'(|\nabla u|) \eta^2\d\mm &=\int\Delta u \Delta_\infty u  \frac{\psi'(|\nabla u|)}{|\nabla u|}\eta^2\d\mm \\
			\overset{\eqref{eq:pde sviluppata}}&{=} \int \Delta_\infty u f \frac{\psi'(|\nabla u|)}{|\nabla u|}\eta^2-(p-2)\frac{(\Delta_\infty u)^2\psi'(|\nabla u|)}{|\nabla u|(|\nabla u|^2+\eps)}\eta^2\,\d\mm \\
			&\le \begin{multlined}[t]\int
				\left(\delta_2\frac{(\Delta_\infty u)^2}{|\nabla u|^2+\eps}
				+\frac{f^2(|\nabla u|^2+\eps)}{\delta_2}\right)\frac{|\psi'(|\nabla u|)|}{|\nabla u|}\eta^2\,\d\mm\\-(p-2)\int\frac{(\Delta_\infty u)^2\psi'(|\nabla u|)}{|\nabla u|(|\nabla u|^2+\eps)}\eta^2 \,\d\mm\end{multlined}\\
			&\le \begin{multlined}[t]\int - (p-2)\frac{(\Delta_\infty u)^2\psi'(|\nabla u|)}{|\nabla u|(|\nabla u|^2+\eps)}\eta^2\d \mm\\+\int\left(\delta_2|\H u|^2+ \frac{f^2}{\delta_2}\right)2 |\alpha|\psi(|\nabla u|)|\eta^2 \,\d\mm,\end{multlined}
		\end{align*}
		\egroup
		where in the last step we used \eqref{eq:psi' precise}. 
		Finally, we can estimate the last term containing the Laplacian in \eqref{eq:starting point} by using \eqref{eq:pde sviluppata} as follows
		\begin{align*}
			\int  (\Delta u)^2\psi(|\nabla u)|\eta^2 \d\mm&\le  \int ((p-2)^2+\delta_2)\frac{(\Delta_\infty u)^2 \psi(|\nabla u)|}{(|\nabla u|^2+\eps)^2}\eta^2+\left(1+\delta_2^{-1}\right) f^2\psi|\nabla u|\eta^2\,\d\mm\\
			&\le \int \left((p-2)^2\frac{(\Delta_\infty u)^2}{(|\nabla u|^2+\eps)^2}+\delta_2 |\H u|^2+\left(1+\delta_2^{-1}\right) f^2\right)\psi|\nabla u|\eta^2\,\d\mm
		\end{align*}
		for every $\delta_2>0.$ The conclusion follows putting all together, using \eqref{eq:psi' precise}, choosing $\delta_1=\delta/2$, $\delta_2=\delta/8(1+|\alpha|)^{-1}$ and absorbing the terms containing $|\H u|^2$ into the left-hand side. 
	\end{proof}
	We can finally prove the main result of this section.
	\begin{proof}[Proof of Theorem \ref{thm:p-calderon}]
		\noindent{\textsc{Case 1:}}  $\alpha \ge 0$.  Consider the first two terms in the right-hand side of \eqref{eq:bochner modified}. We have
		\begin{equation}\label{eq:negative term}
			-2\alpha(p-2) \frac{(\Delta_\infty u)^2 |\nabla u|_{M}^{\alpha}}{(|\nabla u|^2+\eps)^2}\eta^2-2\alpha  |\nabla |\nabla u||^2 \frac{|\nabla u|^2}{|\nabla u|^2+\eps}|\nabla u|_{M}^{\alpha}\le 0
		\end{equation}
		which is immediate for $p\ge 2$, while for $1<p<2$  follows recalling that $|\Delta_\infty u |\le |\nabla |\nabla u|||\nabla u|^2.$  Plugging \eqref{eq:negative term} in \eqref{eq:bochner modified} 
		\begin{align*}
			\int &\left((1-\delta)|\H u|^2 + \frac{(\tr \H u-\Delta u)^2}{(N-\dim(\X))}\right)|\nabla u|_{M}^{\alpha} \eta^2\d \mm\\
			&\qquad \qquad\le \begin{multlined}[t][.4\textwidth]\int (p-2)^2\frac{(\Delta_\infty u)^2 (|\nabla u|_{M})^{\alpha}}{(|\nabla u|^2+\eps)^2}\eta^2\d \mm\\\begin{aligned}&+ C\int (1+|\alpha|^2)(D_{\eps,p}(u))^2|\nabla u|_{M}^{\alpha}\eta^2\d \mm\\&+\int(|\nabla \eta|^2+K^-\eta^2)|\nabla u|^2|\nabla u|_{M}^{\alpha}\d \mm,
				\end{aligned}
			\end{multlined}
		\end{align*}
		where $C$ depends only on $p,\delta$ and the term containing $\dim(\X)$ is not present if $\dim(\X)=N$ or $N=\infty$. Using the estimate \eqref{eq:infinity laplacian estimate} (with the choice $\alpha=0$), choosing $\delta>0$  small enough depending on $p$ and $N$ and up to increasing the constant $C$ (depending now also on $N$) we can absorb the first term in the right-hand side into the left-hand side and obtain \eqref{eq:hessian estimate}. Note that \eqref{eq:infinity laplacian estimate} is available with $\alpha=0$, provided $0>p-3-\frac{p-1}{N-1}$, that is $p<3+\frac{2}{N-2}$, which matches the assumption on $p.$

		\noindent{\textsc{Case 2:}} $\alpha < 0$. In this case all the integrals in \eqref{eq:bochner modified} are over the whole $\X$. Hence we can collect  the terms containing $\Delta_\infty u$ and directly plug the estimate  \eqref{eq:infinity laplacian estimate} to obtain
		\begin{equation*}
			\begin{multlined}[c][.9\textwidth]
				(1-\delta-\lambda)\int \left(|\H u|^2 + \frac{(\tr \H u-\Delta u)^2}{(N-\dim(\X))}\right)(|\nabla u|^2+\eps)^{\alpha} \eta^2\d \mm\\
				\le \tilde C\int (1+|\alpha|^2)(D_{\eps,p}(u))^2(|\nabla u|^2+\eps)^{\alpha}\eta^2+(|\nabla \eta|^2+K^-\eta^2)|\nabla u|^2(|\nabla u|^2+\eps)^{\alpha}\d \mm,
			\end{multlined}
		\end{equation*}
		for every $\delta>0$, where  $\tilde C>0$ is a constant depending on $N,p,\alpha$ and $\delta>0$, while $\lambda\in(0,1)$ is  a constant depending on $N,p,\alpha$ (but not on $\delta$). Choosing any $\delta<1-\lambda$ concludes the proof.
		
		\noindent{\textsc{Case 3:}}    $|\H u||\nabla u|^{\alpha }, |\nabla u|^{\alpha +1}\in L^2(\mm)$. In this case, we can take the limit as $M\to +\infty$ in \eqref{eq:bochner modified} using dominated convergence so that all the integrals in \eqref{eq:bochner modified} become over the whole $\X$ and then plug in directly \eqref{eq:infinity laplacian estimate} as in the previous case.
	\end{proof}

	\subsection{Gradient estimates}\label{sec:developed grad est}
	In the next proposition, we prove $L^\infty$-gradient bounds in terms of integral bounds on $D_{\eps,p}$. As mentioned in the introduction, the argument is based on the use of suitable test functions in the Bochner inequality and Moser iteration.  Similar computations in the $\rcd$ setting are present in \cite[Section 4]{Hua-Kell-Xia13} in the case of harmonic functions, but assuming the boundedness of the gradient.
	Since the needed estimates using the Bochner inequality have already been carried out in \textsection \ref{sec:eps regularity}, we are now able to start the iteration directly from the integral inequality given by Theorem \ref{thm:p-calderon}. 
	\begin{prop}\label{prop:lip a priori}
		Let $N\in (2,\infty)$, $q\in(N,\infty]$, $p \in(1,2+\frac{3}{N-2})$  and fix constants $R_0>0,$ $C_0>0$.  Let $\Xdm$ be an $\rcd(K,N)$ space.
		Suppose that $u\in \dom(\Delta)$ satisfies
		\begin{equation}\label{eq:lip a priori ass}
			\big | (|\nabla u|^2+\eps)^{\frac{p-2}{2}} D_{p,\eps} u\big |\le f,\quad \mea\text{-a.e.\ in $B_{R}(x)$},
		\end{equation}
		for some  $\eps \in(0,1)$, $B_{R}(x)\subset \X$, with $R\le R_0$ and $f \in L^{q}(B_{R}(x);\mea)$ with $\fint_{B_{R}(x)} |f|^{q} \d \mm \le C_0$. Then for all $m\ge 1$  it holds
		\begin{equation}\label{eq:apriori lip}
			\| |\nabla u|\|_{L^\infty(B_{r}(x))}\le \frac{CR^{\frac Nm}}{(R-r)^{\frac Nm}} \left(\frac{}{}\bigg(\fint_{B_{R}(x)}  |\nabla u|^{m}\, \d \mm  \bigg)^\frac{1}{m}+1\right), \quad \forall r<R,
		\end{equation}
		where the constant $C\ge 1$ depends only in $p,N,K,q,C_0,m$ and $R_0.$  Moreover, if $|\nabla u|\in L^\infty(B_{R}(x))$ then all the above holds for all $p\in(1,\infty)$.
	\end{prop}
	\begin{proof}
		Fix $B\coloneqq B_R(x)\subset \X$, $R\le R_0$ and $u$ as in the hypotheses. Thanks to the scaling of \eqref{eq:apriori lip} we can assume that $\mea(B_{R}(x))=1.$   We will denote by $C\ge 1$ a constant whose value may change from line to line, but which will ultimately depend only on $N,K,p,q,C_0,m$ and $R_0.$
		
		Let $\eta \in \LIP_c(B)$ be such that $0\le \eta\le 1$. Let also $M>0$ be an arbitrary constant, if $|\nabla u|\in L^\infty(B_{R}(x))$ we take $M=+\infty.$
		For every $\beta\ge 1/2$  define the function 
		$$v\coloneqq (|\nabla u|^2+ \eps)^\frac12((|\nabla u|\wedge M)^2+\eps)^{\beta-\frac12}\in \W(\X).$$ 
		Then, we have
		\begin{align*}
			|\nabla v|&\le|\nabla |\nabla u||((|\nabla u|\wedge M)^2+\eps)^{\beta-\frac12}+\nchi_{\{|\nabla u|\le M\}}(2\beta-1) |\nabla |\nabla u||((|\nabla u|\wedge M)^2+\eps)^{\beta-\frac12}\\
			&\le 2\beta|\H u| ((|\nabla u|\wedge M)^2+\eps)^{\beta-\frac12}, \quad \mea\text{-a.e..} 
		\end{align*}
		By Theorem \ref{thm:p-calderon} applied  with $\alpha=2\beta-1\ge 0$ and using \eqref{eq:lip a priori ass} we have 
		\begin{equation}\label{eq:start moser}
			\begin{split}
				\int |\nabla v|^2\eta^2 \d \mm &\le C_{p,N}  (1+\beta^4)\int f^2 \eta^2 (|\nabla u|^2+\eps)^{2-p}((|\nabla u|\wedge M)^2+\eps)^{2\beta-1}+ (|\nabla \eta|^2+K^-\eta^2) v^{2}\d \mm\\
				&\le   C_{p,N} (1+\beta^4)\int  f^2 \eta^2 v^{\frac{2\beta-p+1}\beta}+ (|\nabla \eta|^2+K^-\eta^2) v^{2}\d \mm,
			\end{split}
		\end{equation}
		where $C_{p,N}>0$ is some constant depending only on $p $ and $N$ (since $\alpha \ge p-2$). In the second inequality above we used  that 
		\begin{equation}
			\begin{split}
				(|\nabla u|^2+\eps)^{2-p}\overset{2\beta \ge 1}&{\le} (|\nabla u|^2+\eps)^{1-\frac{p-1}{2\beta}}((|\nabla u|\wedge M)^2+\eps)^{\frac{p-1}{2\beta}-(p-1)}\\
				&=v^{\frac{2\beta-p+1}\beta}((|\nabla u|\wedge M)^2+\eps)^{1-2\beta}.
			\end{split}
		\end{equation}
		In the case $|\nabla u|\in L^\infty(B_{R}(x))$ we have $|\H u||\nabla u|^{\alpha },|\nabla u|^{\alpha +1}\in L^2(\supp(\eta);\mm)$ and so the above (and the rest of the proof) holds for all $p\in(1,\infty).$
		We estimate the first term as follows
		\begin{equation}\label{eq:moser holder}
			\begin{split}
				\int f^2\eta^2 v^{\frac{2\beta-p+1}\beta}\d \mm &\le 	 \| f\|_{L^{q}(B_R)}^2  \|\eta^2  v^{\frac{2\beta-p+1}\beta}\|_{L^{\frac{q}{q-2}}(\mm)}\le  \| f\|_{L^{q}(B_R)}^2\|\eta^{2\frac{2\beta-p+1}{2\beta}} v^{\frac{2\beta-p+1}\beta}\|_{L^{\frac{q}{q-2}}(\mm)} \\
				&= \| f\|_{L^{q}(B_R)}^2\|\eta  v \|^{2\lambda_\beta}_{L^{\lambda_\beta\frac{2q}{q-2}}(\mm)}
				\le \| f\|_{L^{q}(B_R)}^2\|\eta  v \|^{2\lambda_\beta}_{L^{\frac{2q}{q-2}}(\mm)},
			\end{split}
		\end{equation}
		where $\lambda_{\beta}\coloneqq \frac{2\beta-p+1}{2\beta}\in[0,1).$ In the first line we used H\"older inequality, $\eta\le 1$ and $\lambda_\beta\in[0,1)$, while in the second line, we employed the Jensen inequality (since $\mm(B)=1$). 
		Plugging \eqref{eq:moser holder} in \eqref{eq:start moser}, by the local Sobolev inequality applied in $B_R(x)$ (see Proposition \ref{prop:sobolev rcd}) we get
		\begin{equation}\label{eq:pre interpolation}
            \begin{split}
			\|\eta v\|_{L^{2^*}(\mm)}^{2}&\le C (1+\beta^4) \left(\| f\|_{L^{q}(B_R)}^2\|\eta  v \|^{2\lambda_\beta}_{L^{\frac{2q}{q-2}}(\mm)} + (R^2\Lip(\eta)^2+1)\|v\|_{L^2(\supp(\eta))}^2\right)\\
            &\le C (1+\beta^4) \left(\max\left(\|\eta v\|^{2}_{L^{\frac{2q}{q-2}}(\mm)},1\right) + (R^2\Lip(\eta)^2+1)\|v\|_{L^2(\supp(\eta))}^2\right),
            \end{split}
		\end{equation}		
        where we used the inequality $t^{\lambda}\le \max(t,1)$ for all $t\ge 0$ and $\lambda \in [0,1]$.
		Since by assumption $q>N>2$, we can check that $2<\frac{2q}{q-2}<2^*=\frac{2N}{N-2}$. Therefore, taking $\lambda\coloneqq \frac N{q}\in (0,1)$ we have $\lambda/2^*+(1-\lambda)/2=1/(2q/(q-2))$ and by interpolation
		\[
		\| \eta v\|_{L^{\frac{ 2q}{q-2}}(\mm)}^2\le \|\eta v\|_{L^{2^*}(\mm)}^{2\lambda}\|\eta v\|_{L^{2}(\mm)}^{2(1-\lambda)}\le   \delta^{1/\lambda}\|\eta v\|_{L^{2^*}(\mm)}^2+\delta^{1/(\lambda-1)}\|\eta v\|_{L^{2}(\mm)}^2,
		\]
		for every $\delta>0$. Plugging the above in \eqref{eq:pre interpolation}, we get
		\begin{multline}
		\|\eta v \|_{L^{2^*}(\mm)}^{2}\le C (1+\beta^4) \bigg(\max\left( \delta^{1/\lambda}\|\eta v\|_{L^{2^*}(\mm)}^2+\delta^{1/(\lambda-1)}\|\eta v\|_{L^{2}(\mm)}^2,1\right) \\+ (R^2\Lip(\eta)^2+1)\|v\|_{L^2(\supp(\eta))}^2\bigg).
		\end{multline}
		Set now $\delta\coloneqq \big( 2C (1+\beta^4))^{-\lambda} $. We distinguish the two cases: $\delta^{1/\lambda}\|\eta v\|_{L^{2^*}(\mm)}^2> 1$ and $\delta^{1/\lambda}\|\eta v\|_{L^{2^*}(\mm)}^2\le 1.$
		In the first case, we can absorb the $L^{2^*}$-norm into the left-hand side and obtain
		\[
		\| \eta v \|_{L^{2^*}(\mm)}^2  
		\le C  (1+\beta^4)^{\alpha}  \left( R^2\Lip(\eta)^2+1\right)\| v\|_{L^2(\supp(\eta))}^2,
		\]
		where $\alpha\coloneqq 1+\frac{N}{q-N}$ and possibly increasing the constant $C$. In the second case, we directly have
		\[
		\| \eta v \|_{L^{2^*}(\mm)}^2  \le C  (1+\beta^4).
		\]
		In both cases, we thus have
		\begin{equation}\label{eq:final moser}
			\| \eta v\|_{L^{2^*}(\mm)}^2  
			\le C  (1+\beta^4)^{\alpha}\max \left( \left( R^2\Lip(\eta)^2+1\right)\|v\|_{L^2(\supp(\eta))}^2,1\right).
		\end{equation}
		We now proceed with the Moser iteration. For all $k \in \nn$ and $r,s \in[R/2,R]$ with $r<s$, we choose a function $\eta$ as before and satisfying
		\[
		\eta=1 \text{ in $B_{  r+\frac{(s-r)}{2^{k+1}}}(x)$}, \quad \eta \in \supp(\eta)\subset \bar B_{ r+\frac{(s-r)}{2^{k}}}(x) , \quad \Lip(\eta)\le  \frac{2^{k+1}}{s-r}.
		\]
		We suppose for now $m\ge 2$ and set $\gamma_0\coloneqq m/2$  and  $\gamma_k\coloneqq N \gamma_{k-1}/(N-2)>\gamma_{k-1}$ for all $k\in \nn$.  Finally, we set $B_k\coloneqq B_{ r+\frac{(s-r)}{2^{k+1}}}(x)$ for $k \in \nn$ and $B_0\coloneqq B_s(x)$. Choosing $\beta=\gamma_{k-1}/2$ (note that $2\beta\ge 1$) in the inequality \eqref{eq:final moser} and recalling the expression of $v$ gives
		\begin{align*}
			{\bf a}_k^{\gamma_{k-1}}
			&\le 2C  (1+\gamma_{k-1}^4)^{\alpha} \max\left(\frac{2^{2(k+1)}R^2}{(s-r)^2}{\bf a}_{k-1}^{\gamma_{k-1}},1\right)\\
			&\le \frac{8CR^2}{(s-r)^2} \left[4\left(\frac{N}{N-2}\right)^{4\alpha}\right]^{k}  (1+\gamma_0^4)^{\alpha} \max\left({\bf a}_{k-1}^{\gamma_{k-1}},1\right), \quad \forall k \in \nn,
		\end{align*}
		where 
		\[
		{\bf a}_k\coloneqq \left(\int_{B_k}  (|\nabla u|^2+ \eps)((|\nabla u|\wedge M)^2+\eps)^{\gamma_{k}-1} \, \d \mm   \right)^\frac 1{\gamma_k}, \quad k\in \nn\cup \{0\}.
		\]
		In the second line we used that $\frac{2^{2(k+1)}R^2}{{(s-r)^2}}\ge 1$ and that $ N/(N-2)>1$.
		Note also that we used that $|\nabla u|^2+\eps\ge (|\nabla u|\wedge M)^2+ \eps$ $\alme$ on the left-hand side.
		
		Raising to the $1/\gamma_{k-1}$,
		we reach
		\begin{equation}\label{eq:end moser}
			{\bf a}_{k}\le \frac{A_k R^\frac{2}{\gamma_{k-1}}}{(s-r)^\frac{2}{\gamma_{k-1}}}\max\left({\bf a}_{k-1},1\right),\quad \forall k \in \nn,
		\end{equation}
		where
		\[
		A_k\coloneqq \left\lbrace 8C   \left[4\left(\frac{N}{N-2}\right)^{4\alpha}\right]^{k} (1+\gamma_0^4)^{\alpha} \right\rbrace^{\frac{1}{\gamma_{k-1}}}\ge 1, \quad \forall k\in \nn.
		\]
		Observe that
		$\sum_{k\in \nn} \frac{1}{\gamma_{k-1}}=\gamma_0^{-1}\frac N2$ and that $\sum_{k\in \nn} \frac{k}{\gamma_{k-1}}<+\infty.$ Therefore,
		\[
		\prod_{k=1}^\infty A_k \le C,
		\]
		up to increasing the constant $C.$
		Iterating \eqref{eq:end moser}, we obtain 
		\[
		\|(|\nabla u|\wedge M)^2+\eps\|_{L^{\gamma_k}(B_k)} \le {\bf a}_{k}\le ( {{\bf a}_{0}}+1) \prod_{k=1}^k \frac{A_k R^\frac{2}{\gamma_{k-1}}}{(s-r)^\frac{2}{\gamma_{k-1}}} ,
		\]
		where we used again that $|\nabla u|\ge |\nabla u|\wedge M$ $\alme.$ Letting $k\to + \infty$, we reach
		\begin{equation}\label{eq:almost the end}
			\begin{split}
				\| (|\nabla u|\wedge M)^2 \|_{L^{\infty}(B_{r}(x))} &\le \frac{CR^{\frac{2N}{m}}}{(s-r)^{\frac{2N}{m}}} ( {\bf a}_{0}+1) \le \frac{CR^{\frac{2N}{m}}}{(s-r)^{\frac{2N}{m}}}  \left(\||\nabla u|^2+\eps\|_{L^{\frac m2 }(B_s(x))}+1\right).
			\end{split}
		\end{equation}
		Sending $M\to +\infty$, recalling that $\mm(B_r(x))\le\mm(B_R(x)) \le 1$ and taking the square root we obtain
		\begin{equation}\label{eq:final m=2}
			\| |\nabla u| \|_{L^{\infty}(B_{r}(x))} \le 
			\frac{C R^\frac Nm}{(s-r)^\frac Nm}  \left( \||\nabla u|\|_{L^m(B_s(x)}+1\right),
		\end{equation}
		up to further increasing $C.$
		Taking  $s=R$ this concludes the proof of \eqref{eq:apriori lip} for $m\ge 2$. The case $m<2$ can be obtained from the case $m=2$ with the following standard argument (cf.\ with \cite[p.\ 75]{HanLinbook}):
		\begin{align*}
			\| |\nabla u| \|_{L^{\infty}(B_{r}(x))}  \overset{\eqref{eq:final m=2}}&{\le} 
			\frac{C R^\frac N2}{(s-r)^\frac N2}  \left(\| |\nabla u| \|_{L^{\infty}(B_{s}(x))}^{1-\frac{m}{2}} \||\nabla u|\|_{L^m(B_s(x)}^\frac{m}{2}+1\right)\\
			&\le \frac12\| |\nabla u| \|_{L^{\infty}(B_{s}(x))} +  \frac{C R^\frac Nm}{(s-r)^\frac Nm} \||\nabla u|\|_{L^m(B_s(x)}+ \frac{C R^\frac N2}{(s-r)^\frac N2}\\
			&\le \frac12\| |\nabla u| \|_{L^{\infty}(B_{s}(x))} +  \frac{C R^\frac Nm}{(s-r)^\frac Nm} \left(\||\nabla u|\|_{L^m(B_R(x)}+1\right),
		\end{align*}
		where in the second line we used Young's inequality and in the third that $\frac{s-r}{R}\le 1$.
		Hence, the conclusion follows using e.g.\ \cite[Lemma 4.3]{HanLinbook}.
	\end{proof}
 
		\begin{remark}
			In the case $q=\infty$, the inequality \eqref{eq:apriori lip} can be improved to
			\begin{equation}\label{eq:apriori lip improved}
				\| |\nabla u|\|_{L^\infty(B_{r}(x))}\le \frac{CR^{\frac Nm}}{(s-r)^{\frac Nm}} \left(\frac{}{}\bigg(\fint_{B_{R}(x)}  |\nabla u|^{m}\, \d \mm  \bigg)^\frac{1}{m}+\bigg(\fint_{B_{R}(x)}  |\nabla u|^{m}\, \d \mm  \bigg)^\frac{\alpha}{m}+\eps\right),
			\end{equation}
			for some $\alpha\in(0,1)$ depending only on $p,N,m$, provided $m>(p-1)(N+2).$
			
			Indeed plugging \eqref{eq:moser holder} in \eqref{eq:start moser} directly gives the following version of \eqref{eq:final moser}
    \begin{align*}
    \| \eta v\|_{L^{2^*}(\mm)}^2  &\le C (1+\beta^4) \left(\| f\|_{L^{q}(B_R)}^2\|\eta  v \|^{2\lambda_\beta}_{L^{\frac{2q}{q-2}}(\mm)} + (R^2\Lip(\eta)^2+1)\|v\|_{L^2(\supp(\eta))}^2\right)\\
			&\le C  (1+\beta^4) (R^2\Lip(\eta)^2+1)\max  \left(\|  v \|^{2\lambda_\beta}_{L^{\frac{2q}{q-2}}(\supp(\eta))} ,\|v\|_{L^2(\supp(\eta))}^2\right).
    \end{align*}
   Recalling that $\lambda_{\beta}= \frac{2\beta-p+1}{2\beta}$ and choosing $\beta = \gamma_{k-1}/2$, we get a new version of \eqref{eq:end moser} which reads as
			\begin{equation}\label{eq:end moser improved}
				{\bf a}_{k}\le \frac{A_k R^\frac{2}{\gamma_{k-1}}}{(s-r)^\frac{2}{\gamma_{k-1}}}\max\left({\bf a}_{k-1},({\bf a}_{k-1})^{1-\frac{p-1}{\gamma_{k-1}}}\right),\quad \forall k \in \nn.
			\end{equation}
			Now, if ${\bf a}_{k}\le {\bf a}_{0}$ for infinitely many $k$'s, arguing as in \eqref{eq:almost the end} and \eqref{eq:final m=2} we would obtain \eqref{eq:apriori lip improved}. So we can assume that ${\bf a}_{k_0}\le {\bf a}_{0}$ for some $k_0$ and ${\bf a}_{k}> {\bf a}_{0}$ for all $k>k_0.$ We can also assume ${\bf a}_{0}\le 1,$ otherwise \eqref{eq:apriori lip improved} follows directly from \eqref{eq:apriori lip}. By \eqref{eq:end moser improved}, we have 
			\begin{equation}\label{eq:end moser improved k big}
				{\bf a}_{k}\le \frac{A_k R^\frac{2}{\gamma_{k-1}}}{(s-r)^\frac{2}{\gamma_{k-1}}}\frac{{\bf a}_{k-1}}{{\bf a}_0^{\frac{p-1}{\gamma_{k-1}}}},\quad \forall k\ge k_0+2.
			\end{equation}
			Iterating \eqref{eq:end moser improved k big} for all $k\ge k_0+2$, \eqref{eq:end moser improved} for $k=k_0+1$, and finally that ${\bf a}_{k_0}\le {\bf a}_0$, we reach
			\[
			{\bf a}_{k}\le \frac{CR^{\frac{2N}{m}}}{(s-r)^{\frac{2N}{m}}}  \frac{{\bf a}_{0}^{1-\frac{p-1}{\gamma_{k-1}}}}{{\bf a}_{0}^{\sum_{k\in \nn} \frac{p-1}{\gamma_{k-1}}}}
			\le \frac{CR^{\frac{2N}{m}}}{(s-r)^{\frac{2N}{m}}} {\bf a}_{0}^{1-\frac{p-1}{m} (N+2)},
			\quad \forall k\ge k_0+2,
			\]
			having used that  $\sum_{k\in \nn} \frac{1}{\gamma_{k-1}}=\gamma_0^{-1}\frac N2$ and that $\gamma_0=m/2.$ Assuming that $m>(p-1)(N+2)$ and proceeding again as in \eqref{eq:almost the end} and \eqref{eq:final m=2}, we reach \eqref{eq:apriori lip improved}. \fr
	\end{remark}
	\section{Existence of regular solutions of \texorpdfstring{$\Delta_{p,\eps}(u)=f$}{\textDelta\_(p,\textepsilon)=f}}
	\label{sec:existence_regularised}
	To show the existence of regular solutions for $\Delta_{p,\eps}$, we will first show the existence for an auxiliary \textit{linear operator} $\LL_{v,\eps}$ (defined below). This operator is obtained by freezing the nonlinear part of the developed $(p,\eps)$-Laplacian $D_{\eps,p}$ introduced at the beginning of \textsection \ref{sec:eps regularity}.
	
	\begin{definition}\label{def:auxiliary}
		Let $\Xdm$ be an $\RCD(K,\infty)$ space, $v \in L^0(T\X)$ and $\eps> 0$. The map $\LL_{v,\eps}: \dom (\Delta) \to L^2(\mm)$ is defined by
		\[
		\LL_{v,\eps}(u)\coloneqq \Delta u+(p-2) \frac{\H u( v, v)}{| v|^2+\eps} \in L^2(\mm),
		\]
		which makes sense because $\dom(\Delta)\subset W^{2,2}(\X).$
	\end{definition}
	The operators $\LL_{v,\eps}$ and  $D_{\eps,p}(u)$ are connected by the following immediate identity
	\begin{equation}\label{eq:L=D}
		\LL_{\nabla u,\eps}(u)=D_{\eps,p}(u), \quad \forall u\in \dom (\Delta).
	\end{equation}
	In \textsection \ref{sec:existence pt1} we will apply a fixed point argument to obtain an existence result for $\LL_{v,\eps}$ together with  $L^2$-Laplacian estimates on the solution. Building on top of this and using a second fixed point argument we will establish in \textsection \ref{sec:existence pt2} the main existence result for $\Delta_{p,\eps}(u)=f$, together with uniform estimates independent of $\eps$  based on the technical results in \textsection \ref{sec:eps regularity}.

	\subsection{Existence for an auxiliary operator}\label{sec:existence pt1}
	
	Before stating the existence result, we introduce an important space of functions that will be frequently used here and also in \textsection \ref{sec:existence pt2}.
	Given  $\Xdm$ an $\rcd(K,\infty)$  space with $\mea(\X)<+\infty$, we define
	\begin{equation}\label{eq:def dom0}
		\dom_0(\Delta)\coloneqq \left \{u\in \dom(\Delta)  \ \ : \ \int_\X u\, \d \mm=0 \right \}
	\end{equation}
	and we endow it with the norm $\|\Delta u\|_{L^2(\mm)}$. In the following proposition, we collect a few basic but important properties of this space.
	\begin{prop}\label{lem:norms}
		Let $\Xdm$ be a bounded $\rcd(K,\infty)$ space. Then $\left(\dom_0(\Delta),\|\Delta(.)\|_{L^2(\mm)}\right)$ is a Hilbert space. Moreover, the inclusion $\left(\dom_0(\Delta),\|\Delta(.)\|_{L^2(\mm)}\right)\hookrightarrow \left(\dom_0(\Delta),\|\cdot\|_{\W(\X)}\right)$ is compact.
	\end{prop}
	\begin{proof}
		$\dom_0(\Delta)$ is a vector space thanks to the linearity of the Laplacian and one can readily check that $\|\Delta(.)\|_{L^2(\mm)}$ is a norm which satisfies the parallelogram identity. Since $\lambda_1(\X)>0$ (recall Theorem \ref{thm:pPoincare}) and using \eqref{eq:Laplacian poincare} we have $\lambda_1\|u\|_{W^{1,2}(\X)}^2\le 2\|\Delta u\|_{L^2(\mm)}^2$ for all $u \in \dom_0(\Delta)$, from which the completeness easily follows recalling \eqref{eq:closure laplacian}.
		It remains to check the compactness of the inclusion in the statement. Consider a sequence $u_n\in \dom_0(\Delta)$ with $\sup_n \|\Delta(u_n)\|_{L^2(\mm)}\le C<+\infty$.  From $\lambda_1\|u\|_{W^{1,2}(\X)}^2\le \|\Delta u\|_{L^2(\mm)}^2$, it follows that the sequence is also bounded in $\W(\X)$. From the compactness of the embedding $W^{1,2}(\X)\hookrightarrow L^2(\mm)$ (see \cite[Proposition 6.7]{GMS15}), up to a subsequence, $u_n\to u \in L^2(\mm)$. Moreover, since $\sup_n\|u_n\|_{W^{1,2}(\X)}<+\infty$, by lower-semicontinuity we have  $u \in \W(\X)$  (see e.g.\ \cite[Proposition 2.1.19]{GP20}). Finally integrating by parts
		\begin{align*}
			\int |\nabla(u_n-u_m)|^2 \d \mm &= -\int \Delta(u_n-u_m)(u_n-u_m)\d \mm \le\|\Delta(u_n-u_m)\|_{L^2(\mm)}\|u_n-u_m\|_{L^2(\mm)}\\&\le 2C\|u_n-u_m\|_{L^2(\mm)},
		\end{align*}
		from which follows that $|\nabla u-\nabla u_n|\to 0$ in $L^2(\mm)$. Finally $\Delta u_n \rightharpoonup g$  in $L^2(\mm)$ for some $g \in L^2(\mm)$ and since $u_n \to u$ in $\W(\X)$, this shows that $u \in \dom_0(\Delta)$ (recall \eqref{eq:closure laplacian}). The proof is now concluded.
	\end{proof}
	
	The goal of this part is to prove the following preliminary existence result for  $\LL_{v,\eps}$. 
	\begin{prop}[Existence of solutions for $\LL_{v,\eps}$]\label{prop:step 1}
		Let $\Xdm$ be a bounded $\rcd(K,N)$ space, with  $N\in[2,\infty]$ and fix $p \in \mathcal{RI}_{\X}$ where $ \mathcal{RI}_{\X}\subset(1,\infty)$ is given by \eqref{eq:regularity interval}. Then for all 
		$f \in L^2(\mm)$, $v \in L^0(T\X)$ and $\eps>0$ there exists (unique) $u \in \dom_0(\Delta)$ that solves
		\begin{equation}\label{eq:step 1}
			\LL_{v,\eps}(u)=f-\theta_v^{-1}\fint_\X\theta_v(f-\LL_{v,\eps}(u))\d \mm, 
		\end{equation}
		where $\theta_v\in L^\infty(\mm)$ is defined as
		\begin{equation}\label{eq:theta}
			\theta_v\coloneqq	\begin{cases}
				1,& \text{ if  $p\in(1,2)$},\\
				\frac{N+g}{N+g^2+2g},& \text{otherwise},
			\end{cases}
		\end{equation}
		where we denote $g\coloneqq(p-2)\frac{|v|^2}{|v|^2+\eps}$. 
		Moreover, there exists a constant $C$ depending only on $p$, $N$ and  $\lambda_1(\X)K^-$ such that
		\begin{equation}\label{eq:estimates solution}
			\|\Delta u\|_{L^2(\mm)}\le C \|f\|_{L^2(\mm)}.
		\end{equation}
	\end{prop}
	By \eqref{eq:theta} we have that $\theta_v\ge c>0$ for some constant $c$ depending only on $p$ and $N$, hence \eqref{eq:step 1} makes sense. 
	The key feature of the estimate \eqref{eq:estimates solution} is that the constant $C$ \emph{does not} depend on $v$, which will be essential for the arguments in the next subsection.

	Before going into the proof of Proposition \ref{prop:step 1}, we outline the core idea in the case $K=0$ and $N=+\infty$. We borrow some of the techniques from \cite[Section 1.2]{cordesbook}, which date back to the theory of Cordes conditions in $\rr^n$ for equations with measurable coefficients in non-divergence form (see  also \cite{Cordes56,TalentiSopra,Campanato89,Campanatonear}). 
	To solve the equation $\LL_{v,\eps}(u)=f$ it is sufficient to find a fixed point of the map $T: \dom(\Delta)\to \dom(\Delta)$, where $T(w)\coloneqq u$ is the solution to
	\[
	\Delta u=\Delta w-\LL_{v,\eps}(w)+f.
	\]
	The key observation is that 
	\begin{equation}\label{eq:near laplacian example}
		|\Delta w-\LL_{v,\eps}(w))|\le  |p-2| |\H w|,\quad \alme.
	\end{equation}
	Indeed combining \eqref{eq:near laplacian example} with inequality \eqref{eq:talenti} we have
	\begin{align*}
		\int (\Delta T(w_1-w_2))^2&=\int ((\Delta -\LL_{v,\eps})(w_1-w_2))^2 \le (p-2)^2 \int |\H {w_1-w_2}|^2\\\overset{\eqref{eq:talenti}}&{\le }(p-2)^2 \int (\Delta (w_1-w_2))^2,
	\end{align*}
	which shows that $T$ is a contraction  with respect to $\|\Delta (\cdot) \|_{L^2}$  provided $p \in(1,3)$. In fact, we will see that if $N<\infty$ this estimate can be further refined to include a slightly larger range of the parameter $p$ (at least if $K=0$ or small enough). The auxiliary function $\theta_v$ defined in \eqref{eq:theta} will be needed exactly for this purpose.
	
	We start by proving a refined version of \eqref{eq:near laplacian example}.	
	\begin{lemma}[$\LL_{v,\eps}$ is close to the Laplacian]\label{lem:near laplacian}
		Let $\Xdm$ be an $\rcd(K,N)$ space with $N\in[2,\infty)$ and let $p \in[2,\infty]$. Then
		\begin{equation}\label{eq:near laplacian}
			\begin{multlined}[c][.9\textwidth]
				\left | \Delta u -\frac{N+g}{N+g^2+2g}\LL_{v,\eps}(u) \right |^2 \\\le \frac{(p-2)^2(N-1)}{N+2(p-2)+(p-2)^2}\left[|\H u|^2  +\frac{\left(\tr \H u - \Delta u\right)^2 }{N-\dim(\X)}\right], \quad \mea\text{-a.e.,}
			\end{multlined}
		\end{equation}
		for every $u \in D(\Delta)$, $v \in L^0(T\X),$ where $g\coloneqq(p-2)\frac{|v|^2}{|v|^2+\eps}$ and the last term is not present if $N=\infty$ or $\dim(\X)=N.$
	\end{lemma}
		To prove Lemma \ref{lem:near laplacian} we will need the following elementary estimate, which proof is omitted for brevity.
	\begin{lemma}\label{lem:elementary estimate}
		Fix $N\ge2.$ Then for every $t\ge 0$ it holds
		\begin{equation}\label{eq:elementary estimate}
			\left [\frac{t^2+t}{N+t^2+2t}A-\frac{tN+t^2}{N+t^2+2t}B \right]^2\le \frac{t^2(N-1)}{N+2t+t^2}\left[\frac{(A-B)^2}{N-1}+B^2\right], \quad \forall A,B \in \rr.
		\end{equation}
	\end{lemma}

	\begin{proof}[Proof of Lemma \ref{lem:near laplacian}]
		Set $n\coloneqq \dim(\X)$ (if $N<\infty$).
		Note that \eqref{eq:near laplacian} is trivially satisfied when $|v|=0$, since the left-hand side vanishes and the right-hand side is non-negative. Hence, we can work in the set $\{|v|>0\}$.
		
		Set $g\coloneqq(p-2)\frac{|v|^2}{|v|^2+\eps} \in L^\infty(\mm)$. Rewriting $\LL_{v,\eps}(u)$ as $\LL_{v,\eps}(u)=\Delta u-g\frac{\H u(v,v)}{|v|^2}$, we find
		\begin{equation*}
			\left | \Delta u -\frac{N+g}{N+g^2+2g}\LL_{v,\eps}(u) \right |^2=\left |\frac{g^2+g}{N+g^2+2g}\Delta u -\frac{gN+g^2}{N+g^2+2g}\frac{\H u(v,v)}{|v|^2} \right |^2, \quad \mea\text{-a.e..}
		\end{equation*}
		Lemma \ref{lem:elementary estimate} yields
		\begin{align*}
			\left | \Delta u -\frac{N+g}{N+g^2+2g}\LL_{v,\eps}(u) \right |^2\overset{\eqref{eq:elementary estimate}}&{\le} \alpha \left[\frac{1}{N-1}\left(\Delta u - \frac{\H u(v,v)}{|v|^2}\right)^2+\frac{\H u(v,v)^2}{|v|^4}\right]\\
			&\le \begin{multlined}[t]
				\alpha \left[\frac{1}{n-1}\left(\tr \H u - \frac{\H u(v,v)}{|v|^2}\right)^2\right.\\\left.\vphantom{\left(\tr \H u - \frac{\H u(v,v)}{|v|^2}\right)^2}+\frac{1}{N-n}\left(\tr \H u - \Delta u\right)^2+\frac{\H u(v,v)^2}{|v|^4}\right]\end{multlined}\\
			&\le \alpha \left[|\H u|^2  +\frac{1}{N-n}\left(\tr \H u - \Delta u\right)^2 \right],
		\end{align*}
		where  $\alpha\coloneqq \frac{(p-2)^2(N-1)}{N+2(p-2)+(p-2)^2} $ and in the second step we have used Young's inequality. In the first line we also used that $g\le (p-2)$ and the function $\frac{t^2}{N+2t+t^2}$ is increasing on $[0,\infty)$.
	\end{proof}

	We are now ready to prove the first existence result.

	\begin{proof}[Proof of Proposition \ref{prop:step 1}]
		The proof consists in using a fixed-point argument.
		
		Recall the space $\dom_0(\Delta)=\dom(\Delta)\cap \{u \ : \int u \d \mm=0\}$ defined in \eqref{eq:def dom0}. Fix also $v\in L^0(T\X)$, a constant $\eps>0$ and a function $f \in L^2(\mm)$. Define the map $$T_{f,v}: \dom_0(\Delta) \to \dom_0(\Delta)$$ given by $T_{f,v}(w)\coloneqq u,$ where $u$ is the solution to the following equation
		\begin{equation}\label{eq:regularised eps integro-diff}
			\begin{cases}
				&\Delta u=\Delta w+\theta_v f-\theta_v \LL_{v,\eps}(w)- \fint \theta_v (f-\LL_{v,\eps}(w)\d\mm\, \in L^2(\mm),\\
				&\int u \, \d \mm=0,
			\end{cases}
		\end{equation}
		where $\theta_v $ is as in \eqref{eq:theta}. Note that $T_{f,v}$ depends also on $\eps$, even if not expressed in our notation.
		It can be easily verified that $0<c\le \theta_v  \le C<+\infty$, for some constants $c,C$ depending only on $N$.
		$T_{f,v}$ is well defined. Indeed for every $h \in L^2(\mm)$ with zero mean there exists a unique function $u \in\dom(\Delta)$ such that $\Delta u=h$ and $\int u \, \d \mm=0$ (see Proposition \ref{prop:p-existence}).
		
		\smallskip
		
		\noindent\textsc{Step 1: Fixed point.} We claim $T_{f,v}$  is a contraction with respect to the norm $\|\Delta (\,\cdot\,)\|_{L^2(\mm)}$. Set
		$$\alpha_p\coloneqq \begin{cases} \frac{(p-2)^2(N-1)}{N+2(p-2)+(p-2)^2},& \text{if $p\ge 2$, $N<\infty$,}\\(p-2)^2, & \text{if $p\in(1,2)$ or $N=\infty$.}\end{cases}$$
		For every $w_1,w_2\in \dom_0(\Delta)$, since $\int (h-\fint h)^2\d \mm\le \int h^2\d \mm$ for all $h \in L^2(\mm)$ and $\int \Delta w \,\d \mm=0$ (recall \eqref{eq:zero mean laplacian}) for all $w \in \dom(\Delta)$, we have
		\begin{align*}
			\|\Delta T_{f,v}(w_1-w_2)\|_{L^2(\mm)}^2 &\le  \|\Delta(w_1-w_2)-\theta_v  \LL_{v,\eps}(w_1-w_2)\|_{L^2(\mm)}^2\\
			&\le \alpha_p \int |\H {w_1-w_2}|^2+\frac{(\tr \H{w_1-w_2}-\Delta(w_1-w_2))^2}{N-\dim(\X)}  \d \mm\\
			\overset{\eqref{eq:talenti}}&{\le} \alpha_p  (1+K^-\lambda_1)\|\Delta (w_1-w_2)\|_{L^2(\mm)}^2,
		\end{align*}
		where in the second inequality we used \eqref{eq:near laplacian} whenever $p\ge 2$ and $N<\infty$, ignoring the term with $\dim(\X)$ if $\dim(\X)=N$ or $N=\infty$. In particular, $T_{f,v}$ is a contraction if $\alpha_p  (1+K^-\lambda_1)<1$. This is where we use the fact that $p \in \mathcal{RI}_{\X}$. For convenience of notation, we set  $\delta_{\X}\coloneqq \frac{\lambda_1K^-}{1+\lambda_1K^-}\in[0,1)$. 
		Consider first the case $p\in(1,2)$. We have a contraction provided
		\[
		2-\sqrt{1-\delta_{\X}}<p\le 2,
		\]
		which is ensured by $p\in \mathcal{RI}_{\X}$.
		We now turn to the case $p\ge 2$. If $N=\infty$ then $\alpha_p  (1+K^-\lambda_1)=(p-2)^2  (1+K^-\lambda_1)<1$ is equivalent to $p< 2+\sqrt{1-\delta_{\X}}$. Conversely, observe that the function $ f(t)\coloneqq \frac{t^2(N-1)}{N+2t+t^2}$ is strictly increasing for $t\ge 0$, and, thus, invertible. Hence, we have that $\alpha_p  (1+K^-\lambda_1) = f(p-2)(1+K^-\lambda_1)<1$ if and only if $0\le p\le \overline{p}$, where $f(\overline{p}-2)=(1+K^{-}\lambda_1)^{-1}$. A direct computation gives
		\[
		\overline{p}=2+\sqrt{1-\delta_{\X}} \frac{\sqrt{1-\delta_{\X}}+\sqrt{(N-1)^2+\delta_{\X}(N-1)}}{(N-2+\delta_{\X})}\ge 2+\sqrt{1-\delta_{\X}} \frac{N-\delta_{\X}}{N-2+\delta_{\X}},
		\]
		where $\overline{p}\coloneqq +\infty$ if $N=2$ and $\delta_{\X}=0$. Also in this case we have that $T_{f,v}$ is a contraction for every $p \in \mathcal{RI}_\X$.
		
		\smallskip
		
		\noindent\textsc{Step 2: equation for the fixed point.} From \eqref{eq:regularised eps integro-diff} the  unique fixed point $u \in \dom_0(\Delta)$ of the map $T_{f,v}$ satisfies 
		$$\theta_v  \LL_{v,\eps}(u)=\theta_v  f- \fint_\X \theta_v (f-\LL_{v,\eps}(w))\d\mm $$
		and, since as observed above $\theta_v \ge c>0$ $\mea$-a.e., \eqref{eq:step 1} follows.
		
		\smallskip
		
		\noindent\textsc{Step 3: Estimate.} Let $u \in \dom_0(\Delta)$ be again the fixed point  of the map $T_{f,v}$. In particular
		\[
		\Delta u=\Delta u-\theta_v\LL_{v,\eps}(u)+\theta_v\LL_{v,\eps}(u)=\Delta u+\theta_v  f-\theta_v\LL_{v,\eps}(u)- \fint \theta_v (f-\LL_{v,\eps}(w))\d\mm
		\]
		Hence arguing as above using $\int (h-\fint h)^2\d \mm\le \int h^2\d \mm$ and $\int \Delta u\d \mm=0$ we obtain
		\begin{align*}
			\|\Delta u\|_{L^2(\mm)}&\le \left\|\Delta u-\theta_v  \LL_{v,\eps}(u)+ \theta_v  f\right \|_{L^2(\mm)}\\
			\overset{\eqref{eq:near laplacian}, \eqref{eq:talenti}}&{\le } \sqrt{\alpha_p}\sqrt{1+K^-\lambda_1} \|\Delta u\|_{L^2(\mm)}+ C \left\|f\right\|_{L^2(\mm)},
		\end{align*}
		were we used that $\theta_v\le C$. Since as observed in Step 1 we have  $\alpha_p(1+K^-\lambda_1)<1$ if $p\in \mathcal{RI}_\X$, this proves \eqref{eq:estimates solution}.
	\end{proof}
	
	\subsection{Main existence result}	\label{sec:existence pt2}
	The goal of this part is to prove the following.
	\begin{theorem}[Existence of regular solutions to the regularised equation]\label{thm:regularity epsilon}
		Let $\Xdm$ be an $\rcd(K,N)$ space, with  $N\in[2,\infty]$, $\diam(\X)\le D<\infty $. Let $p \in \mathcal{RI}_\X$, where $ \mathcal{RI}_{\X}\subset(1,\infty)$ is given by \eqref{eq:regularity interval}. Then, for every $f \in L^q(\mm)$, $q\coloneqq \max(2,\frac{2}{p-1})$, with zero mean there exists a function $u \in  \dom(\Delta)\cap \dom(\Delta_{p,\eps})$ such that
		\begin{equation}\label{eq:ep laplacian}
			D_{\eps,p}u(|\nabla u|^2+\eps)^\frac{p-2}{2}=\Delta_{p,\eps}(u)=f.
		\end{equation}
		Moreover, $v\coloneqq(|\nabla u|^2+\eps )^{\frac{p-2}{2}}\nabla u \in H^{1,2}_C(T\X)$ and for every $R_0>0$ and $B_R(x)\subset \X$ with $R\le R_0$  the following hold:	
		\begin{enumerate}[label=\roman*)]
			\item\label{item:covariant regularized} \hfill$\displaystyle\int_{B_{R/2}(x)} |v|^2+|\nabla v |^2 \d \mm\le C_1 \int_{B_{R}(x)} f^2 \d \mm+ C_2\mea(B_{R}(x))^{-1}  \left(\int_{B_R(x)} |\nabla u|^{p-1}\d \mm\right)^2,$\hfill\refstepcounter{equation}{\normalfont(\theequation)}\label{eq:covariant regularized}\\
			where $C_1,C_2$ are constants depending only on  $K,N,p$,$R_0$ if $N<\infty$,  while in the case $N=\infty$ all integrals are on the whole $\X$ and  $C_2=K^-C_1$ with $C_1$ a constant depending on $D,p$ and $K.$
			\item\label{item:apriori lip regularized} if $N<\infty$ and $f \in L^{q}(B_{R}(x);\mea)$ for some $q>N$, with $\fint_{B_{R}(x)} f^{q}\d \mm\le C_0$, then
			\begin{equation}\label{eq:apriori lip regularized}
				\| |\nabla u|\|_{L^\infty(B_{r}(x))}\le \frac{CR^{\frac Nm}}{(R-r)^{\frac Nm}} \left(\frac{}{}\bigg(\fint_{B_{R}(x)}  |\nabla u|^{m}\, \d \mm  \bigg)^\frac{1}{m}+1\right), \quad \forall r<R,
			\end{equation}
			where the constant $C\ge 1$ depends only in $p,N,K,q,C_0,m$ and $R_0.$  
		\end{enumerate}
	\end{theorem}
	The existence part of Theorem \ref{thm:regularity epsilon} is based on the existence result given in the previous section (Proposition \ref{prop:step 1}) and on an application of the Schauder fixed point theorem.
	\begin{theorem}[{Schauder fixed point theorem, see e.g.\ \cite[Corollary 11.2]{GilbargTrudinger}}]
		Let $M$ be a Banach space and $T:M\to M$ be a continuous map such that $T(M)$ is relatively compact in $M$. Then, $T$ has a fixed point.
	\end{theorem}
	The regularity estimates stated in $\ref{item:covariant regularized}$ and $\ref{item:apriori lip regularized}$ will be instead deduced by the uniform a priori estimates previously obtained in Section \ref{sec:eps regularity}.

	\begin{proof}[Proof of Theorem \ref{thm:regularity epsilon}]
		By the scaling property of the statement we can assume that $\mm(\X)=1$.
		Fix $p \in \mathcal{RI}_\X$ and $f\in L^q(\mm)$ as in the statement.
		We aim to apply Schauder fixed point theorem. Fix $\eps,M>0$. We define the map $S_{f,M,\eps}: W^{1,2}_0(\X)\to \dom_0(\Delta)\subset  W^{1,2}_0(\X)$ given by $S_{f,M,\eps}(w)\coloneqq u$, where $u\in \dom_0(\Delta)$ is the (unique) solution to
		\begin{equation}\label{eq:pre Schauder fixed}
			\begin{cases}
				&	\LL_{\nabla w ,\eps}u=h(\eps, w )-(\theta_{\nabla w })^{-1}\int \theta_{\nabla w }(h(\eps, w )-	\LL_{\nabla w ,\eps}u)\d \mm, \\
				&\int u \, \d \mm=0,
			\end{cases}
		\end{equation}
		where $\theta_{\nabla w}$ is as in \eqref{eq:theta} and $h(\eps,w)\in L^2(\mm)$ is defined as
		\begin{equation}\label{eq:definition of h}
			h(\eps,w)\coloneqq
			\begin{cases}
				\frac{f}{\left (|\nabla w|^2+\eps \right)^{\frac{p-2}{2}}},& \text{if $p\ge 2$}\\
				\frac{f}{\left ((|\nabla w|\wedge M)^2+\eps \right)^{\frac{p-2}{2}}}, & \text{if $p<2$}.
			\end{cases}	
		\end{equation}
		The map $S_{f,\eps,M}$ is well defined, since \eqref{eq:pre Schauder fixed} admits a unique solution $u$ by Proposition \ref{prop:step 1}. We also stress that $S_f$ is non-linear. 
		
		For brevity, we almost always omit the dependence on $\eps$ and $M$  and simply write $S_f$. The parameter $\eps$ will remain fixed throughout the proof, while $M$ will appear only in the case $p<2$ and we will take $M\to +\infty$ towards the end.
		
		To apply the Schauder fixed point theorem, it is sufficient to show that $S_f$ is continuous with relatively compact image.
		
		\smallskip
		
		\noindent \textsc{Compactness.} The image $S_f(W^{1,2}_0(\X))$ is relatively compact in $W_0^{1,2}(\X).$  Indeed, from \eqref{eq:estimates solution} we have
		\begin{equation}\label{eq:a priori bounds}
			\|\Delta S_f(w)\|_{L^2(\mm)}\le C  \|h(\eps,w)\|_{L^2(\mm)}\le  \begin{cases}
				C \eps^{\frac{2-p}{2}}\|f\|_{L^2(\mm)},& \text{if $p>2$}\\
				C (M^2+\eps)^{\frac{2-p}{2}}\|f\|_{L^2(\mm)}, & \text{ if $p<2$}, 
			\end{cases}
		\end{equation}
		for every $w \in W_0^{1,2}(\X)$, where $C$ is a constant depending only on $p$, $N$ and  $\lambda_1(\X)K^-$. Moreover, the inclusion $\dom_0(\Delta)\hookrightarrow \W_0(\X)$ is compact by Proposition \ref{lem:norms}. Observe that the above estimates depend  on $\eps$ and also on $M$ (if $p<2$), but this is sufficient at this stage of the argument. 
		
		\noindent \textsc{Continuity.} The map $S_f$ is continuous in $W^{1,2}(\X)$. It is enough to prove that  if $w_n \to w$ in $\W(\X)$ then $\|\Delta(S_f(w_n)-S_f(w))\|_{L^2(\mm)}\to 0.$ Indeed the inclusion $(\dom_0(\Delta),\|\Delta(\cdot)\|_{L^2(\mm)})$ in $W^{1,2}(\X)$ is compact by Proposition \ref{lem:norms} and thus continuous. It is sufficient to show that this holds for every subsequence, up to a further subsequence. Fix then a non-relabelled subsequence.
		Up to extracting a further non-relabelled subsequence, we can then assume that $|\nabla w_n-\nabla w|\to 0$ and $|\nabla w_n|\to |\nabla w|$ in $\alme.$ For ease of notation, we write $\theta_n\coloneqq \theta_{\nabla w_n}$, $\theta=\theta_{\nabla w}$, $h_n\coloneqq h(\eps,w_n),h\coloneqq h(\eps,w)$,  $\LL_n\coloneqq\LL_{\nabla w_n,\eps}$,  $\LL\coloneqq\LL_{\nabla w ,\eps}$, $u_n\coloneqq S_f(w_n)$ and $u\coloneqq S_f(w).$  Recall that $ 0<c\le\theta_n \le C_N$, for some constants $c,C_N$ depending only on $N$. From the current assumptions and definitions we have
		\begin{equation}\label{eq:conv theta}
			\theta_n \to \theta, \quad \alme.
		\end{equation}
		Recall also that for every $p \in \mathcal{RI}_\X$, there exists a constant $c<1$ depending only on $p,N$ and the value of $\lambda_1(\X)K^-,$ such that
		\begin{equation*}
			\|\Delta(u_n-u)- \theta_n \LL_n(u_n-u)\|_{L^2(\mm)}\le c \|\Delta(u_n-u)\|_{L^2(\mm)}, \quad \forall n \in \nn.
		\end{equation*}
		(see the proof of Proposition \ref{prop:step 1}). Therefore using the equation \eqref{eq:pre Schauder fixed} and the inequality above we can compute
		\begin{align*}
			\|\Delta& (u_n-u)\|_{L^2(\mm)}\\
			&= \|\Delta(u_n-u)+ \theta_n ( h_n-\LL_n(u_n))+\theta (\LL(u)- h)-
			\int_\X \theta_n ( h_n-\LL_n(u_n))+\theta (\LL(u)- h) \|_{L^2(\mm)}\\
			&\le \|\Delta(u_n-u)- \theta_n \LL_n(u_n)+\theta_n h_n+\theta \LL(u)-\theta h\|_{L^2(\mm)}\\
			&\le \|\Delta(u_n-u)-\theta_n\LL_n(u_n)+\theta \LL(u)\|_{L^2(\mm)}+\|\theta_n h_n-\theta h\|_{L^2(\mm)}\\
			&\le \begin{multlined}[t]\|\Delta(u_n-u)- \theta_n \LL_n(u_n-u)\|_{L^2(\mm)}+\|\LL(u)(\theta-\theta_n)\|_{L^2(\mm)}\\
				+\|\theta_n (\LL(u)-\LL_n(u))\|    +\|\theta_n(h_n-h)\|_{L^2(\mm)}+\|h(\theta_n-\theta)\|_{L^2(\mm)}\end{multlined}\\
			& \le\begin{multlined}[t] c \|\Delta(u_n-u)\|_{L^2(\mm)}+\|\LL(u)(\theta-\theta_n)\|_{L^2(\mm)}\\
				+C_N\|\LL(u)-\LL_n(u)\|    +C_N\|h_n-h\|_{L^2(\mm)}+\|h(\theta_n-\theta)\|_{L^2(\mm)},\end{multlined}
		\end{align*}
		where we used that $|\theta_n|\le C_N.$ Since $c<1$ (independent of $n$), we can absorb the first term of the right-hand side into the left-hand side. Hence, it is sufficient to show that all the other terms on the right-hand side go to zero as $n \to +\infty$.
		By dominated convergence and \eqref{eq:conv theta} (recall that by definition $\LL(u)\in L^2(\mm)$), it follows that
		\[
		\|\LL(u)(\theta-\theta_n)\|_{L^2(\mm)}\to 0,\quad \|h(\theta_n-\theta)\|_{L^2(\mm)}\to 0, \quad \text{ as } n\to +\infty.
		\]
		From the definition of the operators $\LL_n,\LL$, we have
		\begin{align*}
			\left |{\LL_n(u)-\LL(u)}\right|&=|p-2|\left  |  \frac{\H {u}(\nabla w_n,\nabla w_n)}{|\nabla w_n|^2+\eps}-  \frac{\H {u}(\nabla w,\nabla w)}{|\nabla w|^2+\eps}\right|\\
			&=|p-2|\left  |  \H{u} \left(\frac{\nabla w_n}{\sqrt{|\nabla w_n|^2+\eps}}-\frac{\nabla w}{\sqrt{|\nabla w|^2+\eps}},\frac{\nabla w_n}{\sqrt{|\nabla w_n|^2+\eps}}+\frac{\nabla w}{\sqrt{|\nabla w|^2+\eps}}\right)\right|\\
			&\le 2|p-2||\H {u}|\left |\frac{\nabla w_n}{\sqrt{|\nabla w_n|^2+\eps}}-\frac{\nabla w}{\sqrt{|\nabla w|^2+\eps}}\right|\\
			&\le2|p-2||\H {u}|\left (\frac{|\nabla w_n-\nabla w|}{\sqrt{|\nabla w_n|^2+\eps}}+\frac{|\nabla w|\, \sqrt{||\nabla w|^2-|\nabla w_n|^2|}}{\sqrt{|\nabla w|^2+\eps}\sqrt{|\nabla w_n|^2+\eps}}\right),
		\end{align*}
		where we used that $|\sqrt{t}-\sqrt s|\le \sqrt{|t-s|}$, for every $t\ge 0,s\ge 0$.
		Therefore, $\left |{\LL_{n}(u)-\LL(u)}\right|\to 0$ $\mea$-a.e.. Moreover, it holds that $\left |{\LL_n(u)-\LL(u)}\right|\le 2|p-2||\H u|\in L^2(\mm)$. Hence from the dominated convergence theorem we deduce that $\|\LL_{n}(u)-\LL(u)\|_{L^2(\mm)}\to 0$. It remains to estimate $\|h_n-h\|_{L^2(\mm)}$. For $p>2$ we have
		\[
		|h_n-h|= |f| \left|\frac{1}{\left (|\nabla w_n|^2+\eps \right)^{\frac{p-2}{2}}}-\frac{1}{\left (|\nabla w|^2+\eps \right)^{\frac{p-2}{2}}}\right|\le \frac{2|f|}{\eps^{\frac{p-2}{2}}},\quad \mea\text{-a.e..}
		\]
		While clearly $h_n\to h$ $\mea$-a.e., from which we conclude  that $\|h_n-h\|_{L^2(\mm)}\to 0$ by dominated convergence theorem. Analogously, for $p<2$
		\[
		|h_n-h|\le |f| \left|\frac{1}{\left ((|\nabla w_n|\wedge M)^2+\eps \right)^{\frac{p-2}{2}}}-\frac{1}{\left ((|\nabla w|\wedge M)^2+\eps \right)^{\frac{p-2}{2}}}\right|\le 2|f|(M+\eps)^{\frac{2-p}{2}}, \quad \mea\text{-a.e..}
		\]
		This yields $\|h_n-h\|_{L^2(\mm)}\to 0$, again by dominated convergence.
		
		\smallskip
		
		\noindent \textsc{Fixed point.} We deduce that $S_f$ has a fixed point $u\in \dom_0(\Delta)$ and recalling \eqref{eq:L=D} we have 
		\begin{equation}\label{eq:almost pde}
			D_{\eps,p}(u)=h(\eps,u)-(\theta_{\nabla u})^{-1}\lambda_u,
		\end{equation}
		where $\lambda_u \in \rr$ is defined as 
		\begin{equation}\label{eq:lambada u}\lambda_u\coloneqq\int \theta_{\nabla u}(h(\eps,u)-	D_{\eps,p}(u))\d \mm,
		\end{equation}
		$h$ is defined in \eqref{eq:definition of h} and $\theta_{\nabla u}\in L^\infty(\mm)$ is as in \eqref{eq:theta}.
		
		\smallskip
		
		\noindent \textsc{Existence of solution to  \eqref{eq:ep laplacian}.} We now have to distinguish the two cases $p<2$ and $p> 2$. Indeed, for $p>2$ we show directly that a fixed point of $S_f$ solves \eqref{eq:ep laplacian}, while the case $p<2$ requires an additional limiting procedure sending $M \to +\infty$.

		\smallskip

		\noindent\textbf{{Case $p>2$}.} If $N>2$, since $|\nabla u|\in \W(\X)$,  by the Sobolev inequality (recall Proposition \ref{prop:sobolev rcd}) we have that $\,|\nabla u|\in L^{2^*}(\mm)$, where $2^*\coloneqq\frac{2N}{N-2}$ ($2^*\coloneqq 2$ if $N=\infty$). Moreover, since $p\in \mathcal{RI}_\X\subset (1,2+\frac N{N-2})$ (recall Definition \ref{def:regularity interval}) we can check that $p-1\le 2^*$ and $2(p-2)\le 2^*$. On the other hand if $N=2$, since $\Xdm$ is also an $\rcd(K,2+\delta)$ space for all $\delta\ge 0$, we have $|\nabla u|\in L^q(\mm)$ for all $q\ge 1$.  In both cases, it holds $|\nabla u|^{p-1}\in L^1(\mm)$ and $|\nabla u|^{2(p-2)}\in L^1(\mm)$. Therefore we can apply Lemma \ref{lem:develop} and deduce from \eqref{eq:almost pde} that
		\[
		\int -(|\nabla u|^2+\eps )^{\frac{p-2}{2}}  \la \nabla \phi,\nabla u\ra=  \int (f-(|\nabla u|^2+\eps )^{\frac{p-2}{2}}(\theta_{\nabla u})^{-1}\lambda_u)\phi\d \mm, \quad \forall \phi \in \LIP(\X).
		\]
		Taking $\phi\equiv 1$ and recalling that $f$ has zero mean we get
		\[
		\lambda_u\int (|\nabla v|^2+\eps )^{\frac{p-2}{2}}(\theta_{\nabla u})^{-1}\d \mm=0.
		\]
		Since $\theta_{\nabla u}^{-1}\ge c>0$ $\mea$-a.e.\ for some constant $c>0$ we obtain that $\lambda_u=0$. Therefore,
		\begin{equation}\label{eq:eps p equation}
			\int -(|\nabla u|^2+\eps )^{\frac{p-2}{2}}  \la \nabla \phi,\nabla u\ra=  \int f\phi\d \mm, \quad \forall \phi \in \LIP(\X).
		\end{equation}
		This would be the desired conclusion once we prove that $u \in W^{1,p}(\X)$. Observe that $u \in W^{1,p}(\X)$ would be a consequence of \eqref{eq:covariant regularized}. Indeed that would imply  $|\nabla u|^{2(p-1)}\in L^{1}(\mm)$ and since $2(p-1)\ge p$ for all $p\ge 2$, one concludes that $u \in W^{1,p}(\X)$. We then turn to prove \eqref{eq:covariant regularized}.	Since $\lambda_u=0$, by \eqref{eq:almost pde} we have that $D_{\eps,p} (u)=f(|\nabla u|^2+\eps )^{\frac{2-p}{2}}$. Therefore, setting  $v_M \coloneqq ((|\nabla u|\wedge M)^2+\eps)^\frac{p-2}{2}\nabla u \in H^{1,2}_C(\X)$,  applying \eqref{eq:covariant estimate} in the case $N<\infty$ 	 we have that for every $R_0>0$, every  $x \in \X$ and $0<R<R_0$,
		\[
		\int_{B_{R/2}(x)} R^{-2}|v_M|^2+|\nabla v_M |^2 \d \mm\le C_1 \int_{B_{R}(x)} f^2 \d \mm+ C_2\mea(B_{R}(x))^{-1}  \left(\int_{B_R(x)} |\nabla u|^{p-1}\d \mm\right)^2,
		\]
		where $C_1,C_2$ are constants depending only on $p,R_0,N,K$. In the case $N=\infty$ we apply instead  \eqref{eq:covariant estimate N=inf} and obtain the same inequality with $R=2\diam(\X)$ and $C_2=K^-\tilde C_1$ with $C_1$ depending on $K,p$ and $D$.
		Recalling that $f \in L^2(\mm)$ and that as observed above $|\nabla u|^{p-1}\in L^1(\mm)$, the right-hand side is uniformly bounded in $M$. Therefore by Lemma \ref{lem:lsc local energy} we obtain that $(|\nabla u|^2+\eps )^{\frac{p-2}{2}}\nabla u \in H^{1,2}_C(T\X)$ together with \eqref{eq:covariant regularized}. Finally, \ref{item:apriori lip regularized} follows directly applying Proposition \ref{prop:lip a priori}.
		\medskip

		\noindent \textbf{Case $p<2$.} Let $u_M\in \dom_0(\Delta )$ be a fixed point of $S_{f,\eps,M}$. By \eqref{eq:almost pde} and Lemma \ref{lem:develop} we have that 
		\begin{equation}\label{eq:eq p small}
			\int -(|\nabla u_M|^2+\eps )^{\frac{p-2}{2}}  \la \nabla \phi,\nabla u_M\ra\d\mm= \int (g-\lambda_{u_M}(|\nabla u_M|^2+\eps)^{\frac{p-2}{2}})\phi\d \mm,
		\end{equation}
		where $g\coloneqq\frac{f(|\nabla u|^2+\eps )^{\frac{p-2}{2}}}{\left ((|\nabla u|\wedge M)^2+\eps \right)^{\frac{p-2}{2}}}  \in L^\infty(\mm)$ and $\lambda_u$ is defined in \eqref{eq:lambada u}.
		The goal now is to send $M\to +\infty$. To do so we need first some uniform estimates.
		
		\noindent\textit{\underline{Estimates independent of $M$}.}
		From the first inequality in  \eqref{eq:a priori bounds} we have that, because $p\in(1,2),$
		\begin{equation}\label{eq:swag}
			\begin{split}
				\int (\Delta u_M)^2\d\mm&\le \tilde C \|h(\eps,\nabla u_M)\|_{L^2(\mm)}^2\le  \tilde C \int f^2 (|\nabla u_M|^2+\eps)^{2-p}\d \mm \\&\le \tilde C \int \eps^{2-p}f^2+ f^2|\nabla u_M|^{2(2-p)}\d \mm \\
				&\le  \tilde C \int \eps^{2-p}f^2 +  \delta^{\frac{1}{(2-p)}}\||\nabla u_M|\|_{L^{2}(\mm)}^{2}+\delta^{-\frac{2}{p-1}}\|f\|_{L^\frac{2}{p-1}(\mm)}^{\frac{2}{p-1}},	
			\end{split}
		\end{equation}
		where $\tilde C$ is a constant, possibly changing from line to line, depending only on $p$ and $N$.
		By  \eqref{eq:Laplacian poincare}  we have
		\begin{equation}\label{eq:unif gradient 2*bound}
			\||\nabla u_M|\|_{L^{2}(\mm)}^{2}\le \lambda_1(\X)^{-1}\int (\Delta u_M)^2\d\mm.
		\end{equation}
		Plugging \eqref{eq:unif gradient 2*bound} into \eqref{eq:swag} and taking $\delta$ small enough we reach
		\begin{equation}\label{eq:stima p small proof}
			\int (\Delta u_M)^2\d \mm \le D \left( \int \eps^{2-p}f^2 \d \mm +\|f\|_{L^{\frac2{p-1}}(\mm)}^{\frac{2}{p-1}}\right),
		\end{equation}
		where $D$ is a constant depending only on $\X$, $N$ and $p$.
		
		\noindent\textit{\underline{Letting $M\to +\infty$.}}  By the uniform estimate \eqref{eq:stima p small proof} we have that 
		\begin{equation}\label{eq:laplacian unif l2 bound}
			\sup_{M>0}\|\Delta u_M\|_{L^2(\mm)}<+\infty
		\end{equation}
		Therefore, Proposition \ref{lem:norms} ensure that it exists a sequence $M_n\to +\infty$, $n \in \nn$ such that $u_n\coloneqq u_{M_n}\to u$ in $W^{1,2}(\X),$ for some $u \in \dom_0(\Delta)$. Recall that by \eqref{eq:eq p small} the function $u_n$ satisfies
		\begin{equation}\label{eq:eq p small n}
			\int -(|\nabla u_n|^2+\eps )^{\frac{p-2}{2}}  \la \nabla \phi,\nabla u_n\ra\d \mm= \int \frac{f(|\nabla u_n|^2+\eps )^{\frac{p-2}{2}}}{\left ((|\nabla u_n|\wedge M_n)^2+\eps \right)^{\frac{p-2}{2}}}\phi-\lambda_{u_n}(|\nabla u_n|^2+\eps)^{\frac{p-2}{2}}\phi\d \mm,
		\end{equation}
		for every $\varphi \in \LIP(\X)$, where
		$$\lambda_{u_n}=\int  \frac{f}{\left ((|\nabla u_n|\wedge M_n)^2+\eps \right)^{\frac{p-2}{2}}} -	D_{\eps,p}(u_n)\d \mm$$
		(recall that  $\theta_{\nabla u_n}\equiv 1$ because $p<2$).
		As the left-hand side of \eqref{eq:eq p small n} is concerned, by strong convergence in $W^{1,2}(\X)$, we have 
		$$\la\nabla u_n,\nabla \phi\ra\overset{L^2(\mm)}{\rightarrow} \la \nabla u,\nabla \phi\ra, \quad \forall \phi \in\LIP(\X)$$
		and clearly $||\nabla u_n|^2+\eps|^{\frac{p-2}{2}} \le \eps^{\frac{p-2}{2}}$. Therefore, for all $\phi \in\LIP(\X)$
		\begin{equation}\label{eq:p<2 first limit}
			\int (|\nabla u_n|^2+\eps )^{\frac{p-2}{2}}  \la \nabla \phi,\nabla u_n\ra\d \mm \rightarrow 	\int (|\nabla u|^2+\eps )^{\frac{p-2}{2}}  \la \nabla \phi,\nabla u\ra\d \mm, \quad \text{ as $n\to +\infty$.}
		\end{equation}
		We now turn to deal with the right-hand side. Observe that the functions $D_{\eps,p}(u_n)$ are uniformly bounded in $L^2(\mm)$. Indeed, by the definition of $D_{\eps,p}(u_n)$ (see \eqref{eq:deps}),
		\[
		|D_{\eps,p}(u_n)|\le |\Delta u_n|+|\H {u_n}|
		\]
		holds $\mea\text{-a.e.}$. Since $\|u_n\|_{\dom_0(\Delta)}$ is uniformly bounded, we have that $|\H {u_n}|$ and $\Delta u_n$ are uniformly bounded in $L^2(\mm)$ (recall Proposition \ref{lem:norms}),   hence so is $|D_{\eps,p}(u_n)|$. 
		Additionally, since $p\in(1,2)$, we have
		\begin{align*}
			\int \left | \frac{f}{\left ((|\nabla u_n|\wedge M_n)^2+\eps \right)^{\frac{p-2}{2}}} \right| \d \mm &\le \int |f| (|\nabla u_n|+\sqrt{\eps})^{2-p} \d \mm\le \|f\|_{L^2(\mm)} \|\nabla u_n|+\sqrt{\eps}\|_{L^{2(2-p)}(\mm)}\\
			&\le \|f\|_{L^2(\mm)} \|\nabla u_n|+\sqrt{\eps}\|_{L^{2}(\mm)},
		\end{align*}
		where we used twice the H\"older inequality and also that $\mm(\X)=1$ (as assumed at the beginning). Note that $\|\cdot \|_{L^{2(2-p)}(\mm)}$ might  not be a norm as it possible that $2(2-p)<1$, however the H\"older inequality still applies. Next we observe that $\sup_n \|\nabla u_n|\|_{L^{2}(\mm)}<+\infty$. This follows from the uniform Laplacian bound in \eqref{eq:laplacian unif l2 bound} together with \eqref{eq:unif gradient 2*bound}.
		Combining this with the $L^2$-uniform bound on $|D_{\eps,p}(u_n)|$, we conclude  that $\sup_n|\lambda_{u_n}|<+\infty.$ Hence, up to a subsequence we can assume that $\lambda_{u_n}\to \lambda\in \rr.$ Observing that, because $p<2$, 
		\begin{equation}\label{eq:ez1}
			\left|\frac{(|\nabla u_n|^2+\eps )^{\frac{p-2}{2}}}{\left ((|\nabla u_n|\wedge M_n)^2+\eps \right)^{\frac{p-2}{2}}} \right|\le 1 \quad \mea\text{-a.e.,}
		\end{equation}
		by the dominated convergence theorem
		\begin{equation}\label{eq:p<2 second limit}
			\lim_{n\to +\infty}\int \frac{f(|\nabla u_n|^2+\eps )^{\frac{p-2}{2}}}{\left ((|\nabla u_n|\wedge M_n)^2+\eps \right)^{\frac{p-2}{2}}}\phi-\lambda_{u_n}(|\nabla u_n|^2+\eps)^{\frac{p-2}{2}}\phi\d \mm
			= \int f\phi-\lambda(|\nabla u|^2+\eps)^{\frac{p-2}{2}}\phi\d \mm.
		\end{equation}
		Combining \eqref{eq:eq p small n} with  \eqref{eq:p<2 first limit} and \eqref{eq:p<2 second limit} we obtain
		\[
		\int -(|\nabla u|^2+\eps )^{\frac{p-2}{2}}  \la \nabla \phi,\nabla u\ra\d \mm=\int f\phi-\lambda(|\nabla u|^2+\eps)^{\frac{p-2}{2}}\phi\d \mm, \quad \forall \phi \in\LIP(\X).
		\]
		Taking $\phi\equiv 1$, since $f$ has zero mean, we must have $\lambda=0.$ This shows that $u \in \dom(\Delta_{p,\eps} u)$ (recall that $u \in W^{1,p}(\X)$ because $p<2$) and the second equality  in \eqref{eq:ep laplacian}. The first identity in \eqref{eq:ep laplacian} then follows from Lemma \ref{lem:develop} and the uniqueness of $\Delta_{p,\eps}u.$
		
		\noindent\textit{Regularity estimates for $p<2$.}
		The fact that  $v\coloneqq(|\nabla u|^2+\eps )^{\frac{p-2}{2}}\nabla u \in H^{1,2}_C(T\X)$ together with \eqref{eq:covariant regularized} is shown exactly as in the case $p\ge 2$ using $D_{\eps,p}u= (|\nabla u|^2+\eps )^{\frac{2-p}{2}}f.$ 
		The gradient estimate in $ii)$ follows again from Proposition \ref{prop:lip a priori}.
	\end{proof}

	\section{Proof of the main results}
	\label{sec: main results}
	We are now ready to deduce our main regularity estimates for the $p$-Laplacian, which will follow by the results of the previous section and by sending $\eps \to 0^+.$
	In \textsection \ref{sec:proof main} we will prove all our main results, except for the one for $p$-harmonic functions (Corollary \ref{cor:pharm}) which will be left to \textsection \ref{sec:pharm}.
	
	\subsection{Proof of Theorem  \ref{thm:main rcd inf}, Theorem \ref{thm:main intro} and of Corollary \ref{cor:p eigen}} \label{sec:proof main}
	The following result essentially contains all our main statements.
	\begin{theorem}\label{thm:main detailed}
		Let $\Xdm$ be an $\rcd(K,N)$ space with  $N\in[2,\infty]$, $\diam(\X)\le D$. Let $p \in \mathcal{RI}_{\X}$ where $ \mathcal{RI}_{\X}\subset(1,\infty)$ is given by \eqref{eq:regularity interval}. Fix also $B_R(x)\subset \X$ with $R\le R_0$ (with $R\coloneqq 8\diam(\X)$ if $N=\infty)$, $u \in \dom(\Delta_p)$ and set $f\coloneqq \Delta_p u \in L^1(\mm)$. Then the following hold:
		\begin{enumerate}[label=\roman*)]
			\item If $f \in L^2(B_R(x))$ then $v\coloneqq |\nabla u|^{p-2}\nabla u\in H^{1,2}_{C,\loc}(T\X;B_{R/2}(x))$ and 
			\begin{equation}\label{eq:covariant final}
				\int_{B_{R/8}(x)} |v|^2+|\nabla v |^2 \d \mm\le C_1 \int_{B_{R}(x)} f^2 \d \mm+ C_2\mea(B_{R}(x))^{-1}  \left(\int_{B_R(x)} |\nabla u|^{p-1}\d \mm\right)^2, 
			\end{equation}
			where $C_1,C_2$ are constants depending only on  $K,N,p$,$R_0$ if $N<\infty$,  while in the case $N=\infty$ all integrals are on the whole $\X$ and  $C_2=K^-C_1$ with $C_1$ a constant depending on $D,p$ and $K.$
			\item If $N<\infty$ and $f \in L^{q}(B_{R}(x);\mea)$ for some $q>N$, with $\fint_{B_{R}(x)} |f|^{q}\d \mm\le C_0$, then for all $m\ge 1$ it holds
			\begin{equation}\label{eq:apriori lip final}
				| |\nabla u|\|_{L^\infty(B_{r}(x))}\le \frac{CR^{\frac Nm}}{(R-r)^{\frac Nm}} \left(\frac{}{}\bigg(\fint_{B_{R}(x)}  |\nabla u|^{m}\, \d \mm  \bigg)^\frac{1}{m}+1\right), \quad \forall r<R,
			\end{equation}
			where $ C\ge 1$ is constant depending only in $p,N,K,q,C_0,m$ and $R_0,$
			\item\label{item:3theoremmain} if $f=0$ in $B_R(x)$ then $u \in \dom(\Delta,B_{R/2}(x))$ and
			\begin{equation}\label{eq:pharmonic w22}
				\int_{B_{R/2}(x)} (\Delta u)^2\d \mm\le \tilde C \int_{B_{R/4}(x))} (1+R^{-2}) |\nabla u|^2 \d \mm,
			\end{equation}
			where $\tilde C>0$ is a constant depending only on $K,N$ and $p$.
		\end{enumerate}
	\end{theorem}
	
	For the proof of item \ref{item:3theoremmain}, we will need one last technical estimate.
	\begin{lemma}\label{lem:pre w22}
		Let $\Xdm$ be an $\RCD(K,N)$ space, $N\in [2,\infty]$ and let $p \in (1,3+\frac{2}{N-2})$ and $\eps>0$. Suppose that $u \in \dom(\Delta)$ satisfies
		\begin{equation}\label{eq:zero developed plapl}
			D_{\eps,p}u=0, \quad \alme \text{ in $B_R(x)\subset \X$.}
		\end{equation}
		Then there exists a constant $C>0$ depending only on $K,N$ and $p$ such that 
		\begin{equation}\label{eq:pre w22 estimate}
			\int_{B_{R/2}(x)} (\Delta u)^2\d \mm\le C \int_{B_R(x)} (1+R^{-2}) |\nabla u|^2 \d \mm.
		\end{equation}
	\end{lemma}
	\begin{proof}
		Take $\eta \in \LIP_{bs}(B_{R}(x))$ such that $\eta=1$ in $B_{R/2}(x)$ and $\Lip(\eta)\le 3/R$.  Applying the improved Bochner inequality \eqref{eq:improvedboch} with $\phi=\eta^2$, integrating by parts and applying the Young's inequality, we obtain
		\begin{align*}
			\int &\left(|\H f|^2+\frac{(\Delta f-{\rm tr}\H f)^2}{N-\dim(\X)}\right)\eta^2\d \mm\\
			&\le \int (\Delta u)^2 \phi + |\nabla \eta^2||\nabla |\nabla u|||\nabla u|+|\Delta u||\nabla u||\nabla \eta^2|+K^-|\nabla u|^2 \eta^2\d \mm\\
			&\le \int (1+\delta)(\Delta u)^2 \eta^2 + \delta |\H u|^2 \eta^2 +  (K^-\eta^2+8\delta \Lip(\eta)^2 )|\nabla u|^2\d \mm.
		\end{align*}
		Therefore 
		\begin{equation}\label{eq:1-delta w22}
			\begin{multlined}[c][.8\textwidth]
				(1-\delta)\int \left(|\H f|^2+\frac{(\Delta f-{\rm tr}\H f)^2}{N-\dim(\X)}\right)^2\eta^2\d \mm\\\le 
				\int (1+\delta)(\Delta u)^2 \eta^2  +  (K^-\eta^2+8\delta \Lip(\eta)^2 )|\nabla u|^2\d \mm.
			\end{multlined}
		\end{equation}
		We now argue similarly to the proof of Proposition \ref{prop:step 1} and define
		$\theta\in L^\infty(\mm)$ as
		\begin{equation}\label{eq:theta2}
			\theta\coloneqq	\begin{cases}
				1,& \text{ if  $p\in(1,2)$},\\
				\frac{N+g}{N+g^2+2g},& \text{otherwise },
			\end{cases}
		\end{equation}
		where  $g\coloneqq(p-2)\frac{|\nabla u|^2}{|\nabla u|^2+\eps}.$ Then, applying  Lemma \ref{lem:near laplacian} when $p\ge 2$, we have
		\begin{align*}
			\| \eta \Delta u  \|_{L^2(\mm)}^2\overset{\eqref{eq:zero developed plapl}}{=} \| \eta (\Delta u - \theta D_{\eps,p} u) \|_{L^2(\mm)}^2 \le \alpha_p \int  \left(|\H f|^2+\frac{(\Delta f-{\rm tr}\H f)^2}{N-\dim(\X)}\right)\eta^2\d \mm,
		\end{align*}
		where $\alpha_p=(p-2)^2$ if $p<2$ and $\alpha_p=\frac{(p-2)^2(N-1)}{N+2(p-2)+(p-2)^2}<1$ for $p\ge 2$. Combining this with \eqref{eq:1-delta w22} and choosing $\delta$ so that $\frac{1+\delta}{1-\delta}\alpha_p<1$ gives \eqref{eq:pre w22 estimate}.
	\end{proof}

	We are now ready to prove the above theorem.
	\begin{proof}[Proof of Theorem \ref{thm:main detailed}]
		We first assume that $f \in L^{\infty}(\mm)$.  Fix a sequence $\eps_n>0$ such that  $\eps_n\to 0$.  Note that $f$ must have zero mean.  Hence, from Theorem \ref{thm:regularity epsilon}  for every $n \in\nn$ there exists a function $u_n \in  \dom_0(\Delta)\cap \dom (\Delta_{p,\eps_n})$ such that $\Delta_{p,\eps_n}(u_n)=f$ and \eqref{eq:covariant regularized}, \eqref{eq:apriori lip regularized} hold. Moreover, by Proposition \ref{prop:existence of weak solutions} it holds $u_n\to u$ in $W^{1,p}(\X)$. In particular, setting $v_n\coloneqq (|\nabla u_n|^2+\eps_n)^\frac{p-2}2\nabla u_n$, we have $|v_n-v|\to 0$ $\mea$-a.e.. From \eqref{eq:covariant regularized}, we obtain
		\[
		\sup_n \int_{B_{r}(y)} |v_n|^2+|\nabla v_n |^2 \d \mm<+\infty, \quad \forall B_{2r}(y)\subset B_R(x).
		\]
		Lemma \ref{lem:lsc local energy} combined with \eqref{eq:covariant regularized} ensures that $v\in H^{1,2}_{C,\loc}(T\X;B_{R}(x))$ together with \eqref{eq:covariant final} (with $B_{R/4}(x)$ instead of $B_{R/8}(x)$). Similarly, if $f \in L^{q}(B_{R}(x);\mea)$ for some $q>N$, since $|\nabla u_n|\to |\nabla u|$ $\mea$-a.e.\ we also get \eqref{eq:apriori lip final} passing to the limit in \eqref{eq:apriori lip regularized}. 
		If $f=0$ in $B_R(x)$ by Lemma \ref{lem:pre w22} we have that  $\Delta u_n$ is bounded in $L^2(B_{R/2}(x))$. Hence, up to a subsequence $\Delta{u_n}\restr{B_{R/2}(x)}\rightharpoonup G \in L^2(B_{R/2}(x))$. By \eqref{eq:apriori lip final} and \eqref{eq:apriori lip regularized} we have both that $u_n,u \in W^{1,2}_\loc(B_{R/2}(x))$  and that  $|\nabla u_n-\nabla u|\to 0$ in $L^2(B_{R/2}(x))$. Recalling \eqref{eq:closure laplacian}, we conclude that $u \in \dom(\Delta, B_{R/2}(x))$ with $\Delta u=G$. Moreover \eqref{eq:pharmonic w22} follows from \eqref{eq:pre w22 estimate} and lower-semicontinuity.

		For a general $f \in L^1(\mm)$, we consider $f_n\coloneqq (-n)\vee f\wedge n-c_n \in L^\infty(\mm)$, where $c_n\in \rr$ is chosen so that $\int f_n \d \mm=0$. For every $n \in \nn$ by Proposition \ref{prop:p-existence} there exists functions $u_n \in \dom(\Delta_p)$ with zero mean and $\Delta_p (u_n)=f_n$. Since $f_n \to f$ in $L^1(\mm)$, applying Proposition \ref{prop:L1 theory} and up to passing to a subsequence we have $|\nabla u_n-\nabla u|\to 0$, $\mea$-a.e..  Since \eqref{eq:covariant final} holds for $u_n$ with $B_{R/4}(x)$ in place of $B_{R/8}(x)$ (as observed above), Lemma \ref{lem:lsc local energy} gives $|\nabla u|^{p-2}\nabla u\in H^{1,2}_{C,\loc}(T\X;B_{R/2}(x))$ and \eqref{eq:covariant final} for $u$. Similarly, \eqref{eq:apriori lip final} follows passing to the limit in \eqref{eq:apriori lip final} for $u_n$. Finally \eqref{eq:pharmonic w22} follows observing that \eqref{eq:apriori lip final} (for $u_n$ and $u$) implies that  $u_n,u \in W^{1,2}_\loc(B_{R/4}(x))$ and  $|\nabla u_n-\nabla u|\to 0$ in $L^2(B_{R/4}(x))$ and arguing as in the previous part of the proof, using \eqref{eq:pharmonic w22} for $u_n.$
	\end{proof}

	\begin{remark}[About weaker notions of solutions]\label{rmk:sola}
		Inspecting the proof of Theorem \ref{thm:main detailed}, we only used that there exists a sequence of functions $u_n \in \dom(\Delta_p)$ and $f_n \in L^\infty(\mm)$  satisfying $\Delta_p u_n=f_n$, $f_n \to f$ in $L^1(\mm)$ and $|\nabla u_n-\nabla u|\to 0$ $\mea$-a.e.. Also (except for item $iii)$) we never used that $\nabla u$ is actually the gradient of $u$. The conclusion can be read as a statement only about $\nabla u$  as an element of $L^0(T\X)$. This shows that we could weaken the assumption of Theorem \ref{thm:main detailed} assuming only the existence of such approximating sequence, without even assuming that $u$ belongs to a Sobolev space $W^{1,r}(\X)$, but rather that it has a gradient $\nabla u\in L^0(T\X)$ in some generalized sense. This is what was done in Euclidean setting in \cite{CiaMaz18}, where second-order estimates (as in our main theorem) are deduced for a suitable notion of generalized solutions defined via approximation in the same spirit as above (see \cite[Section 4]{CiaMaz18}). The reason for doing it is that solutions to $\Delta_p u=f \in L^2$ sometimes exist only in this generalized sense (recall Remark \ref{rmk:existence}). We also observe that is not possible to weaken the assumption of Theorem \ref{thm:main detailed} to ask only that $u \in W^{1,r}(\X)$ with $r<p$ (note that for the definition of $p$-Laplacian in \eqref{eq:def plapl} to make sense $r\ge p-1$ suffices). Indeed recently in \cite{ColTi22} for all $p\in(1,\infty)$ it has been constructed  a $p$-harmonic function  $u \in W^{1,r}_\loc$ with $r>p-1$ for which  $|\nabla u|\notin L^p_\loc$ and in particular we cannot have that $|\nabla u|^{p-1}\in \W_\loc $ if $p\ge 2^*/(2^*-1).$\fr 
	\end{remark}
	
	We can prove our main statement for $\RCD(K,\infty)$ spaces.
	\begin{proof}[Proof of Theorem \ref{thm:main rcd inf}]
		Applying Theorem \ref{thm:main detailed}  we have directly that $|\nabla u|^{p-2}\nabla u \in H^{1,2}_{C}(T\X)$ and by Proposition \ref{prop:gradgrad} also that $|\nabla u|^{p-1}\in W^{1,2}(\X)$.  Finally, inequality \eqref{eq:p-calderon intro} is just \eqref{eq:covariant final}.
	\end{proof}
	
	Next, we show the main result in finite dimension. 
	\begin{proof}[Proof of Theorem \ref{thm:main intro}]
		Fix $\Xdm$, $p$, $\Omega$ and $u$ as in the statement. Assume that $\Delta_p(u)\in L^2(\Omega).$ From \eqref{eq:covariant final} of Theorem \ref{thm:main detailed} we have that $|\nabla u|^{p-2}\nabla u \in H^{1,2}_{C,\loc}(T\X;B_{R}(x))$ for every $B_R(x)\subset \Omega$. From this it follows that $|\nabla u|^{p-2}\nabla u \in H^{1,2}_{C,\loc}(T\X,\Omega)$ (recall Remark \ref{rmk:local to global}) and by \eqref{eq:grad grad local} also that $|\nabla u|^{p-1}\in W^{1,2}_\loc(\Omega)$, which proves \ref{item:main intro item1}. Assume now that $\Delta_p(u)\in L^q(\Omega)$ for some $q >N$. Then again by \eqref{eq:apriori lip final} of Theorem \ref{thm:main detailed} we get that $|\nabla u|\in L^{\infty}(B_{R/2}(x))$ for every $B_{R}(x)\subset \Omega.$ By the Sobolev-to-Lipschitz property (see \cite{Gigli13,AGS14}) property this implies that $u \in \LIP_\loc(B_{R/2}(x))$. By the arbitrariness of $B_{R}(x)$ we obtain that $u \in \LIP_\loc(\Omega)$, that proves \ref{item:main intro item2} and concludes the proof.
	\end{proof}

	Before continuing we comment on our main statement.
	\begin{remark}\label{rmk:counterexampleW22}
		Under the assumptions of Theorem  \ref{thm:main rcd inf} we do not necessarily have that $u \in W^{2,2}(\X)$, not even in the smooth setting. Indeed, as explained in \cite[Remark 2.7]{CiaMaz18}, the function $u: \R^n \to \R$ defined as $u(x_1,\dots,x_n)\coloneqq |x_1|^\frac{3-\delta}{2}$ with $\delta\in(0,1)$ satisfies  $\Delta_p(u)\in L^2_\loc(\rr^n)$ for any $p>1+\frac{1}{1-\delta}$, while  $u \notin W^{2,2}_\loc$ in any neighbourhood of the origin. Then, by a cut-off argument, a function with the same property can be easily built on the $n$-dimensional flat torus $\mathbb T^n$. However note that  $p\in \mathcal {RI}_{\mathbb T^n}$ if $\delta$ is small, which shows that we can not improve the conclusion of Theorem \ref{thm:main rcd inf} to $u \in W^{2,2}_\loc(\X)$, not even when $p$ is close to 2.\fr
	\end{remark}

	\begin{remark}[Previous results]\label{rmk:references}
		In the Euclidean setting the second-order regularity result in Theorem \ref{thm:main rcd inf} was obtained in \cite{CiaMaz18} for all $1<p<+\infty$ (see \eqref{eq:p reg intro}).
		We mention also an earlier  result in \cite{Lou08} showing  that $|\nabla u|^{p-1} \in \W_\loc(\Omega)$ whenever $\Delta_p(u)=f\in L^{q}_\loc$ with $q> \max(2,n/p)$.   Second-order regularity estimates of the type in \eqref{eq:p reg intro} are classical for (non-degenerate) quasilinear elliptic equations (see \cite[Chap. 4]{UralLadybook}). 
		For $p$-harmonic functions in $\rr^n$ it is also well known that  $|\nabla u|^{\frac{p-2}{2}}\nabla u\in \W_\loc$ for all $p\in(1,\infty)$ (see e.g.\ \cite{Uh77,Ur68,BojIwa84}). Recently in  \cite{Sa22,DPZZ20}  this was improved to $|\nabla u|^{\frac{p-s}{2}}\nabla u\in \W_\loc$ for any $s< 3+\frac{p-1}{n-1} $ (see also \cite{DD16} and \cite{LZZ21} for the same result in weighted smooth Riemannian manifolds). In particular when $1<p<3+\frac{2}{n-2}$, taking $s=p$, this implies that $u \in W^{2,2}_\loc$ (see also  \cite{ManWeit88} for an earlier prof of this fact). Note that this last property is  coherent with $iii)$ in Theorem \ref{thm:main detailed}. It is instead not known whether $u \in W^{2,2}_\loc$ in the case $p\ge 3+\frac{2}{n-2}$ for $n\ge 3.$  In $\RCD(K,N)$ spaces was also shown in \cite{gigli_monotonicityformulasharmonicfunctions_2023} that $|\nabla u|^{ 2-s/2} \in \W_\loc$ for all $s< 3+\frac{1}{N-1}$, for any $u$  harmonic (note that this corresponds to the results that we just mentioned, in the case $p=2$). In weighted smooth Riemannian manifolds the Lipschitzianity of $p$-harmonic functions was shown in \cite{DD16} (see also \cite{WZ11}).  It is also worth mentioning that in the Heisenberg group of any dimension (which is a PI space  \cite[Appendix A.6]{BB13}), $p$-harmonic functions are known to be Lipschitz for any $p\in(1,\infty)$ and actually $C^{1,\alpha}$ \cite{MukZhong21,zhong17} (see also \cite{DomMan05,MinZhong09,DomMan05bis} for similar results and references). \fr
	\end{remark}

	We can now deduce the regularity of $p$-eigenfunctions.
	\begin{proof}[Proof of Corollary \ref{cor:p eigen}]
		As in the statement, we take $\Xdm$  a bounded $\rcd(K,N)$, with  $N\in[2,\infty)$ and suppose $u\in W^{1,p}(\X)$ satisfies $\Delta_p u=-\lambda u |u|^{p-2}$ for some constant $\lambda \ge 0$ and  $p \in \mathcal{RI}_{\X}$. Then, by Theorem \ref{thm:holder} we have that $u \in L^\infty(\mm)$ (in fact $u$ has a H\"older continuous representative). Hence, $\Delta_p u \in L^\infty(\mm)$ and the conclusion follows applying Theorem \ref{thm:main intro}.
	\end{proof}

	\subsection{Regularity of \texorpdfstring{$p$}{p}-harmonic functions} \label{sec:pharm}
	We pass to the proof of Corollary \ref{cor:pharm} about the regularity of $p$-harmonic functions. The main idea is, given a $p$-harmonic function $u$ defined only on some open subset $\Omega$, to find a suitable cut-off $\phi$ such that $\phi(u)$ has global $p$-Laplacian in $L^1(\mm)$  and coincides with $u$ in a subset of $\Omega$. Then the required regularity will follow automatically from Theorem \ref{thm:main intro}. We will need the following chain rule for the $p$-Laplacian.
	\begin{prop}[Chain-rule for $\Delta_p$]\label{prop:chain rule}
		Fix $p \in (1,\infty)$. Let $\Xdm$ be an $\rcd(K,\infty)$ space and $\Omega\subset \X$ be open. Then for every $u \in \dom(\Delta_p,\Omega)$ and $\phi \in C^2\cap \LIP(\rr)$ such that $\psi(t)\coloneqq \phi'(t)|\phi'(t)|^{p-2}\in \LIP\cap L^\infty(\rr)$ it holds $\phi(u)\in \dom(\Delta_p,\Omega)$ and 
		\begin{equation}\label{eq:chain rule}
			\Delta_p(\phi(u))= \psi(u)\Delta_p(u)+\psi'(u)|\nabla u|^p.
		\end{equation}
	\end{prop}
	\begin{proof}
		By the chain rule for the gradient we have $\phi(u),\psi(u)\in W^{1,p}_\loc(\Omega)$ and $\nabla \phi(u)=\phi'(u)\nabla u,$ $\nabla \psi(u)=\psi'(u)\nabla u.$ Since $\psi(u)\in L^\infty(\Omega)$, by the Leibniz rule, for every $\eta \in \LIP_{bs}(\Omega)$ it holds that $\eta \psi(u)\in W^{1,p}\cap L^\infty(\X)$ with bounded support in $\Omega$. In particular, we can use $\eta \psi(u)$ as a test function in the definition of $\Delta_p(u)$ (recall Remark \ref{rmk:test w1p plapl}).
		Hence
		\begin{align*}
			\int |\nabla \phi(u)|^{p-2}\la \nabla \phi(u),\nabla \eta\ra \d \mm&=  \int |\nabla u|^{p-2}\la \nabla u,(\phi'(u)|\phi'(u)|^{p-2})\nabla \eta\ra \d \mm\\
			&=\int |\nabla u|^{p-2}\la \nabla u,\psi(u)\nabla \eta\ra \d \mm=\int -\Delta_p(u) \psi(u)\eta - |\nabla u|^p \psi'(u)\eta,
		\end{align*}
		which is what we wanted.
	\end{proof}
	
		In the following we state the existence of cut-off functions $\phi$ such that $\phi'(t)|\phi'(t)|^{p-2}$ is Lipschitz. The proof is elementary and left to the reader (note that for $p\ge 2$ any  $\phi \in C_c^2([0,1])$  would work).
	\begin{lemma}\label{lem:p-cutoff}
		For every $\eps\in(0,1)$ and every $p\in(1,\infty)$ there exists  $\phi \in C_c^2([0,1])$ such that $\phi(t)=t$ in $[\eps,1-\eps]$ and $\psi\coloneqq \phi' |\phi '|^{p-2} \in \LIP([0,1])$.
	\end{lemma}

	We can now prove the regularity of $p$-harmonic functions with relatively compact level sets. 
	\begin{proof}[Proof of Corollary \ref{cor:pharm}]
		Let $\Xdm$, $p $ and $u$ as in the hypotheses and suppose that $u^{-1}((a,b))$ is relatively compact in $\Omega.$ Up to translating and rescaling $u$ we can assume that $a=0$ and $b=1.$   Fix $\eps>0$. By Lemma \ref{lem:p-cutoff} there exists a function $\phi \in C_c^2([0,1])$ such that $\phi(t)=t$ in $[\eps,1-\eps]$ and $\psi\coloneqq \phi' |\phi '|^{p-2} \in \LIP(\rr)$. Then by the chain rule of Proposition \ref{prop:chain rule} we have $\phi(u)\in \dom(\Delta_p,\Omega)$ with
		\[
		\Delta_p (\phi(u))= \psi'(u)|\nabla u|^p \in L^1(\Omega).
		\]
		Moreover, since $\psi\equiv 0$ in $\rr \smallsetminus (0,1)$, we have that $\Delta_p (\phi(u))\equiv 0$ in $\Omega\smallsetminus \{0<u<1\}$. Observe that $\{0\le u\le 1\}$ is relatively compact in $\Omega$. Using the definition of $\Delta_p$, we can thus check that $u\in \dom(\Delta_p,\X)$ with $\Delta_p (\phi(u))= \psi'(u)|\nabla u|^p\in L^1(\mm)$. Since $\Delta_p (\phi(u))=0$ in $\{\eps< u<1-\eps\}$ and $\phi(u)=u$ in $\{\eps< u<1-\eps\}$ (recall that this set is open because $u$ is continuous), Theorem \ref{thm:main detailed} yields $u \in \LIP_{\loc}\cap H^{2,2}_\loc(\{\eps< u<1-\eps\})$. By the arbitrariness of $\eps$ this concludes the proof.
	\end{proof}
	
	\medskip
	The statement of Corollary \ref{cor:pharm} pertains to a class of functions referred to as $p$-capacitary potentials, which represent solutions denoted as $u$ to the following boundary value problem:
	\begin{equation}\label{eq:potential_equation}
		\begin{cases}
			\Delta_p u =0 & \text{in $ \Omega \setminus K$,} \\
			u =1  & \text{on $\partial K$,} \\
			u=0 & \text{on $\partial \Omega$.}
		\end{cases}
	\end{equation}
	Here, $\Omega$ signifies an open, bounded subset within a compact Riemannian manifold $(M, g)$, while $K$ is precompact within $\Omega$. The pair $(K, \Omega)$ is commonly referred to in the literature as a condenser. 
	Notably, when $p = 2$, the solution $u$ takes on the interpretation of the electrostatic potential of the condenser. In this case, the $p$-Dirichlet energy quantifies the total electrostatic energy contained within the domain of the condenser, often referred to as the capacity.
	
	The function $u$ complies with the assumptions detailed in Corollary \ref{cor:pharm}. It is a $p$-harmonic function, and the maximum principle ensures that the sets $u^{-1}([a, b])$ are relatively compact within $\Omega \setminus K$ for all $0 < a \leq b < 1$. As a result, the second-order Sobolev regularity and local Lipschitz continuity properties can be applied to this function.
	
	A promising area for future research involves extending the above result to its application in unbounded $\RCD(0,N)$ spaces $\Xdm$. With this extension, one can show the regularity of solutions to \eqref{eq:potential_equation} taking $\Omega=\mathrm{X}$, using an exhaustion argument.
	
	Emphasizing the importance of regularity of nonlinear potentials of bounded domains is essential, as it deeply shapes various outcomes in Riemannian geometry. One achievement involves proving a Minkowski-type inequality which establishes a scale-invariant lower limit for the $L^1$-norm of mean curvature with respect to the perimeter. This inequality has been proven in the context of Riemannian manifolds with non-negative Ricci curvature by the first author of this paper in collaboration with Fogagnolo and Mazzieri in \cite{benatti_minkowskiinequalitycompleteriemannian_2022}. This result is the counterpart for $p\neq 2$ of the previous result by Agostiniani, Fogagnolo, and Mazzieri in \cite{AFM18} which is been extended in the non-smooth setting, by the second author with Gigli in \cite{gigli_monotonicityformulasharmonicfunctions_2023} (see also \cite{VPhDThesis}). 
	
	Building upon the above regularity and the defined regularity range, it may become feasible to establish the existence of a weak inverse mean curvature flow (IMCF for short). This conceptual tool, initially introduced by Huisken and Ilmanen in their work \cite{huisken_inversemeancurvatureflow_2001}, held a pivotal role in their demonstration of the Riemannian Penrose Inequality. The existence of this flow can be rigorously demonstrated by employing an approximation involving the $p$-capacitary potential. Indeed, taking $u_p$ the $p$-capacitary potential of some bounded $K$, the sequence $w_p= - (p-1) \log u_p$ converges in a suitable sense to the weak IMCF as the parameter \(p\to 1^+\). This has been shown firstly by Moser \cite{moser_inversemeancurvatureflow_2007,moser_inversemeancurvatureflow_2008} in Euclidean setting. It was subsequently extended to curved spaces by Kotschwar and Ni \cite{kotschwar_localgradientestimatesharmonic_2009}, and more recently by Mari Rigoli and Setti \cite{mari_flowlaplaceapproximationnew_2022}. 
	
	We emphasize that the specified regularity interval $\mathcal{RI}_{\mathrm{X}}$ in \eqref{eq:regularity interval} is not excessively restrictive in the context of $\RCD(0,N)$ spaces, since for $K=0$ we have $\mathcal{RI}_{\X}=(1,3+2/(N-2))$.

	\appendix
	\section{Appendix: H\"older regularity for elliptic equations in PI spaces}\label{sec:moser}
	In this section, we give a proof of Harnack's inequalities and H\"older continuity for a broad class of elliptic equations in PI spaces involving the $p$-Laplacian.  The validity of this results seems to be folklore among the community and already appeared often in different settings and in different forms, even if mostly for $p=2$ (see for example \cite[Chapter 8]{BB13}, \cite{KiSha01,BjM06,LMP06,Hua09},\cite[Section 4]{Jiang}, \cite[Theorem 6.7]{G19}, \cite[Lemma 3.5]{ZZweyl} and also the recent \cite[Section 6]{honda2023sharp}). Since we explicitly need this for $p\neq 2$ both in Corollary \ref{cor:p eigen} and Corollary \ref{cor:pharm}, but also for future use, we decided to include a self-contained proof.
	
	The main idea for the proof is that, thanks to the doubling property, the local Poincaré inequality improves to a local Sobolev inequality  \cite{HK00}  (see also \cite{Saloff92}), which then allows to use of the iteration methods of De Giorgi-Nash-Moser.

	Throughout this section we assume that $\Xdm$ is a PI-space (see Definition \ref{def:PI}) satisfying
	\begin{equation}\label{eq:pi dimension app}
		\frac{\mea(B_r(y))}{\mea(B_{R}(y))}\ge c(R_0)\left(\frac rR\right)^s, \quad \forall\, 0<r<R\le R_0, \, \forall y \in \X,
	\end{equation}
	for some constant $s>1$ and some function $c:(0,\infty)\to (0,\infty)$.  
	
	\begin{remark}\label{rmk:app}
		If $\Xdm$ is a PI-space then \eqref{eq:pi dimension app} always holds for some $s>0$ (see \cite[Lemma 3.3]{BB13}). Moreover, by the locally doubling assumption, if \eqref{eq:pi dimension app} holds for some $s$ then it holds also for any $s'>s$, up to changing the function $c$. In particular \eqref{eq:pi dimension app} always holds for some $s>1.$  \fr
	\end{remark}
	
	\begin{theorem}[Harnack inequality]\label{thm:holder}
		Let $\Xdm$ be an infinitesimally Hilbertian PI-space satisfying \eqref{eq:pi dimension app} with some $s>1$ and let $c_0>0$ be any constant. Fix also $p\in(1,s]$ and $q>s/p.$  Suppose $u\in W^{1,p}_\loc(B)$, where $B\coloneqq B_R(x)\subset \X$,  satisfies
		\begin{equation}\label{eq:sol}
			\int |\nabla u|^{p-2}\la \nabla u, \nabla \phi\ra\d \mm= \int g u|u|^{p-2} \phi + f \phi \d \mm, \quad \forall \phi \in \LIP_{c}(B),
		\end{equation}
		for some $f,g\in L^{q}(B)$ and such that $R^p\left(\fint_B |g|^q\d \mm\right)^\frac{1}{q}\le c_0$. Then $u$ has a locally H\"older continuous representative in $B$. Moreover, if $u$ is also non-negative,  for every $B_{70\lambda r}(y)\subset B$ with $r\le R$ it holds
		\begin{equation}\label{eq:harnack}
			\esssup_{B_r(y)} u \le C \essinf_{B_r(y)} u + C  r^{\frac{p-s/q}{p-1}} R^{\frac s{q(p-1)}}\left(\fint_B |f|^q\d \mm\right)^\frac{1}{q(p-1)},
		\end{equation}    
		where $C$ is a constant depending only on $p,q,\lambda,C_D,C_P,R_0,s,c$ and $c_0$ ($C_P,C_D,\lambda$ being the constants appearing in  Definition \ref{def:PI} of PI-space and $c$ the function  in \eqref{eq:pi dimension app}).
	\end{theorem}
	
	\begin{remark}[The case $s>p$]
		Recall that  if $\Xdm$ is a PI-space satisfying \eqref{eq:pi dimension app} for some $s>0$, then every function $u \in W^{1,p}_\loc(\Omega)$, with $p>s$ and $\Omega \subset \X$ open, is automatically locally H\"older continuous (see \cite[Theorem 5.1]{HK00} or \cite[Corollary 5.49]{BB13}). \fr
	\end{remark}
	
	We start by proving a Harnack-type inequality for supersolutions.
	\begin{theorem}\label{thm:supersol}
			Let $\Xdm$ be an infinitesimally Hilbertian PI-space satisfying \eqref{eq:pi dimension app} with some $s>1$ and $c_0>0$ be a constant. Fix also $p\in(1,s]$ and $q>s/p.$  Suppose $u\in W^{1,p}_\loc(B_R(x))$, for some $B=B_R(x) \subset \X$  with $R\le R_0,$   satisfies
		\begin{equation}\label{eq:supersol}
			\int |\nabla u|^{p-2}\la \nabla u, \nabla \phi\ra\d \mm\le \int g |u|^{p-1} \phi + f \phi \d \mm, \quad \forall \phi \in \LIP_{c}(B_R(x)), \,\, \phi\ge 0,
		\end{equation}
			for some $f,g\in L^{q}(B)$, $g\ge 0,f\ge 0$ and such that
			\begin{equation*}
				R^{p}\left(\fint_{B}|g|^q \d\mm  \right)^\frac1q\le c_0.
			\end{equation*} 
		Then for all $B_r(y)\subset B$ with $r\le R$, every $\theta\in (0,1)$ and every  $l\in[1,\infty)$ it holds
		\begin{equation}\label{eq:sup estimate}
			\esssup_{B_{\theta r}(y)} u^+ \le \frac C {(1-\theta)^{\frac s l}}\left( \fint_{B_r(y)}|u^+|^l \d \mm \right)^\frac1l+C\left( \fint_{B_r(y)}|r^pf|^q \d \mm \right)^\frac1{(p-1)q},
		\end{equation}
		where  $C$ is a constant depending only on $p,q,\lambda,C_D,C_P,R_0,s,c,l$ and $c_0$ ($C_P,C_D,\lambda$ being the constants appearing in  Definition \ref{def:PI} of PI-space and $c$ the constant in \eqref{eq:pi dimension app}).
	\end{theorem}
	\begin{proof}
		It is enough to show \eqref{eq:sup estimate} for $r=R$ and $y=x$. Indeed suppose for a moment that this holds. Then for every $B_t(z)\subset B_R(x)$ with $t\le R$ we have
		\begin{equation}\label{eq:from R to r}
			\begin{split}
				t^{p}\left(\fint_{B_t(z)}|g|^q \d\mm  \right)^\frac1q&= \frac{t^p}{\mea(B_t(z))^\frac 1q}\left(\int_{B_R(x)}|g|^q \d\mm  \right)^\frac1q\\\overset{\eqref{eq:pi dimension app}}&{\le }c(2R_0)(2R)^\frac sq \frac{t^{p-\frac s q}}{\mea(B_{2R}(z))^\frac 1q}\left(\int_{B_R(x)}|g|^q \d\mm  \right)^\frac1q\\
				&\le  c(2R_0)(2R)^p \left(\fint_{B_R(x)}|g|^q \d\mm  \right)^\frac1q\le  2^pc_0c(2R_0)
			\end{split}
		\end{equation}
		having used that $p-\frac s q>0,$ where $c(\cdot)$ is the constant given by \eqref{eq:pi dimension app}. Therefore the hypotheses of the theorem are satisfied taking $B\coloneqq B_t(z)$  and with $c_0$ replaced by $2^pc_0c(2R_0)$. Then we can just apply \eqref{eq:sup estimate} taking $B_r(y)=B$ (which we are assuming to hold). 
		
		Next, we observe that is sufficient to prove \eqref{eq:sup estimate} in the case $\theta=\frac12$. Indeed the general case follows considering the balls $B_{(1-\theta)r}(z)\subset B$ for all $z \in B_{\theta r}(y)$ and using that $\mea(B_{(1-\theta)r}(z))\ge 2^{-s}(1-\theta)^{s}\mea(B_r(y))$, thanks to \eqref{eq:pi dimension app} (for the second term in the right-hand side of \eqref{eq:sup estimate} we also use $pq>s$). 
		
		Moreover, we need to prove \eqref{eq:sup estimate} only for $l\ge p$. Indeed if this was true, for the case $l< p$ we would conclude 
			with a standard iteration argument as in the proof of Proposition \ref{prop:lip a priori} (see also \cite[p.\ 75]{HanLinbook}).
	
		All in all we only need to prove \eqref{eq:sup estimate} for $y=x$, $r=R$, $\theta=\frac12$ and  $l\ge p.$ Along the proof $C>0$ will denote a constant depending only on $p,q,s,c_0,R_0,C_P,C_D,c,\lambda,l$ and possibly changing from line to line. 
		Thanks to the scaling of the statement we can assume that $\mea(B_R(x))=1.$ 
		By the Sobolev embedding (recall Theorem \ref{thm:improved poincaret}) $W^{1,p}_\loc(B_R(x_0))\subset L^{p^*}_\loc(B_R(x_0))$, with $p^*=\frac{sp}{s-p}$ if $s<p$ and for any $p^*>1$ if $s=p$, hence by the integrability assumptions on $g$ and $f$ (and the density of Lipschitz functions) we have that \eqref{eq:supersol} holds also for all $\phi \in W^{1,p}_\loc(B_R(x_0)).$ Fix $\eps>0$ arbitrary and set $k\coloneqq (\|R^pf\|_{L^q(B_R(x))}+\eps)^{\frac1{p-1}}<+\infty$. Set $\bar u\coloneqq (u^++k)$ and for every $m\ge k$ define  $\bar u_m\coloneqq \bar u \wedge m.$ By the chain rule it holds $\bar u, \bar u_m\in W^{1,p}_\loc(B_R(x_0))$ and $\bar u_m \in L^{\infty}(B_R(x_0))$. Moreover $\bar u,\bar u_m\ge k.$ For every $\beta\ge 1$ and $\eta \in \LIP_c(B_R(x))$, with $\eta \ge 0$, we choose as test function
		\[
		\phi \coloneqq \eta^p(\bar u \bar u_m^{p(\beta-1)}-k^{p(\beta-1)+1})\ge 0, \quad \alme.
		\]
		Note that that $\phi=0$ in $\{u\le 0\}$ and $\phi \le \eta^p\bar u \bar u_m^{p(\beta-1)}$. In particular $\nabla u=\nabla \bar u$ $\alme$ in $\{\phi >0\}$. 
			By the Leibniz rule we have $\phi \in W^{1,p}(\X)$ with
		\begin{equation}\label{eq:grad phi}
			\nabla \phi= p\eta^{p-1}\nabla \eta (\bar u \bar u_m^{p(\beta-1)}-k^{p(\beta-1)+1})+\eta^p\left(\bar u_m^{p(\beta-1)}\nabla \bar u + p(\beta-1)\bar u^{p(\beta-1)} \nabla \bar u_m\right)
		\end{equation}
		Substituting $\phi$ in  \eqref{eq:supersol} and using \eqref{eq:grad phi}
		\begin{align*}
			\int |\nabla \bar u|^p& \bar u_m^{p(\beta-1)}\eta^p+p(\beta-1)|\nabla \bar u_m|^p\bar u^{p(\beta-1)}\eta^p\d\mm\\
			&\le \int p|\nabla \bar u|^{p-1}|\nabla \eta| \eta^{p-1}\bar u \bar u_m^{p(\beta-1)}+g \bar u^p\bar u_m^{p(\beta-1)}\eta^p + f \bar u\bar u_m^{p(\beta-1)}\eta^p\d \mm\\
			&\le \int (\delta p)^\frac{p}{p-1}|\nabla \bar u|^p \bar u_m^{p(\beta-1)} \eta^p +\delta^{-p}
			|\nabla \eta|^p  \bar u^p\bar u_m^{p(\beta-1)} + \left(g+\frac{f}{k^{p-1}}\right) \bar u^p\bar u_m^{p(\beta-1)}\eta^p\,\d \mm,
		\end{align*}
		for every $\delta>0$, having used that 
		\[
		|\nabla u|^{p-2}\la \nabla \bar u, \nabla u\ra=|\nabla \bar u|^p, \quad \alme, \quad
		|\nabla u|^{p-2}\la \nabla \bar u_m, \nabla u\ra=|\nabla \bar u_m|^p,\quad \alme,
		\]
		by the locality.  In the last step, we used also that $\bar u \ge k.$ Choosing $\delta=1/(2p)$ and absorbing the first term of the right-hand side into the left one we reach
		\begin{equation}\label{eq:first moser bound}
			\begin{multlined}[c][.7\textwidth]
				\int |\nabla \bar u|^p \bar u_m^{p(\beta-1)}\eta^p+p(\beta-1)|\nabla \bar u_m|^p\bar u^{p(\beta-1)}\eta^p\d\mm\\\le C\int
				|\nabla \eta|^p  \bar u^p\bar u_m^{p(\beta-1)} +  R^{-p}h\bar u^p\bar u_m^{p(\beta-1)}\eta^p \d \mm,
			\end{multlined}
		\end{equation}
		where  $h\coloneqq R^{p}g+\frac{R^{p}f}{k^{p-1}}\in L^{q}(B_R(x))$. Note that  by how we chose $k$ it holds $ \|h\|_{L^{q}(B_R(x))}\le c_0+1$. 
		Set $w\coloneqq \bar u \bar u_m^{\beta-1}$. As above  we have $w \in W^{1,p}_\loc(\Omega)$ with
		\[
		|\nabla w|\le|\nabla \bar u|  \bar u_m^{(\beta-1)}+(p(\beta-1))|\nabla \bar u_m|\bar u^{(\beta-1)} , \quad \alme.
		\]
		Therefore by \eqref{eq:first moser bound} 
		\[
		\int |\nabla (w \eta) |^p\d \mm\le \int 2^p(|\nabla w|^p |\eta|^p+|w|^p |\nabla \eta|^p)\d \mm\le C \beta^p \int
		|\nabla \eta|^p w^p + R^{-p}h w^p\eta^p\d \mm,
		\]
		having used that $(\beta-1)^p\le (\beta-1)\beta^p$, since $\beta \ge 1.$
		We define $p^*\coloneqq \frac{ps}{s-p}$ if $p<s$, while if $p=s$ we set $p^*\coloneqq \frac{2sq}{q-1}$ (any number strictly greater than $\frac{sq}{q-1}$ would do). 
		Applying the Sobolev inequality in \eqref{eq:local sobolev} and recalling $\eta \in \LIP_c(B_R(x))$
		\begin{equation}\label{eq:sobolev in moser}
			\begin{split}
				\|w \eta \|_{L^{p^*}(\mm)}^p&\le C \beta^p \int
				R^p|\nabla \eta|^p w^p +(h+1) w^p\eta^p \d \mm \\&\le  C \beta^p \int
				R^p|\nabla \eta|^p w^p\d \mm +C\beta^p \|w \eta \|_{L^{\frac{pq}{q-1}}(\mm)}^p,    
			\end{split}
		\end{equation}
		where in the Sobolev inequality we used that $\mea(B_{2\lambda R}(x))\ge \mea(B_R(x))\ge 1$ and in the last step the H\"older inequality together with $\|h\|_{L^{q}(B_R(x))}\le C$. Since $q>s/p$, we have that $p<\frac{pq}{q-1}<p^*$ (this holds also if $s=p$, since in that case we defined $p^*=\frac{2pq}{q-1}$). Hence by interpolation of norms, we have that for every $\delta>0$ it holds
		$$
		\|w \eta \|_{L^{\frac{pq}{q-1}}(\mm)}^p\le \delta^{-\alpha_1} \|w \eta \|^{p}_{L^p(\mm)}+\delta^{\alpha_2}\|w \eta\|_{L^{p^*}(\mm)}^p,
		$$ for suitable exponents $\alpha_1,\alpha_2>0$ depending only on $q,p$ and $s$. Plugging this estimate in \eqref{eq:sobolev in moser} and choosing $\delta\coloneqq   (\delta'\beta^{-p})^\frac{1}{\alpha_2}$ for some $\delta'>0$ small enough we obtain
		\begin{equation}\label{eq:moser lol}
			\|w \eta \|_{L^{p^*}(\mm)}^p\le C \beta^{\alpha}  \int
			R^p|\nabla \eta|^p w^p +|w\eta |^p \d \mm,
		\end{equation}
		for some number $\alpha>0$   depending only on $q,p$ and $s$. Now we fix $1/2\le a<b\le 1$ arbitrary and  choose $\eta$ so that $\eta=1$ in $B_{aR}(x)$, $0\le \eta \le 1$, $\eta\in \LIP_c(B_{bR}(x))$ and $\Lip \eta \le \frac{2}{(b-a)R}$. Using this choice in \eqref{eq:moser lol}, recalling that  $w=\bar u \bar u_m^{(\beta-1)}$ and sending $m\to +\infty$ we reach
		\[
		\|\bar u\|_{L^{\chi \gamma}(B_{aR}(x))}^{\gamma} \le  \frac{C \beta^{\alpha}}{(b-a)^p}\|\bar u\|_{L^{\gamma}(B_{bR}(x))}^{\gamma}, 
		\]
		where $\gamma\coloneqq p\beta$, $\chi\coloneqq \frac{s}{s-p}>1$ if $p<s$ and   $\chi\coloneqq \frac{2q}{q-1}>1$ if $p=s$. Raising both sides to the $\frac1\gamma$ and choosing $\beta=\frac{l}{p}\chi^{i}\ge 1$ (recall that $\beta$ was arbitrary), $a=\frac12+\frac{1}{2^{i+2}}$ and $b=\frac12+\frac{1}{2^{i+1}}$ for $i=0,1,\dots,+\infty$,  we obtain 
		\[
		\|\bar u\|_{L^{\gamma_{i+1}}(B_{r_{i+1}}(x))} \le  (C (l/p)^{\alpha})^\frac{1}{\gamma_i} (\chi)^\frac{\alpha i}{\gamma_i} (4)^\frac{pi}{\gamma_i}\|\bar u\|_{L^{\gamma_i}(B_{r_{i}}(x))}\le (\tilde C)^{\frac{i}{\chi^i}}\|\bar u\|_{L^{\gamma_i}(B_{r_{i}}(x))}, \;\forall \, i=0,1,\dots,+\infty,
		\]
		where  $\gamma_i\coloneqq l\chi^i$,  $r_i\coloneqq \frac R2+\frac{R}{2^{i+1}},$ for $i=0,1,\dots,+\infty$, and where $\tilde C>0$ is a constant depending only on $p,q,s,c_0,R_0,C_P,C_D,c,\lambda,l$.  Iterating this estimate and recalling that $\bar u=u^++k$ and $\mea(B_R(x))=1$
		we obtain 
		\[
		\esssup_{B_{R/2}(x)} u^+ \le \|\bar u\|_{L^\infty({B_{\frac R2}(x)})}\le (\tilde C)^{\sum_{i=0}^{+\infty} \frac{i}{\chi^i}} \|\bar u\|_{L^l(B_R(x))}\le  C \left(\|u\|_{L^l(B_R(x))}+k\right),
		\]
		where as above $C$ is a constant  depending only on $p,q,s,c_0,R_0,C_P,C_D,c,\lambda,l$.
		Recalling the value of $k$ and by the arbitrariness of $\eps>0$ 
		gives \eqref{eq:sup estimate} in the case $x=y$, $r=R$, $\theta=1/2$ and $l\ge p$. This, as observed at the beginning, is sufficient to conclude.
	\end{proof}
	
	We pass to Harnack's inequality for non-negative subsolutions.
	\begin{theorem}\label{thm:subsol}Let $\Xdm$ be an infinitesimally Hilbertian PI-space satisfying \eqref{eq:pi dimension app} with some $s>1$ and let $c_0>0$ be a constant. Fix also $p\in(1,s]$ and $q>s/p.$  Suppose $u\in W^{1,p}_\loc(B)$, for some $B\coloneqq B_R(x) \subset \X$  with $R\le R_0,$ is non-negative and satisfies
		\begin{equation}\label{eq:subsol}
			\int |\nabla u|^{p-2}\la \nabla u, \nabla \phi\ra\ge \int -g u^{p-1} \phi  -f \phi \d \mm, \quad \forall \phi \in \LIP_{c}(B), \,\, \phi\ge 0,
		\end{equation}
		for some $f,g\in L^{q}(B)$, $g\ge 0,f\ge 0$ and such that
		\begin{equation*}
			R^{p}\left(\fint_{B}|g|^q \d\mm  \right)^\frac1q\le c_0.
		\end{equation*} 
	Then there exist  constants $C>0$ and $l\in[1,\infty)$ both depending only on $p,q,\lambda,C_D,C_P,R_0,s,c$ and $c_0$ ($C_P,C_D,\lambda$ being the constants appearing in  Definition \ref{def:PI} of PI-space and $c$ the constant in \eqref{eq:pi dimension app}) such that for every $B_{35\lambda r}(y)\subset B$ with $r\le R$ it holds
	\begin{equation}\label{eq:inf estimate}
		\essinf_{B_{r/2}(y)} u^+ + \left( \fint_{B_{25r}(y)}|r^pf|^q \d \mm \right)^\frac1{(p-1)q} \ge C \left( \fint_{B_r(y)}|u^+|^l \d \mm \right)^\frac1l.
	\end{equation}
\end{theorem}
\begin{proof}
	Fix $B_{35\lambda r}(y)\subset B$ with $r\le R$.
	Along the proof $C>0$ will denote a constant depending only on  $p,q,\lambda$,$C_D,C_P,R_0,s,c$, $c_0$ and possibly changing from line to line. As for \eqref{eq:supersol} we have that \eqref{eq:subsol} holds also for all $\phi \in W^{1,p}_\loc(B).$ By scaling we can assume that $\mea(B_r(y))=1.$ Let $\eps>0$ be arbitrary and set $k\coloneqq (\|r^pf\|_{L^q(B_{25r}(y))}+\eps)^{\frac1{p-1}}$. Setting $\bar u \coloneqq (u+k)$ and choosing $\phi\coloneqq \psi \bar u^{-2(p-1)}$, with $\psi \in \LIP_c(B_{r}(y))$, we obtain
	\[
	\int |\nabla (\bar u)^{-1}|^{p-2}\la \nabla  (\bar u)^{-1}, \nabla \psi\ra \le \int\left(g+\frac{f}{k^{p-1}}\right)   |\bar u^{-1}|^{p-1} \psi \d \mm,
	\]
	having neglected the term $\la \nabla u, \nabla \bar u^{-2(p-1)}\ra \psi\le 0$ $\mea$-a.e.. Setting $h\coloneqq g+\frac{f}{k^{p-1}} $ and arguing exactly as in \eqref{eq:from R to r}, since $\mea(B_r(y)=1$, we obtain $r^p\|h\|_{L^q(B_r(y))}\le \tilde c$, where $\tilde c$ is a constant depending only on $p,R_0,c_0$ and $c(\cdot)$ (i.e.\ the constant given by \eqref{eq:pi dimension app}). 
	Therefore the hypotheses of Theorem \ref{thm:supersol} are satisfied taking $B=B_r(y)$, $u=\bar u^{-1}$, $g=h$, $f=0$ and with the constant $2^pc_0+1$ in place of $c_0$. Hence we obtain that for every $l \in[1,\infty)$ and every $B_{35\lambda r}(y)\subset B_R(x)$ with $r\le R$ it holds
	\begin{equation}\label{eq:moser da sotto quasi}
		\frac{1}{\essinf_{B_{r/2}(y)} u+k}=\esssup_{B_{r/2}(y)} \bar u^{-1} \le C_l \left( \fint_{B_r(y)}|\bar u|^{-l} \d \mm \right)^\frac1l,
	\end{equation}
	where $C_l$ is a constant depending only on $p,q,\lambda,C_D,C_P,R_0,s,c$, $c_0$ and  also on $l.$
	Therefore to conclude it is enough to show that there exists $l\in[1,\infty)$ (depending only on  $p,q,\lambda,C_D,C_P,R_0,s,c$, $c_0$) such that
	\begin{equation}\label{eq:pp}
		\fint_{B_{r}(y)} \bar u^{l}\d \mm \cdot  \fint_{B_{r}(y)} \bar u^{-l}\d \mm \le 9.
	\end{equation}
	Indeed plugging this estimate in \eqref{eq:moser da sotto quasi}, combined with the assumption $\mea(B_r(y))=1$ and our choice of $k$ would yield \eqref{eq:inf estimate}.
	To show \eqref{eq:pp}, it is sufficient to show that
	\begin{equation}\label{eq:bmo}
		\fint_{B_t(z)} \left|w-\fint_{B_t(z)} w\d \mm \right|\d \mm \le C,
	\end{equation}
	for every $B_t(z)\subset B_{5r}(y),$ where $w\coloneqq \log(\bar u)\in W^{1,p}_\loc(\Omega)$, i.e.\ that $w\in {\rm BMO}(B_{5r}(y),\X)$. 
		Then \eqref{eq:pp} would follow applying the  John-Nirenberg lemma in the metric setting (see \cite[Corollary 3.21]{Bjorn-Bjorn07}). 
	
	To show \eqref{eq:bmo} fix $B_{t}(z)\subset B_{5r}(y)$ arbitrary. We can assume that $t\le 10r,$ and in particular $B_{2t}(z)\subset B_{25r}(y)$ and $B_{3\lambda t}(z)\subset B_{35\lambda r}(y)\subset B.$ Let $\psi \in \LIP_c(B_{2t}(z))$ satisfy $\psi=1$ in $B_t(z)$, $\psi\ge 0$ and $\Lip(\psi)\le 2/t.$
	Choosing as test function  $\phi\coloneqq \psi^p \bar u^{1-p}$ in \eqref{eq:subsol} we obtain
	\[
	\begin{split}
		(p-1)\int |\nabla w|^p \psi^p \d \mm& \le \int \psi^p\left(g+\frac{f}{k^{p-1}}\right)+|\nabla \psi^p||\nabla w|^{p-1}\d \mm\\
		&\le c_p \int \psi^ph+\delta^{p/(p-1)}|\nabla w|^p \psi^p+\delta^{-p}|\nabla \psi|^p  \d \mm,
	\end{split}
	\]
	for every $\delta>0,$ where $c_p>0$ is a constant depending only on $p$ and where $h\coloneqq g+\frac{f}{k^{p-1}}. $ In particular taking $\delta>0$ small enough with respect to $p,$ we can (and will) absorb the second term in the right-hand side to the left of the inequality.  Applying the H\"older inequality we estimate the first term on the right-hand side as follows
	\[
	\begin{split}
		\int \psi^ph\d \mm &\le \int_{B_{2t}(z)} h \d \mm \le \|h\|_{L^q(B_{2t}(z))} \mea(B_{2t}(z))^{1-\frac1q}
		\le \tilde c r^{-p} \mea(B_{2t}(z))^{1-\frac1q},
	\end{split}
	\]
	having used that $\|r^ph\|_{L^q(B_{25r}(y))}\le \tilde c$, where $\tilde c$ is a constant depending only on $p,R_0,c_0$ and $c(\cdot)$, which as above can be shown as in \eqref{eq:from R to r}, recalling that $\mea(B_r(y))=1.$
	
	Plugging  this estimate in the above inequality (using that $\Lip(\psi)\le 2/t$) we obtain
	\begin{align*}
		\int_{B_t(z)} |\nabla w|^p \d \mm &\le C\mea(B_{2t}(z))\left( r^{-p}\mea(B_{2t}(z))^{-\frac 1q}+t^{-p}\right)\\\overset{\eqref{eq:pi dimension app}}&{\le }  C\mea(B_{2t}(z))\left( r^{\frac sq-p} t^\frac{-s}{q}\mea(B_{6r}(z))^{-\frac 1q}+t^{-p}\right)\\
		&\le   C\mea(B_{2t}(z))\left( r^{\frac sq-p}  t^{-p}t^{p-\frac{s}{q}}+t^{-p}\right)\le C\mea(B_{2t}(z))t^{-p} \overset{\eqref{eq:pi dimension app}}{\le } C\mea(B_{t}(z))t^{-p},
	\end{align*}
	where in the third inequality we used that $\mea(B_{6r}(z)\ge \mea(B_r(y))\ge 1$ and in the fourth one that $p-\frac{s}{q}\ge 0$ and that $t\le 10r$.
	From this \eqref{eq:bmo} follows by the Poincar\'e inequality  (see \eqref{eq:improved poincaret}) applied to $w\eta$, where $\eta \in \LIP_c(B_{3\lambda t}(z))$ is a cut-off function with $\eta=1$ in $B_{2\lambda t}(z)$ indeed $w\eta\in W^{1,p}(\X)$, since as observed above  $B_{3\lambda t}(z)\subset B$. This concludes the proof.
\end{proof}
We can now prove the main result of this section.
\begin{proof}[Proof of Theorem \ref{thm:holder}]
	We prove only \eqref{eq:harnack}, from which the H\"older regularity of $u$ follows immediately by standard arguments exactly as in the Euclidean case (see e.g.\ \cite[Corollary 4.18]{HanLinbook}).
	Fix $B_{70\lambda r}(y)\subset B$. Then combining Theorem \ref{thm:supersol}, Theorem \ref{thm:subsol} and the doubling property we get
	\begin{align*}
		\esssup_{B_\frac r2(y)} u&\le C \essinf_{B_\frac r2(y)} u + C\left( \fint_{B_{25r}(y)}|r^pf|^q \d \mm \right)^\frac1{(p-1)q} 
		\le C \essinf_{B_\frac r2(y)} u + C \frac{r^\frac{p}{p-1}}{\mea(B_{25r}(y))^\frac1{q(p-1)}}\|f\|_{L^q(B)}^\frac{1}{p-1}\\
		\overset{\eqref{eq:pi dimension app}}&{\le} C \essinf_{B_\frac r2(y)} u + C \frac{r^{\frac{p-s/q}{p-1}} R^{\frac s{q(p-1)}}}{\mea(B)^\frac1q }\|f\|_{L^q(B)}^\frac{1}{p-1}
	\end{align*}
	which is \eqref{eq:harnack}.
\end{proof}

\begin{remark}
	Both Theorem \ref{thm:supersol} and Theorem \ref{thm:subsol} hold without the infinitesimally Hilbertian assumption, replacing $\la \nabla u, \nabla \phi \ra $ respectively by the objects $D^+\phi (\nabla u)$ and $D^-\phi (\nabla u)$ introduced in \cite[Section 3]{Gigli12}. The proof is exactly the same, only the first part needs to be translated by using the suitable calculus rules present in \cite{Gigli12}.\fr
\end{remark}

	\subsection*{Acknowledgements}
	Part of this work has been carried out during the authors' attendance to the \emph{Thematic Program on Nonsmooth Riemannian and Lorentzian Geometry} that took place at the Fields Institute in Toronto. The authors warmly thank the staff, the organizers and the colleagues for the wonderful atmosphere and the excellent working conditions set up there. 
L.B. is supported by the European Research Council’s (ERC) project n.853404 ERC VaReg -- \textit{Variational approach to the regularity of the free boundaries}, financed by the program Horizon 2020, by PRA\_2022\_11 and by PRA\_2022\_14.  I.Y.V. was partially supported by the Academy of Finland projects No.\ 328846 and No. 321896.
	
	The authors want to thank M. Fogagnolo, N. Gigli, J. Liu, F. Nobili and A. Pluda for their interest in the work and for pleasureful and useful conversations on the subject.

\medskip

\textbf{Competing interests:}
The authors have no competing interests to declare that are relevant to the content of this article.

\def\cprime{$'$}

\end{document}